\documentclass[11pt]{paper}
\usepackage[a4paper, margin=25mm]{geometry}
\usepackage[english]{babel}

\usepackage[utf8x]{inputenc}
\usepackage{amsmath, amsthm, amssymb}
\usepackage{mathtools}
\usepackage{enumitem}
\usepackage[nobysame,initials]{amsrefs}
\usepackage{nameref,hyperref,cleveref,cite}
\usepackage{dsfont,bbm}
\usepackage{graphicx}
\usepackage{xcolor}
\usepackage{textcomp}
\usepackage{comment}

\usepackage{wrapfig}
\usepackage{pdfpages}
\usepackage{comment}

\usepackage{array}
\newcolumntype{C}[1]{>{\centering\arraybackslash}p{#1}}
\renewcommand{\arraystretch}{1.5}
\usepackage{multirow}

\AtBeginDocument{%
  \def\MR#1{}%
}

\usepackage{pgfplots}
\pgfplotsset{compat=1.12}
\usepackage{tikz}
\usepackage{tikzscale}
\usetikzlibrary{shapes,arrows,calc,math,3d,patterns}
\usepackage{subcaption}
\usetikzlibrary{arrows.meta}

\title{Planarity and number of crossings in general models of random geometric graphs}

\author{Hanna Döring \thanks{Osnabrück University, Germany, hanna.doering@uni-osnabrueck.de} \and Kinga Nagy \thanks{Osnabrück University, Germany, kinga.nagy@uni-osnabrueck.de}}
\date{\today}

\theoremstyle{plain}
\newtheorem{theorem}{Theorem}[section]
\newtheorem{lemma}[theorem]{Lemma}

\newtheorem{claim}[theorem]{Claim}

\theoremstyle{definition}

\newtheorem{remark}[theorem]{Remark}

\Crefname{claim}{Claim}{Claims}
\crefname{appendix}{Appendix}{Appendices}
\Crefname{appendix}{Appendix}{Appendices}

\newcommand{\EE}{\mathbb{E}}
\newcommand{\PP}{\mathbb{P}}
\newcommand{\VV}{\mathbb{V}}
\DeclareMathOperator{\Po}{Poisson}
\newcommand{\ind}{\mathbf 1}

\DeclareMathOperator{\dTV}{d_{TV}}
\DeclareMathOperator{\Pareto}{Pareto}

\DeclareMathOperator{\Vol}{Vol}
\DeclareMathOperator{\diam}{diam}

\renewcommand{\l}{\left}
\renewcommand{\r}{\right}

\newcommand{\dd}{{\rm d}}

\newcommand{\con}{\leftrightarrow}

\newcommand{\caseif}{\text{if }}

\newcommand{\R}{\mathbb{R}}
\newcommand{\N}{\mathbb{N}}

\renewcommand{\c}{\mathcal}

\begin{document}

\maketitle

\begin{abstract}
    Consider a random geometric graph with vertices given by a Poisson point process, and whose edges depend on independent marks corresponding to the vertices and pairs of vertices.
    In this paper, we study two related questions on this general model: the number of edge crossings in a projection of this graph, and its graph-theoretical planarity. 
    We focus on models with heavy-tailed mark distributions, in particular with polynomial tails with arbitrary exponents. 
    We show that the asymptotic behaviour varies significantly depending on this exponent.
\end{abstract}



\tableofcontents

\section{Introduction}

When considering a geometrically defined graph in $\R^2$, one might look at the following two problems: how many edge crossings occur in this embedding of the graph, and is it graph-theoretically possible to draw it without any crossings? 

Take the hard threshold random geometric graph (Gilbert graph) $\c G_t$, whose vertices are given by a homogeneous Poisson point process in a unit volume convex body $W\subset \R^d$ with intensity $t$, and edges given by pairs of points having distance at most $r_t>0$. Consider the model as $t\to \infty$, and $r_t\to 0$ as $t\to \infty$.
In \cite{CDR18}, Chimani, Döring and Reitzner studied the issue of crossings in arbitrary dimensions: for the Gilbert graph $\c G_t$ with $d\geq 2$, consider its orthogonal projection to an arbitrary plane, and take the number of crossings $X_t$ in the projected construction. They obtained asymptotic expectation and variance results of $X_t$, and in particular showed that the expected number of crossings is of the order $t^4 r_t^{2d+2}$ \cite{CDR18} (for an extended version see \cite{CDRarxiv}). 
Döring and de~Jonge \cite{DdJ25} also derived normal- and Poisson approximation results for $X_t$, including showing that the point process defined by the edge crossings of the projection tends to a Poisson point process on the plane. 

As a contrast to this, Biniaz et al. \cite{Biniaz} considered the planarity of the binomial hard random geometric graph in $d=2$, that is, whether it can be drawn in the plane without crossings. In particular, they aimed to see how this compares to having zero crossings in the geometric representation given by the construction itself; they found the threshold function of the latter independently from \cite{CDR18}.
Specifically, it was shown \cite{Biniaz} that the threshold function of $r_t$ for planarity is $t^{-\frac 58}$, hence if $r_t=o(t^{-\frac 58})$ as $t\to \infty$, then $\c G_t$ is planar with high probability. 
Comparatively, the connection distance needs to satisfy the stronger $r_t = o(t^{-\frac 23})$ condition for zero crossings in its intrinsic drawing with high probability (w.h.p.).

The planarity of the other central random graph in the literature, the Erd\H{o}s--R\' enyi model, has also been investigated. In particular, as shown in the classical paper by Erd\H{o}s and R\' enyi \cite[Theorem~8b]{ER60}, the threshold connection probability for planarity is $p=\frac 12n$, where $n$ is the number of vertices. More precisely, they show that the graph is planar with high probability if $p<cn$ for some $c<\frac 12$, and non-planar w.h.p. if $p>cn$ for some $c>\frac 12$. More recently, work has also been done to understand planarity in the critical window, see e.g. \cite{NVJ15}.

In this paper, we aim to expand on both the question of planarity and number of crossings, by considering a wider range of random geometric graphs, where the existence of each edge is determined by marks on its endpoints, as well as the pair of points itself.
More formally, given independent, identically distributed vertex markings $R_y$ and i.i.d. edge markings $W_{y,z}$, two distinct vertices $y$ and $z$ are connected whenever their distance is at most $\varphi(R_y,R_z,W_{y,z})$ for a certain function $\varphi$.
We focus particularly on models allowing for long edges, corresponding to marks with heavy-tailed distributions.
Observe that for any pair of points, their connection probability depends only on their distance: if $\|y-z\|=x$, then the probability that they form an edge is $\PP(\varphi(R_y,R_z,W_{y,z})\geq x)\eqcolon g(x)$. 

To show the flavour of our general results, we state the following.
\begingroup
\renewcommand{\thetheorem}{A}
\begin{theorem}[Expected number of crossings]\label{thm:A}
    Let $\c G_t$ denote any model of random geometric graph with vertex- and edge markings as defined above, and denote by $X_t(L)$ the number of edge crossings of $\c G_t$ when projected onto the plane $L$.
    Then there exist constants $c_1,c_1'$ and $c_2,c_2'$ depending only on $W$ such that
    \[c_1' t^4 \l[ \int_0^{c_1} g(x) x^{d} \dd x \r]^2 \leq \EE X_t(L)\leq c_2' t^4 \l[ \int_0^{c_2} g(x) x^{d} \dd x\r]^2.\]
    Further, if there exists $(R_t)_{t>0}$ with $R_t\to 0$ as $t\to \infty$ such that the integral of $g(x)x^d$ is dominated by the contribution over the interval $[0,R_t]$,
    then
    \[
    \EE X_t(L) = t^4 \l[\int_0^{R_t} g(x) x^{d}\dd x\r]^2 \frac 18 
    \begin{multlined}[t]
    c_d (d+1)^2 I_W(L)\\ 
        \times \big(1 +\c O_t(R_t) \big)\l(1+\c O_t \l(\frac{ \int_{R_t}^{c_2}g(x) x^{d}\dd x}{\int_0^{R_t}g(x) x^{d}\dd x}\r)\r)
    \end{multlined}
    \]
    with $c_d$ and $I_W(L)$ given in \eqref{eq:crossing constant} on p.~\pageref{eq:crossing constant}.
\end{theorem}
\addtocounter{theorem}{-1}
\endgroup

Note that the precise definition of $X_t(L)$ is given in \eqref{eq:kernel} and the subsequent paragraph on page~\pageref{eq:kernel}.

The paper is structured as follows. 
\Cref{sec:prelim} contains the geometric, probabilistic, and graph-theoretical ingredients necessary for the rest of the paper, including a detailed description of the models.
In the following sections, for each model we determine the asymptotic probability that the graph is planar, as well as study the number of crossings in the projected graph. Apart from standard tools such as the Mecke formula and the second-moment method, we use techniques from integral geometry, as well as an atypical application of induction to find certain subgraph counts. To be able to use induction i.e. build a graph vertex by vertex, the key idea is to consider subgraphs with restricted edge lengths or marks. This allows us to derive a recurrence at the level of functions, with the function parameter being the maximal edge length or mark.
\Cref{sec:RGG} gives an overview of results for the Gilbert graph.
\Cref{sec:SRGG} and \Cref{sec:RR} consider planarity and crossings in two specific heavy-tailed models: one defined purely by edge markings, and one purely by vertex markings.  In both cases, we let the mark distribution have a polynomial tail with an arbitrary exponent.
\Cref{sec:crossings} handles crossings in the most general setting: the proof of Theorem~\ref{thm:A} is given therein (see \Cref{thm:crossing expectation}), as well as certain variance bounds.
Lastly, \Cref{sec:discussion} provides notes and a discussion of the results across the board.

\section{Preliminaries}\label{sec:prelim}


\subsubsection*{Convex- and integral geometry}

Throughout the paper, we work in $d$-dimensional Euclidean space $\R^d$ with $d\geq 2$, equipped with the Euclidean norm $\|\cdot \|$. For a measurable set $A$, write $\lambda_d(A)$ for its Lebesgue measure.
Denote by $B^d(x,r)$ the closed $d$-dimensional ball of radius $r>0$ centred at $x \in \R^d$, and in particular write $B^d$ for the origin-centred unit ball $B^d(\mathbf 0,1)$, and $S^{d-1}$ for its boundary sphere.
We denote the volume of the $d$-dimensional unit ball by $\kappa_d\coloneq \lambda_d(B^d)$.
A convex body is a compact convex set with non-empty interior.

For two points $u,v\in \R^d$, denote the line segment connecting them by $[u,v]$.
For two Borel sets $A,B\subset \R^d$, write $A+B\coloneq \{a+b\colon a\in A, b\in B\}$ for their Minkowski sum. 
Note that $B^d(x,r) = x+rB^d$.
Given a convex body $K$ and positive value $\delta>0$, we write $K_{-\delta}$ for the inner parallel set of $K$, defined as $K_{-\delta}\coloneq \{x\in K\colon x+ \delta B^d \subset K\}$. 

For a plane (i.e. linear subspace of dimension two) $L\subset \R^d$, we write $L^\perp$ for its orthogonal complement; similarly, for a vector $u\in \R^d$, we write $u^\perp$ for the set of vectors orthogonal to $u$, which is the hyperplane (i.e. linear subspace of dimension $d-1$) with normal vector $u$ containing the origin.
Given a plane $L$, and a set $A\subset \R^d$, we write $A|_L$ for the orthogonal projection of $A$ onto $L$; similarly, the projection of a point $v\in \R^d$ onto $L$ is $v|_L$.

\subsubsection*{Poisson processes and markings}

Fix a unit volume convex body $W\subset \R^d$, and let $\eta_t$ be a homogeneous Poisson point process in $W$ with intensity $t>0$.
In addition, we consider an independent $\kappa_v$-marking of $\eta_t$, where $\kappa_v$ is a given probability measure on $[0,\infty)$. For $y\in \eta_t$, write $R_y$ for its corresponding mark, and think of it as a random radius attached to $y$. 
The resulting process $\hat \eta_t \coloneq \{(y,R_y)\colon y\in \eta_t\}$ is a Poisson point process on $\R^d \times [0,\infty)$ with intensity measure $t\lambda_d|_W \otimes \kappa_v$, where $\lambda_d|_W(\cdot)=\lambda_d(\cdot \cap W)$ is the restriction of the Lebesgue measure to $W$. 
Note that we generally write $\mu|_A$ for the restriction of a measure $\mu$ to the set $A$, including using this notation for point processes.

We also consider edge markings of $\eta_t$: attach an independent mark $W_{y,z}$ to every unordered pair of distinct points $y,z\in \eta_t$,  where $W_{y,z}$ is distributed according to the probability measure $\kappa_e$ on $[0,\infty)$.
This can indeed be precisely constructed; for a more rigorous definition, we refer the reader to \cite[Section~2]{LNS21}.

For a positive integer $k$, write $\eta_{t,\neq}^k$ for the set of distinct $k$-tuples in $\eta_t$.
By the Mecke formula (see e.g. \cite[Theorem~4.4]{bookLP18}), 
\[\EE \sum_{(y_1,\ldots, y_k) \in \eta_{t,\neq}^k} h(y_1,\ldots, y_k) = t^k \int\limits_{W^k} h(y_1,\ldots, y_k) \dd y_1 \ldots \dd y_k\]
for any measurable function $h$. 
As an extended version, if $\hat \eta_t$ is a marked process whose mark distribution $\kappa_v$ has Lebesgue density $f$, we have
\begin{equation*}
\EE \sum_{(\hat y_1,\ldots, \hat y_k) \in \hat \eta_{t,\neq}^k} 
h(\hat y_1,\ldots, \hat y_k, \hat \eta_t) = t^k \iint\limits_{(W\times[0,\infty))^k} 
\begin{multlined}[t]
\EE \bigg[ h(\hat y_1,\ldots, \hat y_k, \hat \eta_t \cup \{\hat y_1,\ldots, \hat y_k\})\bigg] \\
\times f(x_1)\cdot \ldots \cdot f(x_k)
\dd y_1 \dd x_1 \ldots \dd y_k \dd x_k.
\end{multlined}
\end{equation*}
where on the right hand side, $\hat y_i\coloneq (y_i,x_i)$.
This equality is also referred to as the Slivnyak--Mecke formula.


\subsubsection*{Random graph models}

As a general setting, we consider a random graph $\c G_t=(V,E)$ with vertex set $V = \eta_t$ and edge set $E$ given by
\[E\coloneq \big\{(y,z)\in \eta_{t,\neq}^2 \colon \|y-z\|\leq \varphi(R_y,R_z,W_{y,z})\big\}\]
for some measurable function $\varphi:[0,\infty)^3\to [0,\infty)$.
We write $y\con z$ for the event that $y,z\in \eta_t$ are connected by an edge.
By construction, the connection probability depends only on the distance of the points, which we denote by $g(\cdot)$: that is, for any two vertices $y\neq z$, 
$g(\|y-z\|) = \PP(y\con z) = \PP(\|y-z\|\leq \varphi(R_y,R_z,W_{y,z})),$
thus $g$ is exactly the tail of $\varphi(R_y,R_z,W_{y,z})$.
Further, the existence of two disjoint edges is independent; however, the existence of two edges that share a common endpoint is generally not.

We mention a few notable examples of specific models. Note that most of this terminology is highly non-standard.
\begin{enumerate}[label = \roman*)]
    \item The hard threshold random geometric graph (RGG or HRGG) or Gilbert graph, given by $\varphi \equiv r_t$.
    \item The soft random geometric graph (SRGG), given by $\varphi = W_{y,z}$. This is a special case of the random connection model (RCM), where each edge exists independently from each other with a given probability, depending on the location of the endpoints.
    \item The vertex marked or random radii RGG using the sum, min, or max kernel, given by $\varphi = R_y+R_z$, $\varphi = \min\{R_y,R_z\}$, or $\varphi = \max \{R_y,R_z\}$, respectively. The model arising from choosing the sum kernel is typically referred to as the (hard) Boolean model, and is understood to mean that two points are connected by an edge whenever the corresponding balls $B^d(y,R_y)$ and $B^d(z,R_z)$ have non-empty intersection. Along the same lines, the min and max kernel mean that either both, or at least one of the balls (respectively) contains the other point.
    \item The soft Boolean model, given by $\varphi = (R_y+R_z) W_{y,z}$, introduced first under this name and formalism in \cite{GGM22} in the context of scale-free percolation.
\end{enumerate}
In this paper, we consider the HRGG, the SRGG, and the random radii RGG with max-kernel, where the mark distribution (both on the vertices and edges) is given by the $\Pareto(r_t,\alpha)$ density $f_{\alpha}(x)\coloneq \alpha r_t^{\alpha} x^{-\alpha-1} \ind (x\geq r_t)$ for some parameter $\alpha>0$. 
Note that $\alpha$ is fixed, and $r_t\to 0$ as $t\to \infty$; for simplicity, we also assume that $r_t$ is monotone decreasing in $t$.
Observe that by construction, each edge in the HRGG is in both of the other models as well.
\Cref{fig:models} shows a realisation of these three models.
Observe the superhubs appearing in the random radii model (e.g. upper left corner): due to the max-kernel, a single large vertex mark produces many edges.

\begin{figure}[ht]
    \centering
    \hfill
    \begin{subfigure}{0.32\linewidth} 
    \includegraphics[width=\textwidth]{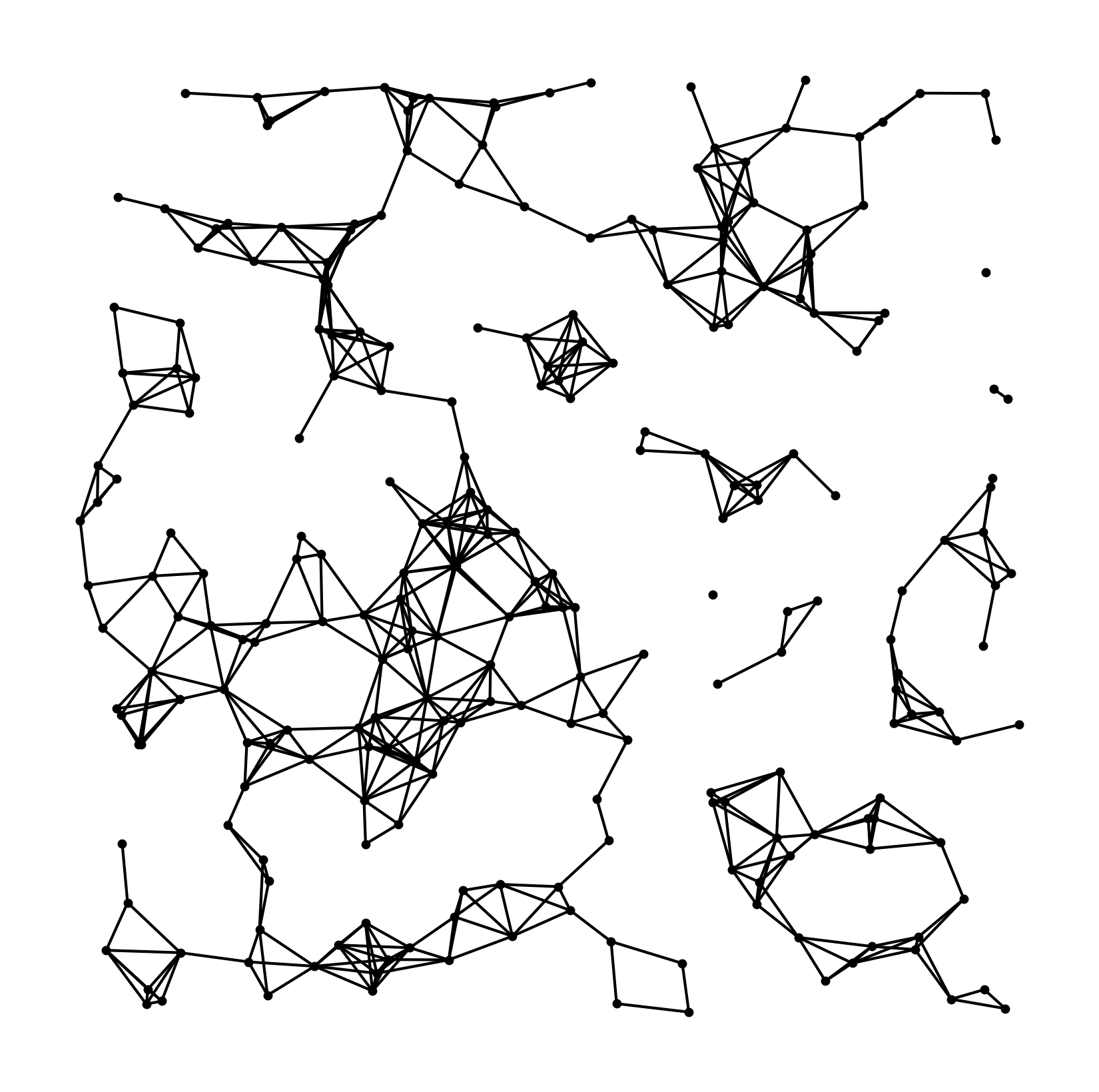}
    \end{subfigure}
    \hfill
    \begin{subfigure}{0.32\linewidth} 
    \includegraphics[width=\textwidth]{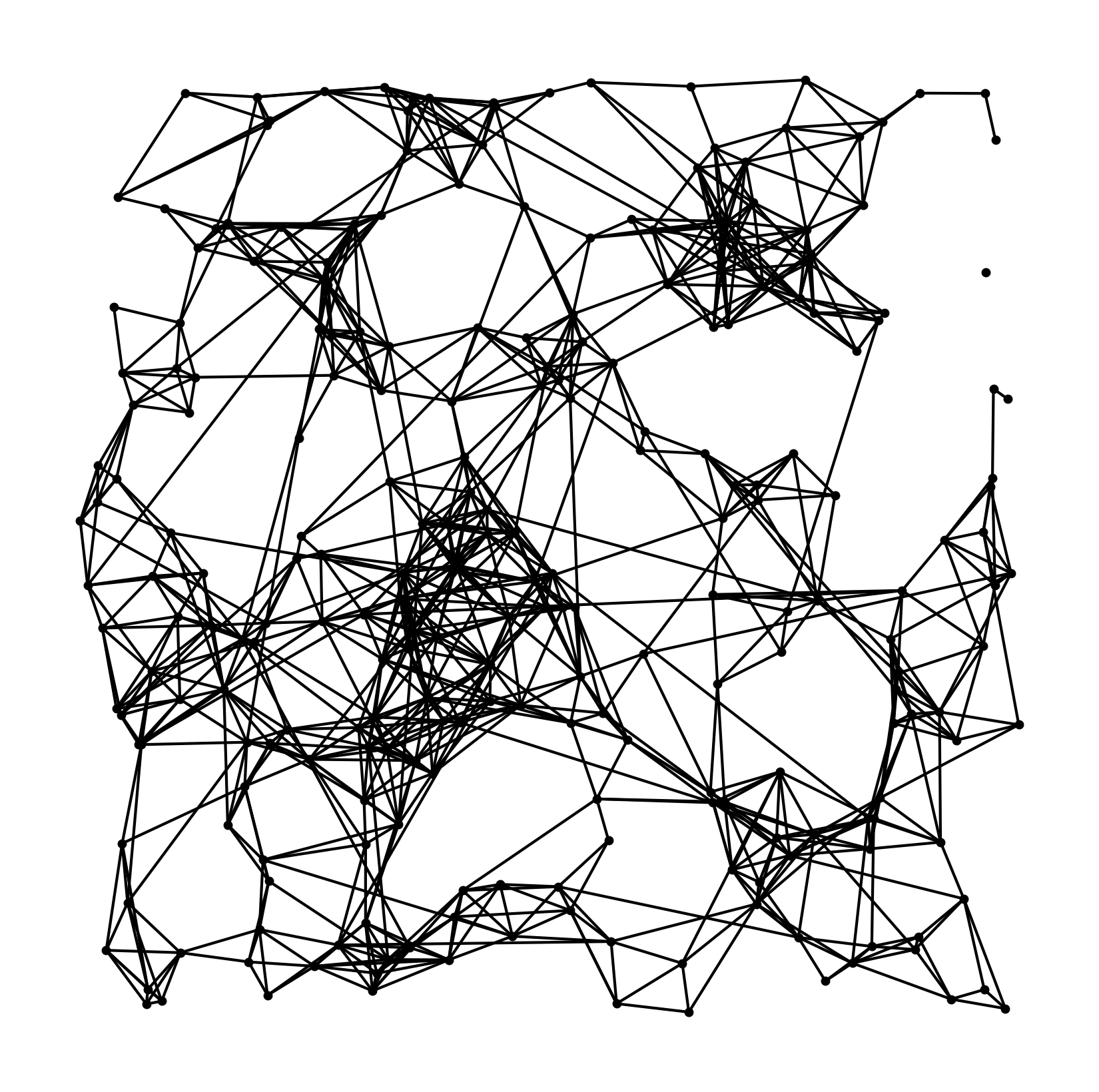}
    \end{subfigure}
    \hfill
    \begin{subfigure}{0.32\linewidth} 
    \includegraphics[width=\textwidth]{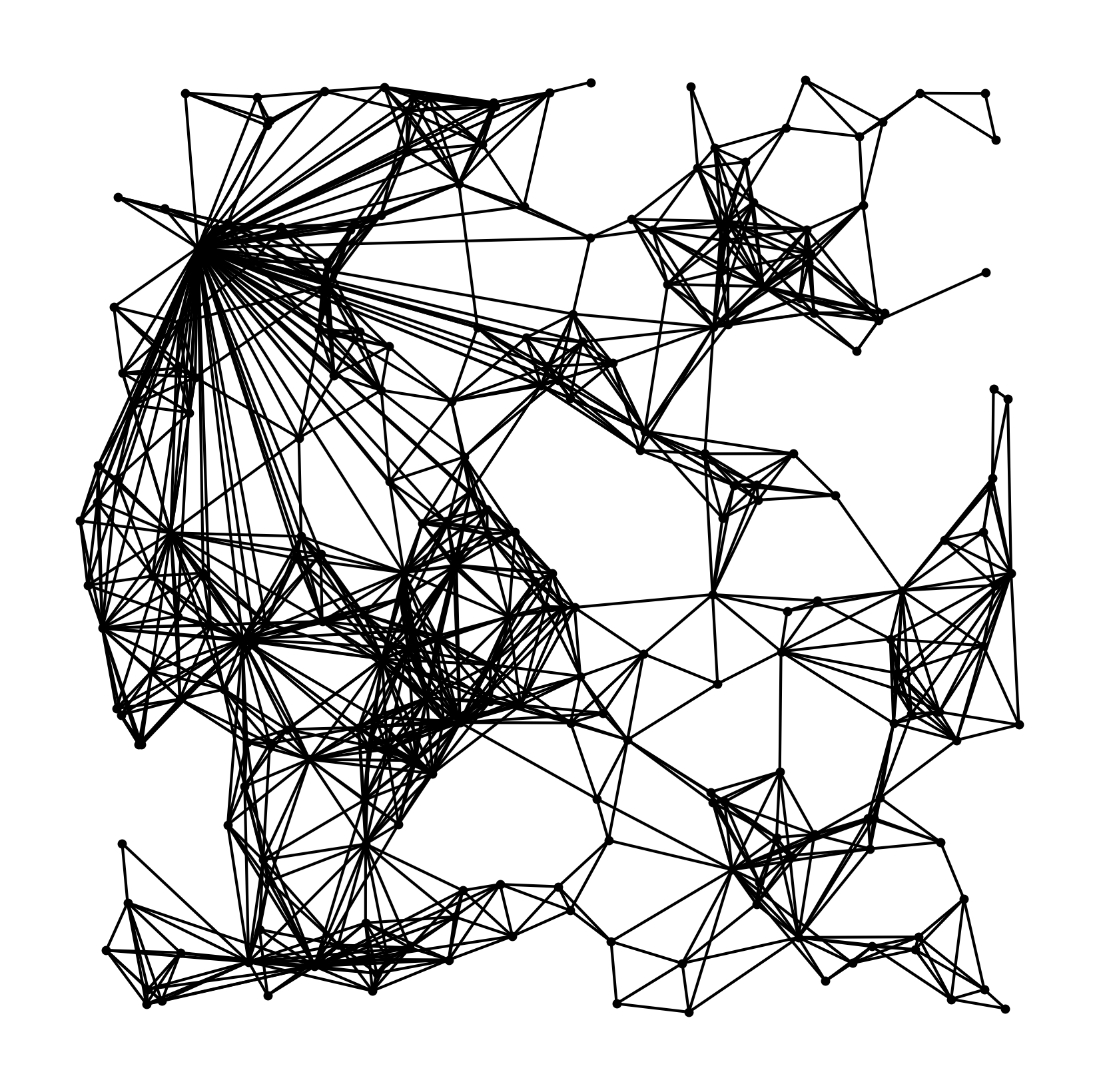}
    \end{subfigure}
    \hfill
    \caption{Left to right: the HRGG, SRGG, and max-kernel random radii RGG on the same point configuration and distance parameter.}
    \label{fig:models}
\end{figure}

We would also like to note that since $\max\{R_y,R_z\}\leq R_y+R_z\leq 2\max\{R_y,R_z\}$, we expect the sum- and max-kernel models to exhibit similar behaviour; we chose the latter purely for computational reasons. The min-kernel model is seemingly much different: due to the symmetry, a vertex with large radius will not necessarily have a high degree.




\subsubsection*{Graph theory}

Let $\Gamma=(V,E)$ be an (abstract) graph with vertex set $V$ and edge set $E$; we only consider simple graphs, meaning no loops or multiple edges allowed. We call the number of vertices $|V|$ the order of the graph.
If $\Gamma'=(V',E')$ with $V'\subset V$ and $E'\subset E$, then $\Gamma'$ is a subgraph of $\Gamma$, and if $E' = E \cap (V')^2$ exactly, then $\Gamma'$ is an induced subgraph. 
Notably, a subgraph does not need to include every edge dictated by $\Gamma$, while an induced subgraph does.
Given a vertex $v\in V$, we write $\Gamma-v$ for the graph without $v$, i.e. the subgraph induced by $V \setminus \{v\}$.
A path of length $n$ is a graph on $n+1$ vertices $x_0,\ldots, x_n$ with exactly the $n$ edges $(x_0,x_1),\ldots, (x_{n-1},x_n)$. 
The complete graph on $n$ vertices (i.e. every pair of distinct points is connected by an edge) is denoted by $K_n$.
Write $K_{n,m}$ for the complete bipartite graph with group sizes $n$ and $m$: that is, $K_{n,m}$ has $n+m$ vertices that can be divided into disjoint sets of $n$ and $m$ elements, where the edges are exactly given by pairs of points in different groups.


A subdivision of $\Gamma$ is a graph where the edges of $\Gamma$ have been replaced with independent paths between the endpoints: that is, none of the inner vertices of the paths intersect each other, or the original vertex set $V$. We refer to the elements of $V$ as the principal vertices, and to the rest as internal vertices.

We say that a graph $\Gamma$ is planar if it has a drawing in the plane without any crossings, that is, its vertices can be embedded into $\R^2$ such that drawing a line segment between vertices forming an edge does not create any edge crossings.
By a standard result in graph theory (see e.g. \cite[Theorem~4.4.6]{Diestel2017}), $\Gamma$ is planar if and only if it does not contain a subdivision of $K_5$ or $K_{3,3}$ as a subgraph. 

For a detailed introduction to graph theory, we refer the reader to \cite{Diestel2017}, and in particular \cite[Chapter~4]{Diestel2017} for a more rigorous introduction to drawings and planarity.

\subsubsection*{Asymptotic notation}

We use the following asymptotic notation: for non-negative functions $h_1$ and $h_2$, we write
\begin{center}
\renewcommand{\arraystretch}{2.1}
    \begin{tabular}{p{2.5cm}cp{4cm}}
    $h_1  = o(h_2)$  & for & $\displaystyle \lim_{x\to a} \frac{h_1(x)}{h_2(x)}  = 0$, \\
    $h_1  = \c O(h_2)$ & for &$\displaystyle\limsup_{x\to a} \frac{h_1(x)}{h_2(x)}<\infty$, \\
    $h_1 \sim h_2$ & for &$\displaystyle \lim_{x\to a} \frac{h_1(x)}{h_2(x)} = 1$,
    \end{tabular}
\end{center}
where we typically either have $a=0$ or $a=\infty$.
Lastly, we write $h_1 = \Theta(h_2)$ if $h_1= \c O(h_2)$ and $h_2 = \c O(h_1)$; this corresponds to the two functions being at most a multiplicative constant away from each other when close enough to $a$. 
We also write $o_x, \c O_x$ and $\Theta_x$ when $a=\infty$, and omit the $x\to \infty$ notation.
In addition, to keep computations more transparent, throughout the proof, we also write $h_1 \lesssim h_2$, $h_1\gtrsim h_2$, $h_1\approx h_2$ whenever $h_1 = \c O_t(h_2)$, $h_2 = \c O_t(h_1)$, and $h_1 = \Theta(h_2)$, respectively.



\section{The hard random geometric graph}\label{sec:RGG}

In this section, we let $\c G_t$ denote the hard threshold random geometric graph, that is, whose edge set is given by $\{(y,z)\in \eta_{t,\neq}^2\colon \|y-z\|\leq r_t\}$ for some $r_t>0$.

First, we state an existing result.

\begin{theorem}[{Number of crossings in the HRGG \cite[Theorem~12]{CDRarxiv}, \cite[Lemma~1]{DdJ25}}]
Let $X_t(L)$ denote the number of crossings of the hard random geometric graph $\c G_t$ when projected onto the plane $L$. Then
\[\EE X_t(L) = t^4 r_t^{2d+2} \frac 18 c_dI_W(L)  (1+ \c O_t(r_t))\]
with $c_d$ and $I_W(L)$ given in \eqref{eq:crossing constant} on p.~\pageref{eq:crossing constant}.
\end{theorem}
\begin{remark}\label{remark:RGG constant}
    The correct constant $c_d$ differs from the one actually stated in \cite[Theorem~12]{CDRarxiv}; this is due to a missing exponent in the computation in \cite[Lemma~8]{CDRarxiv}. 
    Note also that \cite[Theorem~12]{CDRarxiv} states the first order asymptotic, while \cite[Lemma~1]{DdJ25} derives the order of the error term.
\end{remark}

The following result is a consequence of well-known properties of subgraph counts in the Gilbert graph.

\begin{theorem}[{Planarity of the HRGG}]\label{thm:RGGplanarity}
    Let $\c G_t$ denote the hard random geometric graph. 
    Then
    \[\lim_{t\to \infty} \PP(\c G_t \text{ planar}) = 
        \begin{cases}
            1, & \text{ if } \lim_{t\to \infty} r_t t^{\frac{5}{4d}} = 0 \\
            e^{-\lambda}, & \text{ if }   \lim_{t\to \infty} r_t t^{\frac{5}{4d}} \in (0,\infty) \\ 
            0, & \text{ if }   \lim_{t\to \infty} r_t t^{\frac{5}{4d}} = \infty  \\ 
        \end{cases}
        \]
        where $\lambda \coloneq \lim_{t\to \infty} \EE G_{K_5}\in(0,\infty)$ is the expected number of subgraphs of $\c G_t$ isomorphic to $K_5$.
\end{theorem}
\begin{proof}
    For each case, we consider the characterisation of planarity through subgraphs.
    For an arbitrary connected graph $\Gamma$ on $n$ vertices, let $G_\Gamma$ denote the number of induced subgraphs of $\c G$ isomorphic to $\Gamma$. The behaviour of $G_\Gamma$ is well understood; see \cite[Section~6]{RGGbook}.
    In particular (see e.g. \cite[Proposition~3.1]{RGGbook}), the expectation $\EE G_\Gamma$ is of the order $t(tr_t^d)^{n-1}$.
    This implies that if $t(tr_t^d)^4 = t^5 r_t^{4d}\to 0$, no (connected) subgraph of order $5$ appears in $\c G_t$ with high probability, implying the planarity of the graph.
    If $t(tr_t^d)^4\to c\in(0,\infty)$, then on one hand, $t(tr_t^d)^5\to 0$, implying no subgraphs of order $6$ with high probability: thus no subdivisions of $K_{3,3}$, or proper subdivision of $K_5$ with high probability, and the only potential obstruction to planarity is a $K_5$ itself. On the other hand, we have (see \cite[Theorem~3.4]{RGGbook}) that $G_{K_5}$ is asymptotically Poisson distributed with parameter $\lim_{t \to \infty} \EE G_{K_5}=\lambda$. As a consequence, planarity is determined by the non-existence of a $K_5$-subgraph, whose probability converges to $\PP(\Po(\lambda)=0)=e^{-\lambda}$.
    Lastly, if $t(tr_t^d)^4\to \infty$, implying $\EE G_{K_5}\to \infty$, it can be shown using a variance argument (see \cite[Proposition~3.7]{RGGbook}) that $\PP(G_{K_5}\geq 1)\to 1$, implying non-planarity with high probability.
\end{proof}

\section{The soft random geometric graph}\label{sec:SRGG}
{In this section, we consider the soft random geometric graph with $\Pareto(r_t,\alpha)$ edge weights. More precisely, let $W_{y,z}$ be an i.i.d. edge marking of $\eta_t$ with tail distribution $\PP(W_{y,z}\geq s) = \min \{1,r_t^{\alpha} s^{-\alpha}\}$ for parameters $r_t,\alpha>0$. 
We write $\c G_t  = (\eta_t, \{(y,z)\in \eta_{t,\neq}^2\colon \|y-z\|\leq W_{y,z}\})$ for the resulting soft RGG.}

\subsection{Crossings in the SRGG}

\begin{theorem}[Crossings in the Pareto SRGG]\label{thm:SRGG crossings}
    Let $\c G_t$ denote the soft random geometric graph with $\Pareto(r_t,\alpha)$ edge weights, and $X_t(L)$ the number of crossings of $\c G_t$ when projected onto a plane $L$.
    Then, if $\alpha>d+1$,
    \begin{align*}
        \EE X_t(L)  & = t^4r_t^{2d+2} \frac{c_d}{8}\l( \frac{\alpha}{d+1-\alpha}\r)^2I_W(L)\l(1+ \c O_t\l(r_t^{\frac{\alpha -d-1}{\alpha - d}}\r)\r),\\
    \intertext{and if $\alpha=d+1$, }
        \EE X_t(L) & = t^4r_t^{2d+2} \l(\ln \l(\frac{1}{r_t}\r)\r)^2 \frac{c_d (d+1)^2}{8}I_W(L) \l(1+ \c O_t\l( \frac{\ln \l(\ln (1/r_t)\r)}{\ln ( 1/r_t)}\r)\r),
    \end{align*}
    with $c_d$ and $I_W(L)$ given in \eqref{eq:crossing constant}.
    Further, if $\alpha\geq d+1$ and $\lim_{t\to \infty} \EE X_t(L) = \lambda\in (0,\infty)$, then $\VV X_t(L) = \EE X_t(L) + o_t(1)$, and $X_t(L)$ tends in distribution to a Poisson distributed random variable with parameter $\lambda$. 
    
    If $\alpha < d+1$, then
    \[\EE X_t(L) = \Theta_t\l( t^4r_t^{2\alpha}\r),\]
    and if $tr_t^\alpha\to 0$ as $t\to \infty$, then $\VV X_t(L) = \EE X_t(L) + \Theta_t \l(t^6 r_t^{3\alpha} \r).$
\end{theorem}
\begin{remark}
    The $\alpha$-threshold heuristically means that for $\alpha<d+1$, most edge crossings come from crossings between two short edges, and for $\alpha>d+1$, from ones between two long edges.
    In the next subsection (\Cref{thm:SRGGcliques}), we will see that a similar $\alpha$-transition happens for the edges, but at $\alpha = d$. This means that when $d < \alpha < d+1$, the graph itself is dominated by short edges, but the crossings by ones coming from long edges: that is, though there are significantly more short edges, they cannot overtake the long ones in regard to crossings.
\end{remark}
\begin{remark}
    If $\alpha<d+1$ and $\EE X_t(L)$ tends to a non-zero constant, then $\VV X_t(L)-\EE X_t(L)$ tends to a non-zero constant as well: this indicates that the number of crossings is not asymptotically Poisson distributed, unlike in the $\alpha\geq d+1$ case. Heuristically, the variance is larger because a single long edge can generate many crossings simultaneously; we will also be able to see this computationally, as the appearance of Configuration 2.1 of \Cref{fig:variance} yields the dominating expression (as long as $tr_t^\alpha \to 0$).    
    Note also that precise variance asymptotics can be derived following \Cref{sec:crossings variance}; we omit stating this explicitly, as there are many distinct cases (depending on both $\alpha$ and $r_t$). 
\end{remark}
\begin{proof}
    We find the expected number of crossings using Theorem~\ref{thm:A}. Since the probability that two points at distance $x>0$ are connected by an edge is $g(x) = \min\{1,r_t^\alpha x^{-\alpha}\}$, we must consider the integral
    \[\int_0^{R_t} g(x) x^{d} \dd x = \int_0^{r_t} x^{d} \dd x +  r_t^\alpha\int_{r_t}^{R_t} x^{d+1-\alpha-1} \dd x \]
    with some $c_2>R_t\geq r_t$.

    The first integral is $r_t^{d+1}/ (d+1)$ independently of the value of $\alpha$. The second integral depends on the sign of $d+1-\alpha$: in particular, for $\alpha > d+1$, the integral is dominated by the lower endpoint $r_t$ of the interval, and for $\alpha<d+1$, by the upper endpoint $R_t$.
    For $\alpha<d+1$, it follows that no $R_t\to 0$ exists such that the main contribution comes from the interval $[0,R_t]$. If $R_t=c$ for any constant $c>0$, the integral is $\Theta_t(r_t^{\alpha})$, and as a consequence, $\EE X_t(L) = \Theta_t (t^4 (r_t^\alpha)^2$) follows from Theorem~\ref{thm:A}.

    For $\alpha > d+1$, the integral over $[0,c_2]$ is dominated by the integral over $[0,R_t]$ 
    for any $R_t$ with $R_t/r_t\to \infty$.
    In particular, for any such $R_t$,
    \[r_t^{\alpha}\int_{r_t}^{R_t} x^{d+1-\alpha-1}\dd x = r_t^{\alpha} \frac{\l[R_t^{d+1-\alpha} - r_t^{d+1-\alpha} \r]}{d+1-\alpha} \sim \frac{r_t^{d+1}}{\alpha-d-1},\]
    since $d+1-\alpha<0$; recall that the $\sim$ symbol means that the ratio of the quantities converges to $1$ as $t\to \infty$.
    Together with the integral over $[0,r_t]$, this gives the first order term of the expectation. 
    Consider now the error term.
    We aim to optimize the two different types of errors. 
    Some error comes from the omitted part of the integral, corresponding to the negligibility of long edges: this is of the order $(R_t/r_t)^{d+1-\alpha}$.
    In addition, some error comes from approximations in the proof of Theorem~\ref{thm:A}, which has order $R_t$.
    The overall error term is optimal if these two expressions are equal, which happens exactly for $R_t= r_t^{\frac{\alpha-d-1}{\alpha-d}}$. This expression satisfies both $R_t\to 0$ and $R_t/r_t\to \infty$, and hence indeed gives the error term as stated in the theorem.

    For $\alpha=d+1$, 
    \[r_t^\alpha \int_{r_t}^{R_t} x^{-1}\dd x = r_t^\alpha \big (\ln (R_t) - \ln (r_t) \big)
    \sim r_t^\alpha \ln \l(\frac{1}{r_t}\r)\]
    whenever $\ln(R_t)/\ln(r_t)\to 0$. 
    Setting $R_t = \frac{\ln(\ln(1/r_t))}{\ln(1/r_t)}$, this holds, and the first order term of the expectation follows.
    The error term coming from the omitted integral is $\ln(R_t)/\ln(r_t)\sim \ln (\ln(1/r_t))/\ln(1/r_t) =R_t$, yielding the overall error term.

    For the variance, a complete description is given in \Cref{sec:crossings variance}. With the notation used therein, the independence of the edge markings implies that the connection probability is $G(\mathbf s)=\prod_{i=1}^m (r_t^\alpha s_i^{-\alpha})$ for all configurations arising from the application of the Mecke formula. The resulting integrals are of power functions, whose behaviour is determined by $\alpha$, and in particular changes in behaviour occur at $\alpha = d$, $\alpha = d+1$ and $\alpha = d+2$.
    If $\lim_{t\to \infty} \EE X_t(L)<\infty$, then $\VV X_t(L)-\EE X_t(L)$ is dominated by Configuration 3.1 if $\alpha<d+1$, and Configuration 2.1 or 3.1 otherwise. 
    In the former case, Configuration 3.1 yields the order $\Theta_t(t^6 r_t^{3\alpha}) = \Theta_t((t^4 r_t^{2\alpha})^{3/2})$ stated in the theorem.
    For the latter, different dominating expressions arise (depending on $\alpha$ and $r_t)$, but we always have $\VV X_t(L)-\EE X_t(L)= o_t(1)$.

    To prove the convergence to a Poisson distributed random variable in the case when $\alpha\geq d+1$ and $\EE X_t(L)\to \lambda \in (0,\infty)$, we use the Poisson approximation result \cite[Theorem~3.8]{Pianoforte} (see in particular (3.27)). 
    Though they only state the bounds for Poisson point processes (without edge markings), the proof is easily adapted to our setting due to the markings being i.i.d.
    By this result, if $Z$ is a Poisson distributed random variable with parameter $\EE X_t(L)$, then the total variation distance $\dTV$ is bounded by
        \begin{equation*}
            \dTV(X_t(L),Z)
            \leq 16 (\VV X_t(L)-\EE X_t(L)).
        \end{equation*}
    This gives the statement together with the variance bounds, and in particular a rate for the convergence can also be given.
\end{proof}

\subsection{Complete graphs in the SRGG}

For our second main question, planarity, we first aim to understand the number of $K_5$ subgraphs in $\c G_t$.
Therefore in this section, we consider complete graphs generally.

\begin{theorem}[Complete subgraphs in the Pareto SRGG]\label{thm:SRGGcliques}
Let $G_{K_n}$ denote the number of complete subgraphs on $n$ vertices in the soft random geometric graph with $\Pareto(r_t,\alpha)$ edge weights. 
Then
\begin{align*}
    \EE G_{K_n} & = 
    \begin{cases}
        \Theta_t \l(t^n r_t^{d(n-1)}\r), & \text{ if } \frac{\alpha}{d} > \frac 2n,\\
        \Theta_t  \l(t^n r_t^{d(n-1)} \ln \l(\frac{1}{r_t}\r)\r), & \text{ if } \frac{\alpha}{d} =  \frac 2n,\\
        \Theta_t \l( t^n r_t^{\alpha n(n-1)/2}\r),& \text{ if } \frac{\alpha}{d}< \frac 2n.
    \end{cases}
    \intertext{Further,}
    \VV G_{K_n} & =  \EE G_{K_n} + \sum_{m=1}^{n-1} \Theta_t  (V_{n,m}(t))
    \intertext{with}
    V_{n,m}(t) & =  
    \begin{cases}
        t^{2n-m} r_t^{da_{n,m}}, & \text{ if } \frac{\alpha}{d} > \frac{a_{n,m}}{b_{n,m}},\\
        t^{2n-m}  r_t^{da_{n,m}}\ln \l(\frac{1}{r_t}\r), & \text{ if } \frac{\alpha}{d} = \frac{a_{n,m}}{b_{n,m}},\\
        t^{2n-m} r_t^{\alpha b_{n,m}}, & \text{ if } \frac{\alpha}{d} < \frac{a_{n,m}}{b_{n,m}},
    \end{cases}
\end{align*}
    where $a_{n,m}=2n-m-1$ and $b_{n,m}= n(n-1)-\frac 12 m(m-1)$.
    
    If $\lim_{t\to \infty} \EE G_{K_n} =\lambda<\infty$, then $\VV G_{K_n} - \EE G_{K_n} = \Theta_t (V_{n,n-1}(t))$. 
    In addition, if $\lambda>0$, and $Z$ is a Poisson distributed random variable with parameter $\EE G_{K_n}$, then the total variation distance is bounded by
    \[\dTV (G_{K_n}, Z)  = \Theta_t (V_{n,n-1}(t)) = o_t(1),\]
    and in particular, $G_{K_n}$ converges in distribution to a $\Po(\lambda)$ random variable as $t\to \infty$.
\end{theorem}
\begin{remark}
    Depending on $\alpha$, the correct case for $V_{n,m}(t)$ (determined by the sign of $\alpha/d- a_{n,m}/b_{n,m}$) might be different for different values $m$ (for $n$ fixed). 
    Further, which expression dominates across the values of $m$ heavily depends on the asymptotic behaviour of $r_t$ as well.
\end{remark}

\begin{proof}
Let $n\geq 1$ be arbitrary. 
Write $h(y_1,\ldots, y_n)\coloneq \prod_{1\leq i<j\leq n} \ind (y_i\con y_j)$ for the indicator of the event that the points $y_1,\ldots, y_n\in \R^d$ form a complete subgraph in the generated soft random geometric graph.
More precisely, assign an independent, $\textnormal{Pareto}(r_t,\alpha)$ distributed random variable $W_{y_i,y_j}$ to each pair $\{y_i,y_j\}$, and let
$h(y_1,\ldots, y_n) = \prod_{1\leq i < j \leq n} \ind (\|y_i-y_j\| \leq W_{y_i,y_j})$.
We can then write $G_{K_n}$ as a Poisson $U$-statistic, namely
\[G_{K_n} = \frac{1}{n!}\sum_{(y_1,\ldots, y_n)\in \eta_{t,\neq}^n} h(y_1,\ldots,y_n).\]

\paragraph{The expectation: recurrence.} 
By the Mecke formula, 
\begin{align*}
    \EE G_{K_n} 
    & = \frac{t^n}{n!} \int_{W^n} \PP(h(y_1,\ldots, y_n)=1)\dd y_1\ldots \dd y_n \\
    & = \frac{t^n}{n!} \int_W \l[\int_{(W-v)^{n-1}}\PP(h(\mathbf 0,y_2,\ldots, y_n)=1)\dd y_2\ldots \dd y_n\r] \dd v
\end{align*}
with $\mathbf 0 = (0,\ldots, 0)\in \R^d$, where we use that the connection probabilities are translation invariant. Observe also that since the edge weights are independent, the connection probability factors, and in particular $\PP(h(y_1,\ldots, y_n)=1)=\prod_{1\leq i < j\leq n} g(\|y_i-y_j\|)$, where $g(s)\coloneq \PP(W_{y,z}\geq s) = \min\{1, r_t^{\alpha}s^{-\alpha}\}$ for arbitrary $s>0$.

We evaluate the expectation using induction on $n$, by understanding the arising integral in a more granular way by restricting the maximal edge lengths allowed in the construction. In particular, set $y_1\coloneq \mathbf 0\in \R^d$, and for arbitrary $s>0$, write
\[I_{n,s}\coloneq \int_{(s B^d)^{n-1}} \PP(h(y_1,\ldots, y_n)=1) \ind (\|y_i-y_j\|\leq s  \ \forall i,j\in\{1,\ldots, n\})\dd y_2\cdots \dd y_n\]
for $n\geq 2$, and set $I_{1,s}=1$.
Note that though it is not indicated, the connection probability, and hence $I_{n,s}$, is a function of $t$.
This expression is then closely related to $\EE G_{K_n}$: there exist constants $0<\gamma_1 < \gamma_2 <\infty$ such that
\begin{equation}\label{eq:ball sandwich}
    \frac{\gamma_1^d\kappa_d}{n!} t^n I_{n,\gamma_1}\leq \EE G_{K_n} \leq  \frac{\gamma_2^d\kappa_d}{n!}t^n I_{n,\gamma_2}.
\end{equation}
The upper bound can be seen with $\gamma_2\coloneq \diam (W)$ by expanding the points of the inner integral to $\diam (W)B^d\supset W-v$, and using that $\Vol(W)\leq \Vol(\diam (W)B^d) = \gamma_2^d \kappa_d$. 
On the other hand, since $W$ has non-empty interior, there exists $\gamma_1$ such that $W$ contains a ball of radius $2\gamma_1$, i.e. there is $z\in W$ with $z+2\gamma_1B^d\subset W$. Restricting $v$ to $z+\gamma_1B^d$, and the points of the inner integral to $\gamma_1B^d\subset W-v$ yields the lower bound.

In what follows, we aim determine the asymptotic behaviour of $I_{n,s_t}$ as $t \to \infty$ for all $(s_t)_{t>0}$ with $0<s_t<\gamma_2$. For simplicity, we assume that $s_t$ is monotone decreasing; importantly, $s_t\to 0$ as $t\to \infty$ is allowed. 
We will in particular show that $I_{n,c_1} = \Theta(I_{n,c_2})$ as $t\to \infty$ for arbitrary constants $c_1,c_2>0$, yielding that $\EE G_{K_n} = \Theta_t  (t^n I_{n,1})$.

To find the correct asymptotic, we first show that the recurrence
\begin{equation}\label{eq:SRGG clique recurrence}
\int_0^{s_t} (g(x))^{n} I_{n,x}x^{d-1} \dd x \lesssim I_{n+1,s_t}  \lesssim \int_0^{s_t} (g(x))^{n} I_{n,2x}x^{d-1} \dd x
\end{equation}
holds for all $n\geq 1$.
The idea is the following.
In the definition of $I_{n+1,s_t}$, for every fixed $\{y_2,\ldots, y_{n+1}\}$, we can achieve through a permutation of the indices that every edge has length at most $2\|y_{n+1}\|$, that is, $\|y_i-y_j\|\leq 2\|y_{n+1}\|$ for all $i,j=1,\ldots, n+1$; in particular we have $\|y_i\|\leq 2 \|y_{n+1}\|$ for all $i=2,\ldots, n$. 
The existence of such a permutation is guaranteed by the triangle inequality (e.g. if $\|y_{n+1}\|$ is the longest edge), but it is generally not unique, and the number of such permutations may vary. However, this changes the integral by at most a constant factor, giving matching upper- and lower bounds. As a consequence,
\begin{multline*}
    I_{n+1,s_t} \approx \int_{s_t B^d}\bigg[\int_{(2\|y_{n+1}\| B^d)^{n-1}}  \PP(h(y_1,\ldots,y_n)=1) \prod_{i=1}^{n} g(\|y_i-y_{n+1}\|)\\ 
    \times \ind (\|y_i-y_j\|\leq 2\|y_{n+1}\|  \ \forall i,j\in\{1,\ldots, n+1\})\dd y_2\cdots \dd y_{n}\bigg] \dd y_{n+1}.
\end{multline*}
By symmetry, the inner $(n-1)$-fold integral depends only on $\|y_{n+1}\|$. By polar transformation on $y_{n+1}$, this gives
\begin{align}
\begin{split}\label{eq:clique induction}
    I_{n+1,s_t} \approx \int_0^{s_t} \bigg[\int_{(2x B^d)^{n-1}} & \PP(h(y_1,\ldots,y_n)=1)  \prod_{i=1}^{n} g(\|y_i-y_{n+1}(x)\|) \\ 
    &\times  \ind (\|y_i-y_j\|\leq 2x  \ \forall i,j\in\{1,\ldots, n+1\})\dd y_2\cdots \dd y_{n}\bigg]  x^{d-1} \dd x,
\end{split}
\end{align}
where we now set $y_{n+1} = y_{n+1}(x)= (x,0,\ldots, 0)\in \R^d$ in the inner integral.

First, consider a lower bound.
If $y_i \in xB^d$, then $\|y_i-y_{n+1}\|\leq 2x$ is automatically satisfied, hence restricting all points $y_2,\ldots, y_n$ to $xB^d$ allows us to drop the condition on those distances. Further, by the monotonicity of the connection function, $g(\|y_i-y_{n+1}\|) \geq g(2x) = \Theta_t (g(x))$ for all $i=1,\ldots, n$, giving the lower bound
\begin{align*}
    I_{n+1,s_t} & \gtrsim \int_0^{s_t} (g(x))^{n}
    \bigg[ \int_{(x B^d)^{n-1}} 
        \begin{multlined}[t]\
        \PP(h(y_1,\ldots,y_n)=1) \\ \times  \ind (\|y_i-y_j\|\leq x  \ \forall i,j\in\{1,\ldots, n\})\dd y_2\cdots \dd y_{n}\bigg]  x^{d-1} \dd x
        \end{multlined}\\ 
        & = \int_0^{s_t} (g(x))^{n} I_{n,x}x^{d-1} \dd x.
\end{align*}

Secondly, we obtain an essentially matching upper bound. 
By the distance assumptions, $y_2, \ldots, y_{n}\in (2x B^d)\cap (y_{n+1}+2xB^d)\eqcolon A_{x}$, which is a convex body symmetric about the bisector hyperplane $H$ of $y_1=\mathbf 0$ and $y_{n+1}$. Write $A_x^-$ for the points of $A_x$ that are closer to $y_1$ than to $y_{n+1}$.
For any fixed position of $y_2,\ldots, y_n$, reflecting an individual point $y_i$ about $H$ does not change the (joint) probability of being connected to $y_1$ and $y_{n+1}$. However, the probability that $\{y_2,\ldots, y_{n}\}$ forms a complete subgraph is maximal exactly when they are all on the same side of $H$; see \Cref{fig:SRGG clique}. 
In other words, for any configuration of points, replacing the ones closer to $y_{n+1}$ with their reflected pair increases the overall connection probability. 
\begin{figure}[ht]
    \centering
    \begin{tikzpicture}[scale=0.8]
        \coordinate (x) at (0,0);
        \draw [fill=black] (x) circle (2pt);
        \node at (x) [left = 2pt] {$y_1$};
        
        \coordinate (xnn) at (6,0);
        \draw [fill=black] (xnn) circle (2pt);
        \node at (xnn) [right =2pt] {$y_{n+1}$};

        \draw[line width = 1pt](x)--(xnn);
        \draw[line width = 1pt, dashed] (3,-1.5)--(3,3.5);
        \node at (3, -1) [right] {$H$}; 

        \coordinate (y2) at (5,1.5);
        \draw [fill=black] (y2) circle (2pt);
        \node at (y2) [above =2pt] {$y_{2}$};

        \coordinate (y2v) at (1,1.5);
        \draw [fill=black] (y2v) circle (2pt);
        \node at (y2v) [above =2pt] {$y_{2}'$};
        
        \coordinate (y3) at (2.5,2.5);
        \draw [fill=black] (y3) circle (2pt);
        \node at (y3) [above =2pt] {$y_{3}$};

        \draw[line width = 0.5pt, dotted] (x)--(y2)--(xnn);
        
        \draw[line width = 0.5pt, dotted] (x)--(y2v)--(xnn);
        
        \draw[line width = 0.5pt, dashed] (y2)--(y3)--(y2v);
    \end{tikzpicture}
    
    \caption{For fixed $y_1,y_3, y_{n+1}$, choosing $y_2'$ instead of $y_2$ increases the overall connection probability.}
    \label{fig:SRGG clique}
\end{figure}
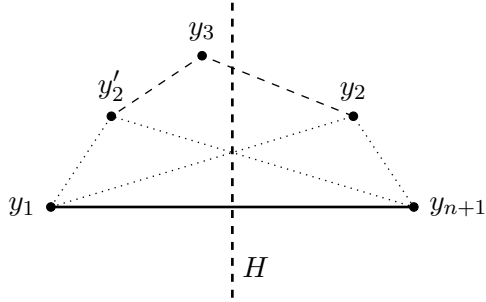

The total inner integral in \eqref{eq:clique induction} can then be upper bounded by a constant times the integral over $A_x^-$. 
Importantly, for $y_i\in A_x^-$, we have $\|y_1-y_i\|\leq \|y_{n+1}-y_i\|$ by definition of the set, which implies by the triangle inequality that $\|y_{n+1}-y_i\|\geq x/2$, and consequently $g(\|y_{n+1}-y_i\|)\lesssim g(x)$.
This leads to the upper bound
\begin{align*}
    I_{n+1,s_t} 
    &\lesssim \int_0^{s_t} (g(x))^n \bigg[\int_{(A_{x}^-)^{n-1}} 
    \begin{multlined}[t]
    \PP(h(y_1,\ldots,y_n)=1) \\ \times  \ind (\|y_i-y_j\|\leq 2x  \ \forall i,j\in\{2,\ldots, n\})\dd y_2\cdots \dd y_{n}\bigg]  x^{d-1} \dd x.
    \end{multlined}
\end{align*}
The inner integral is upper bounded by $I_{n,2x}$ by extending the domain of integration to $(2xB^d)^{n-1}$.

\paragraph{The expectation: explicit computation.} 


We now give a short overview of how the recurrence above leads to the asymptotic behaviour of $I_{n,s_t}$: though the integrals are elementary (power functions and logarithms), the parameters and case distinctions require some attention.

Our claim is the following.
If $s_t\leq 2r_t$ for $t$ sufficiently large, then $I_{n,s_t} \approx s_t^{d(n-1)}$, while if $s_t>2 r_t$,
\begin{equation*}
I_{n,s_t} \approx \begin{cases}
            r_t^{d(n-1)}, & \text{ if } \frac{\alpha}{d} > \frac 2n,\\
            r_t^{d(n-1)} \ln \l(\frac{s_t}{r_t}\r), & \text{ if }  \frac{\alpha}{d} = \frac 2n,\\
            r_t^{\alpha n(n-1)/2} s_t^{d(n-1)-\alpha n(n-1)/2},& \text{ if }  \frac{\alpha}{d} < \frac 2n
        \end{cases}
\end{equation*}
holds for all for $n\geq 2$; recall that we have $I_{1,s_t}=1$ by definition.
Observe that changing the radius $s_t$ by a constant factor only changes the integral by a constant factor. This yields that the lower- and upper bounds in \eqref{eq:SRGG clique recurrence} match; additionally, setting $s_t=1$, the statement of the theorem follows with \eqref{eq:ball sandwich}.

We show the asymptotic behaviour using the recurrence \eqref{eq:SRGG clique recurrence} from $n$ to $n+1$ when $n\geq 2$. We omit the base case $I_{2,s_t}$, as it follows from $I_{1,s_t}=1$ using the same tools presented below for the $n\geq 2$ case.

Let us now begin the explicit computations.
For $s_t\leq 2r_t$, the statement easily follows from the recurrence: since $g(x)\approx 1$ for all $x\leq 2r_t$, we have $I_{n+1,s_t} \approx \int_0^{s_t} s_t^{d(n-1)} x^{d-1} \dd x \approx s_t^{dn}$.

For the rest of the computation, assume that $s_t > 2 r_t$ for large enough $t$. 
Split the integral in \eqref{eq:SRGG clique recurrence} into two parts at $x = r_t$, call them $I_{n+1,s_t}'$ and $I_{n+1,s_t}''$. 
The first part does not depend on $\alpha$: using the induction hypothesis,
\[I_{n+1,s_t}'\approx \int_{0}^{r_t} x^{d(n-1)}x^{d-1}\dd x = \int_0^{2r_t} x^{dn-1}\dd x \approx r_t^{dn}.\]
If $\alpha> \frac 2n d > \frac{2}{n+1} d$, then $d-n\alpha < -d < 0$, and the contribution coming from the lower endpoint of the interval dominates the expression:
\[I_{n+1,s_t}'' 
    \approx \int_{r_t}^{s_t}  r_t^{n\alpha} x^{-n\alpha} \cdot r_t^{d(n-1)}\cdot x^{d-1}\dd x
    \lesssim r_t^{n\alpha + d(n-1)} r_t^{d-n\alpha} = r_t^{dn}.\]

If $\alpha<\frac 2n d$, 
\begin{align*}
    I_{n+1,s_t}'' 
    & \approx \int_{r_t}^{s_t} r_t^{n\alpha} x^{-n\alpha} \cdot r_t^{\alpha n(n-1)/2} x^{d(n-1)-\alpha n(n-1)/2} \cdot x^{d-1} \dd x \\
    & = r_t^{\alpha n(n+1)/2} \int_{r_t}^{s_t} x^{dn-\alpha n(n+1)/2-1}\dd x.
\end{align*}
The behaviour of this expression is determined by the sign of $dn-\alpha n (n+1)/2$.
For negative sign i.e. $\alpha$ larger than the threshold, the main contribution of the integral comes from the lower endpoint $r_t$, the $r_t^\alpha$ terms cancel out, and only $r_t^{dn}$ remains.
For positive sign i.e. $\alpha$ smaller, the upper endpoint dominates, and we get the $s_t$-factor as the statement requires.
In the $\alpha=\frac{2}{n+1} d$ case, we integrate $x^{-1}$ and get an additional logarithmic factor $\ln(s_t/r_t)$. 
In total, we thus get the upper bound $r_t^{dn}(1+\ln(s_t/r_t))\approx r_t^{dn} \ln(s_t/r_t)$, since $\ln(s_t/r_t)\gtrsim 1$ due to the assumption $s_t> 2r_t$.

Lastly, let us consider the case when $\alpha = \frac 2n d$; split the recurrence integral at $x=2r_t$ for simplicity. 
We then have an additional $\ln(x/r_t)$ term compared to the $\alpha>\frac 2n d$ case, in particular
\[I_{n+1,s_t}'' \approx r_t^{n\alpha + d(n-1)}\int_{2r_t}^{s_t} x^{d-n\alpha-1} \ln \l(\frac{x}{r_t}\r)\dd x.\]
Since we have $d-n\alpha=-d<0$, evaluating this integral gives
\begin{align*}
    I_{n+1,s_t}''
    & \approx r_t^{n\alpha+d(n-1)} \l[ x^{d-n\alpha}\l(1- \ln \l(\frac{x}{r_t}\r)\r)\r]_{2r_t}^{s_t} \\
    & \approx r_t^{dn} \l[1- \l(\frac{s_t}{r_t}\r)^{d-n\alpha} + \l(\frac{s_t}{r_t}\r)^{d-n\alpha} \ln \l(\frac{s_t}{r_t}\r)\r].
\end{align*}
Though the logarithmic factor $\ln (s_t/r_t)$ can tend to infinity as $t\to \infty$, the entire term will converge to a constant due to the exponent in the accompanying power function being negative. Consequently, the entire expression in parentheses is simply $\Theta(1)$, and $I_{n+1,s_t}''\approx r_t^{dn}$. 
Simply put, the logarithmic term that appears for $I_{n,s_t}$ if $\alpha=\frac 2n d$, becomes negligible in the induction for $I_{n+1,s_t}$.

As a final remark regarding the computation, observe how the exponents change through the induction: in each step, the exponent of $r_t^d$ increases by one (this comes from the factor $x^{d}$ in the recurrence), and the exponent of $r_t^\alpha$ by $n$ (coming from the $(g(x))^{n}$ factor in the integral). This is what essentially gives the general exponent for $I_{n,s_t}$ for any $n\geq 2$: $\sum_{i=2}^n 1 = n-1$ for $r_t^d$, and $\sum_{i=2}^n (i-1) = \sum_{i=1}^{n-1} i = n(n-1)/2$ for $r_t^\alpha$. This observation helps with the variance computation that follows below.

\paragraph{The variance and Poisson limit theorem.}  
By definition,
\[\EE G_{K_n}^2 = \frac{1}{(n!)^2} \sum_{(z_1,\ldots, z_n)\in \eta_{t,\neq}^n} \sum_{(z_1',\ldots, z_n')\in \eta_{t,\neq}^n} h(z_1,\ldots, z_n) h(z_1',\ldots, z_n').\]
Now, some of the $z_i$ and $z_j'$ might coincide.
By symmetry, only the number $m$ of points in $\{z_1,\ldots, z_n\}\cap \{z_1',\ldots, z_n'\}$ matters; in particular, with $m\in \{0,\ldots, n\}$ given, we may assume $z_{n-m+i} = z_{n-m+i}' \eqcolon y_i$ for $i=1,\ldots, m$. 
For $m=n$ (the two tuples coincide), the formula corresponds to the expectation, and for $m=0$ (all points distinct), to the expectation squared.
Hence
\[\VV G_{K_n}  = \EE G_{K_n} + \sum_{m=1}^{n-1} c_{n,m} \EE \sum_{\substack{(y_1,\ldots, y_m,\\z_1,\ldots, z_{n-m},\\z_1',\ldots, z_{n-m}')\\\in \eta_{t,\neq}^{2n-m}}} h(y_1,\ldots, y_m, z_1,\ldots, z_{n-m}) h(y_1,\ldots, y_m, z_1',\ldots, z_{n-m}'),\]
where $c_{n,m}>0$ is a combinatorial constant. We can then apply the Mecke formula to each summand, and get
\[\VV G_{K_n} =\EE G_{K_n} + \sum_{m=1}^{n-1} c_{n,m} t^{2n-m} \int_{W^{2n-m}} p_{n,m} \dd y_1 \ldots \dd y_{m}\dd z_1\ldots \dd z_{n-m} \dd z_1' \ldots \dd z_{n-m}'\]
with
\[p_{n,m} \coloneq \PP(h(y_1,\ldots, y_m, z_1,\ldots, z_{n-m}) h(y_1,\ldots, y_m, z_1',\ldots, z_{n-m}')=1).\]
Similarly to the expectation, we evaluate the above integral inductively. Let again $y_1 \coloneq \mathbf 0$, and for arbitrary $s>0$ define
\[I_{n,m,s} \coloneq \int_{(sB^d)^{2n-m-1}} p_{n,m} \prod_{(u,v)\in J_{n,m}} \ind(\|u-v\|\leq s)\dd y_1 \ldots \dd y_{m}\dd z_1\ldots \dd z_{n-m} \dd z_1' \ldots \dd z_{n-m}',\]
where $J_{n,m} = \{y_1,\ldots, y_m, z_1,\ldots, z_{n-m}\}^2 \cup \{y_1,\ldots, y_m, z_1',\ldots, z_{n-m}'\}^2$ contains the edges (pairs of points) in the two desired complete graphs.

We again establish a recurrence.
For the base case of the induction, let $m=1$ and $n$ arbitrary: we then simply have $I_{n,1,s_t} = I_{n,s_t}^2$, as the two complete graphs intersect only in a point (the fixed $y_1=\mathbf 0$), but not a full edge. 
Then, our recurrence goes from $(n,m)$ to $(n+1,m+1)$. Observe that the set of points for the latter includes one additional point, $y_{m+1}$, and we have
\[J_{n+1,m+1} = J_{n,m}\cup (\{y_{m+1}\}\times \{y_1,\ldots, y_m, z_1,\ldots, z_{n-m},z_1',\ldots, z_{n-m}'\}).\]
The main idea is the same as for the expectation: we can assume all necessary edges have length at most $2\|y_{m+1}\|$, and deduce that the connection probability of the points to $y_{m+1}$ is $\Theta_t (g(\|y_{m+1}\|))$.
This exactly handles the existence of the $2n-m$ edges in $J_{n+1,m+1}\setminus J_{n,m}$, and only the ones required for $J_{n,m}$ remain in the integral. Altogether, this gives
\[I_{n+1,m+1,s} \approx \int_0^{s_t} (g(x))^{2n-m} I_{n,m,x} x^{d-1} \dd x,\]
where $(g(x))^{2n-m}$ corresponds to the existence probability of the edges with $y_{m+1}$, and $I_{n,m,x}$ to the rest.
Similarly to the expectation, this yields the asymptotic behaviour:
for $s_t<2r_t$, $I_{n,m,s_t}\approx s_t^{a_{n,m}d}$, and for $s_t>2r_t$, 
\[I_{n,m,s_t} \approx
\begin{cases}
    r_t^{a_{n,m}d}, & \text{ if } \frac{\alpha}{d} > \frac{a_{n,m}}{b_{n,m}},\\
    r_t^{a_{n,m}d}\ln \l(\frac{2s_t}{r_t}\r), & \text{ if } \frac{\alpha}{d} = \frac{a_{n,m}}{b_{n,m}},\\
    r_t^{b_{n,m}\alpha} s_t^{a_{n,m}d-b_{n,m}\alpha}, & \text{ if }  \frac{\alpha}{d} < \frac{a_{n,m}}{b_{n,m}},
\end{cases}\]
with $a_{n,m}= 2n-m-1$ and $b_{n,m}=n(n-1)-\frac 12 m(m-1)$. 
Comparing the explicit formulas, we see that $V_{n,m}(t) = \Theta_t(t^{2n-m} I_{n,m,1})$, implying the variance claim of the theorem. 

We omit the explicit computation of these asymptotics, as they are, in essence, the same as the one detailed for the expectation; importantly, only the exponents need to be tracked. To do so, note that the recurrence among the exponents is $a_{n+1,m+1} = a_{n,m}+1$ and $b_{n+1,m+1} = b_{n,m}+(2n-m)$, which essentially comes from the two factors in the integral ($x^d$ to the power $1$, and $g(x)$ to the power $2n-m$).



In full generality, it is difficult to express cleanly which summand in the variance dominates. 
If we assume $\lim_{t\to \infty} \EE G_{K_n} <\infty$, the resulting upper bound on $r_t$ makes it possible (though still not simple) to compare the terms. 
We can in particular derive that $m=n-1$ is one of the dominant terms, (i.e. $V_{n,m}(t)=\c O_t(V_{n,n-1}(t))$ for all $m=1,\ldots, n-1$), and $V_{n,n-1}(t)\to 0$ as $t\to \infty$.

To prove the Poisson limit theorem, we use \cite[Theorem~3.8]{Pianoforte} as before in the proof of \Cref{thm:SRGG crossings}, adapted to edge-marked Poisson point processes.
By this result, if $Z$ is a Poisson distributed random variable with parameter $\EE G_{K_n}$ as in the theorem, $\dTV(G_{K_n},Z)\lesssim \VV G_{K_n}-\EE G_{K_n}$.
This gives the statement together with the variance bounds.   
\end{proof}

\subsection{Planarity of the SRGG}

Finally, we consider the asymptotic probability that the Pareto soft random geometric graph is planar.

\begin{theorem}[Planarity of the Pareto SRGG]
\label{thm:SRGGplanarity}
Let $\c G_t$ denote the soft random geometric graph with $\Pareto(r_t,\alpha)$ edge weights.
Then there exist constants $0<c^*\leq c^{**}<\infty$ such that the following hold.\\
If $\alpha > \frac 45 d$,
\begin{align*}
    \lim_{t\to \infty} \PP(\c G_t \text{ planar}) & =
        \begin{cases}
            1, & \text{ if } t r_t^{\frac 45d} \to 0,\\
            0, & \text{ if } t r_t^{\frac 45d} \to \infty.
        \end{cases}
    \intertext{Additionally, if $\EE G_{K_5}\to \gamma \in (0,\infty)$, i.e. the expected number of subgraphs isomorphic to $K_5$ converges to a non-zero constant, then $\lim_{t\to \infty}\PP(\c G_t \text{ planar})= e^{-\gamma}$. 
    \newline 
    If $\alpha=\frac 45 d$,}
    \lim_{t\to \infty} \PP(\c G_t \text{ planar}) & = 
        \begin{cases}
            1, & \text{ if } t r_t^{\frac 45d} =t r_t^\alpha \to 0,\\
            0, & \text{ if } t r_t^{\frac 45d} = t r_t^\alpha > c^{**} \text{ for $t$ sufficiently large}.
        \end{cases}\\
    \intertext{If $\alpha<\frac 45 d$,}
    \lim_{t\to \infty} \PP(\c G_t \text{ planar}) &= 
        \begin{cases}
            1, & \text{ if } t r_t^{\alpha} < c^{*\phantom{*}}\text{ for $t$ sufficiently large},\\
            0, & \text{ if } t r_t^{\alpha} > c^{**} \text{ for $t$ sufficiently large}.
        \end{cases}
\end{align*}
\end{theorem}
\begin{remark}
    Recall that $G_{K_5}$ denotes the number of subgraphs isomorphic to $K_5$ in $\c G_t$. If $\alpha>\frac 45d$, we will show $\PP(\c G_t \text{ planar}) \sim \PP(G_{K_5} =0)$, and in particular, planarity is purely decided by the non-existence of a $K_5$ in $\c G_t$, similarly to the hard random geometric graph (\Cref{thm:RGGplanarity}).
    For $\alpha> \frac 45d$, this is not the case any more. Rather, the smallest order non-planar subgraph in $\c G_t$ might actually consist of many vertices. 
    Formally, one can show that if $\lim_{t\to \infty} tr_t^\alpha<\infty$, the probability that a non-planar graph of order $n$ appears in $\c G_t$ converges to $0$ for any fixed $n$, and so the size of the smallest non-planar subgraph grows with $t$.
    This is also the reason why the thresholds stated for $r_t$ are not sharp: finding precise existence results for high-order subgraphs appears to be difficult. 
    Whether there exists a sharp phase transition is unclear.
\end{remark}

\begin{proof}
As there is some overlap in the statements of the different $\alpha$ cases, it is sufficient to show that there exist constants $0<c^*\leq c^{**} <\infty$ satisfying the following:
\begin{itemize}
    \item[i)] If $tr_t^{\frac 45d} \to \infty$ or $tr_t^\alpha>c^{**}$, then $\lim_{t\to \infty} \PP(\c G_t \text{ planar})=0$.
    \item[ii)] If $tr_t^{\frac 45d} \to \lambda \in [0,\infty)$ and $tr_t^\alpha<c^{*}$,
     \begin{align*}
       \text{ (a) } & \lim_{t\to \infty} \PP(G_{K_5} = 0) = 
        \begin{cases}
        e^{-\gamma},& \caseif \lambda > 0 \text{ and }\EE G_{K_5}\to \gamma, \\
        1,&  \caseif \lambda = 0;
        \end{cases} \\
        \text{ (b) }& \lim_{t\to \infty} \PP(\c G_t \text{ planar} \mid  G_{K_5}=0) = 1.
    \end{align*}
\end{itemize}
Indeed, these assertions together yield the theorem. 
Firstly, i) yields the last, non-planarity statement of the theorem for each $\alpha$.
Secondly, for $\alpha > \frac 45d$ and $\lim_{t\to \infty} tr_t^{\frac 45d}<\infty$, we have $tr_t^{\alpha}\to 0$, hence $\PP(\c G_t \text{ planar})\sim \PP(G_{K_5}=0)$ as $t\to \infty$ by ii)(b), which yields the first statement of the $\alpha >\frac 4d d$ case by ii)(a), as well as the claim for $\EE G_{K_5}\to \gamma\in (0,\infty)$. 
Lastly, if $\alpha < \frac 45 d$, then $tr_t^\alpha < c^{*}$ implies $tr_t^{\frac 45d}\to 0$, and planarity follows from ii) w.h.p. The threshold case $\alpha = \frac 45d$ can be obtained the same way. 


\subsubsection*{Proof of i) (non-planarity)}

If $tr_t^{\frac45 d}\to \infty$, the statement already follows from the fact that the hard random geometric graph (which is a subgraph of the soft model considered here) is non-planar with high probability (according to \Cref{thm:RGGplanarity}), in particular from $\PP(G_{K_5}\geq 1)\to 1$. Alternatively, one could also use the results for the expectation and variance given in \Cref{thm:SRGGcliques} to show the same statement directly.

Now, we show that there exists $c^{**}$ such that if $tr_t^\alpha > c^{**}$, then $\c G_t$ is non-planar with high probability; we may assume $\alpha < \frac 45d$, as the larger $\alpha$ cases are already solved by the previous argument. 
To prove this statement, we use a standard result from graph theory (see e.g. \cite[Corollary~4.2.8]{Diestel2017}), which states that if a graph with $n$ vertices and $m$ edges is planar, the inequality $m\leq 3n$ must hold.
In our case, writing $f_0=|\eta_t|$ for the number of vertices and $f_1 = G_{K_2}$ for the number of edges, we have that if $\c G_t$ is planar, then $N_t\coloneq 3f_0-f_1\geq 0$ must hold. 
Since the expected number of edges is $\EE f_1 = \Theta(t^2 r_t^\alpha)$ by \Cref{thm:SRGGcliques}, there exists a constant $c^{**}>0$ such that if $tr_t^\alpha\geq c^{**}$, then $\EE f_1>4t$ for large enough $t$, implying that $\EE N_t < -t<0$.
Thus, if $tr_t^\alpha \geq c^{**}$, we obtain the bound
\[\PP(\c G_t \text{ planar})\leq \PP(N_t\geq 0) = \PP(N_t-\EE N_t\geq -\EE N_t)  \leq \frac{\VV N_t}{(\EE N_t)^2}.\]
Clearly, $\VV f_0 = t$, and from \Cref{thm:SRGGcliques}, we also automatically have $(\EE N_t)^2  = \Theta_t( \max\{t^2,t^4 r_t^{2\alpha}\}) = \Theta_t(t^4 r_t^{2\alpha})$ and $\VV f_1 = \Theta_t(t^2 r_t^\alpha)+ \Theta_t(t^3r_t^{2\alpha}) = \Theta_t(t^3 r_t^{2\alpha})$. 
Consequently,
\[\PP(\c G_t \text{ planar})\leq\frac{9\VV f_0 + 6\sqrt{\VV f_0 \VV f_1} +\VV f_1}{(\EE N_t)^2}\lesssim \frac{t^3 r_t^{2\alpha}}{t^4 r_t^{2\alpha}} = \frac 1t \longrightarrow 0\]
as $t\to \infty$.

\subsubsection*{Proof of ii) (Poisson $K_5$-s and planarity)}

Let $tr_t^{\frac 45d} \to \lambda<\infty$ and $tr_t^\alpha<c^{*}$. 
By \Cref{thm:SRGGcliques}, this implies $\gamma\coloneq \limsup_{t\to \infty} \EE G_{K_5}<\infty$, and in particular $\gamma = 0$ exactly when $\lambda = 0$.
The latter automatically yields by Markov's inequality that if $\lambda =0$, then $\PP(G_{K_5}\geq 1)\leq \EE G_{K_5}\to \gamma=0$. 
If $\lambda>0$ and $\EE G_{K_5}$ is convergent, i.e. $\EE G_{K_5}\to \gamma\in (0,\infty)$, then we must have $\alpha \geq \frac 45 d$, and \Cref{thm:SRGGcliques} states that $G_{K_5}$ converges in distribution to a Poisson random variable with parameter $\gamma$, thus $\PP(G_{K_5}=0) \to e^{-\gamma}$ follows. 
Note that if  $t r_t^{\frac 45 d} \to \lambda\in (0,\infty)$, then by \Cref{thm:SRGGcliques}, the upper and lower bound for the expectation are of the same order and converge, but we a priori do not know whether the expectation itself converges, or fluctuates within this range. 

For the rest of the proof, we consider the statement ii)(b), which is equivalent to saying that under the given upper bounds for $r_t$, no subdivision of $K_{3,3}$, or proper subdivision of $K_5$ appears in the graph with high probability (where by proper subdivision we mean one with at least one additional vertex). 

Start with the case $\alpha\geq d$. We show that with high probability, no path of length $5$ appears in the graph (which implies no $K_{3,3}$, or proper subdivision of $K_5$). 
The expected number of paths of length $5$ is given, up to a constant, by
\[ I\coloneq t^6 \int_{W^6} \PP(y_1\con \ldots \con y_6) \dd y_1 \ldots \dd y_6 = t^6 \int_{W^6} g(\|y_1-y_2\|)\cdot \ldots \cdot g(\|y_5-y_6\|) \dd y_1\ldots \dd y_6,\]
where $g(x)$ denotes the connection probability of two points at distance $x$.
Setting $D\coloneq \diam(W)$ to be the diameter of $W$, we can obtain an upper bound by extending the domain of integration of $y_2$ from $W$ to $y_1+DB^d\supset W $, of $y_3$ from $W$ to $y_2+2DB^d\supset y_1+DB^d\supset W$, and generally $y_i$ from $W$ to $y_1+(i-1)D B^d$ for $i=2,\ldots, 6$. Applying a polar transform to the resulting integral, in particular writing $y_i = y_{i-1}+x_i u_i$ for $u_i\in S^{d-1}$ and $0< x_i \leq (i-1)D$, this yields
\[I \leq t^6 \int_W \prod_{i=1}^5 \l[\int_0^{(i-1)D} \int_{S^{d-1}} g(x_i) x_i^{d-1}\dd u_i \dd x_i\r]\dd y_1\lesssim t^6\l[\int_0^{5D} g(x) x^{d-1} \dd x\r]^5,\]
where $x_i^{d-1}$ is the Jacobian of the polar transformation. The last integral can be explicitly evaluated, and we have
\[I \lesssim 
\begin{cases}
    t^6 r_t^{5d},& \caseif \alpha>d \\
    t^6 r_t^{5d} \l(\ln\l(\frac{1}{r_t}\r)\r)^5, & \caseif \alpha = d.
\end{cases}\]
Under the given conditions on the regime, $t^6 r_t^{5d}\lesssim t^{-1/4}$ and $\ln (1/r_t)\lesssim \ln(t)$.
This yields that both above expressions tend to zero, thus $\c G_t$ contains no path of length $5$ with high probability, proving the statement of planarity.

For the rest of the proof, assume that $\alpha<d$. In this setting, paths or cycles with many edges may appear in the graph, and we need a stronger argument to handle planarity. However, general subgraph counts are difficult to determine. To avoid this, we use the following observation: any proper subdivision of $K_5$, or subdivision of $K_{3,3}$ (including $K_{3,3}$ itself), contains a subgraph of at least $6$ vertices where two of the points are connected via three independent paths (see \Cref{fig:subdivisions}). 

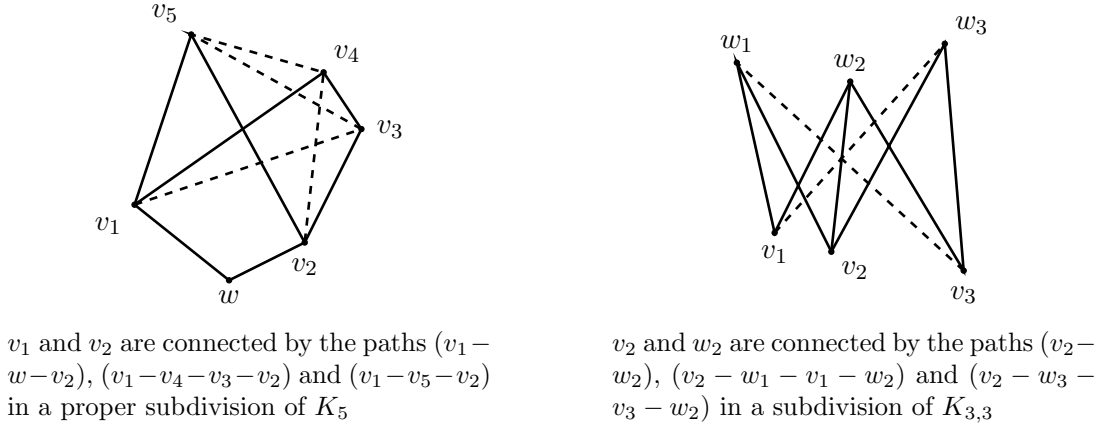
\begin{figure}[ht]
    \centering
    \hfill
    \begin{subfigure}{0.4\textwidth} 
    \centering
    \begin{tikzpicture}[scale=0.5]
        \coordinate (v1) at (0,0);
        \draw [fill=black] (v1) circle (2pt);
        \node at (v1) [below left = 1pt]{$v_1$};
        
        \coordinate (v2) at (4.5,-1);
        \draw [fill=black] (v2) circle (2pt);
        \node at (v2) [below = 2pt] {$v_2$};
        
        \coordinate (v3) at (6,2);
        \draw [fill=black] (v3) circle (2pt);
        \node at (v3) [right = 2pt] {$v_3$};
        
        \coordinate (v4) at (5,3.5);
        \draw [fill=black] (v4) circle (2pt);
        \node at (v4) [above right] {$v_4$};
        
        \coordinate (v5) at (1.5,4.5);
        \draw [fill=black] (v5) circle (2pt);
        \node at (v5) [above left = 1pt] {$v_5$};
        
        \coordinate (w) at (2.5,-2);
        \draw [fill=black] (w) circle (2pt);
        \node at (w) [below] {$w$};

        \draw [line width = 1pt] (v1)--(w)--(v2);
        \draw [line width = 1pt] (v1)--(v4)--(v3)--(v2);
        \draw [line width = 1pt] (v1)--(v5)--(v2);
        \draw [line width = 1pt, dashed] (v1)--(v3)--(v5)--(v4)--(v2);
    \end{tikzpicture}
    
    \caption{$v_1$ and $v_2$ are connected by the paths $(v_1-w-v_2)$, $(v_1-v_4-v_3-v_2)$ and $(v_1-v_5-v_2)$ in a proper subdivision of $K_5$}

    \end{subfigure}
    \hfill
    \begin{subfigure}{0.4\textwidth} 
    \centering
    \begin{tikzpicture}[scale=0.5]
    \coordinate (v1) at (0,0);
    \draw [fill=black] (v1) circle (2pt);
    \node at (v1) [below = 1pt]{$v_1$};
    
    \coordinate (v2) at (1.5,-0.5);
    \draw [fill=black] (v2) circle (2pt);
    \node at (v2) [below right] {$v_2$};
    
    \coordinate (v3) at (5,-1);
    \draw [fill=black] (v3) circle (2pt);
    \node at (v3) [below = 2pt] {$v_3$};        
    
    \coordinate (w1) at (-1,4.5);
    \draw [fill=black] (w1) circle (2pt);
    \node at (w1) [above] {$w_1$};

    \coordinate (w2) at (2,4);
    \draw [fill=black] (w2) circle (2pt);
    \node at (w2) [above] {$w_2$};
    
    \coordinate (w3) at (4.5,5.);
    \draw [fill=black] (w3) circle (2pt);
    \node at (w3) [above right] {$w_3$};
    
    \draw [line width = 1pt] (v2)--(w2);
    \draw [line width = 1pt] (v2)--(w1)--(v1)--(w2);
    \draw [line width = 1pt] (v2)--(w3)--(v3)--(w2);
    \draw [line width = 1pt, dashed] (v1)--(w3);
    \draw [line width = 1pt, dashed] (v3)--(w1);
    \end{tikzpicture}    
    \caption{$v_2$ and $w_2$ are connected by the paths $(v_2-w_2)$, $(v_2-w_1-v_1-w_2)$ and $(v_2-w_3-v_3-w_2)$ in a subdivision of $K_{3,3}$}
    \end{subfigure}
    \hfill
    \caption{Subgraphs induced by subdivisions of $K_5$ and $K_{3,3}$}
    \label{fig:subdivisions}
\end{figure}

As a consequence, it is sufficient for us to show that such subgraphs do not occur in $\c G_t$ with high probability; in some sense, these are the smallest tractable graphs in this model whose non-existence guarantees planarity.

For $1\leq n_1\leq n_2\leq n_3$, denote by $\c H_{n_1,n_2,n_3}$ the number of subgraphs of $\c G_t$ isomorphic to the graph on $n-1\coloneq n_1+n_2+n_3-1$ vertices, where two of the points are connected by three disjoint paths of length $n_1$, $n_2$ and $n_3$. 
By the previous argument, we will assume that $n-1\geq 6$, and also $n_1+n_2\geq 4$; the latter can be observed on \Cref{fig:subdivisions}. 
Formally, it is then sufficient to show that $\sum_{n_1,n_2,n_3} \c H_{n_1,n_2,n_3}$ is zero with high probability, which we show by proving that the expectation converges to zero.

By the Mecke formula,
\begin{align*}
    \EE \c H_{n_1,n_2,n_3} 
    &\leq 12 
    \begin{multlined}[t]
    t^{n-1} 
    \int_{W^{n-1}} 
    \PP(y_1\con u_2 \con \ldots \con u_{n_1} \con y_2) 
    \PP(y_1\con v_2 \con \ldots \con v_{n_2} \con y_2) \\
    \times \PP(y_1\con w_2 \con \ldots \con w_{n_3} \con y_2) 
    \dd u_2\ldots \dd u_{n_1} \dd v_2 \ldots \dd v_{n_2} \dd w_2 \ldots \dd w_{n_3} \dd y_1 \dd y_2,
    \end{multlined}
\end{align*}
where the factor $12$ comes from symmetry: counting the pair $(y_1,y_2)$ in both orders, and the paths in at most $6$ possible orders (if all path lengths are the same).
We evaluate this as follows.
For $0<s\leq D = \diam (W)$ and $n\in \N$, set $y_0\coloneq \mathbf 0 \coloneq (0,\ldots, 0)\in \R^d$, $y_n=y_n(s)\coloneq(s,0,\ldots,0)\in \R^d$, and write 
\[I_n(s) \coloneq  \int_{(\R^d)^{n-1}} \PP( y_0 \con y_1 \con\ldots \con y_n(s)) \ind (\|y_i-y_j \|\leq D \ \forall i,j=0,\ldots, n) \dd y_2\ldots \dd y_{n-1}\]
for $n\geq 2$, and $I_1(s)=\PP(y_0\con y_n(s))$ for $n=1$. 
Roughly, this corresponds to the probability that two points at distance $s$ are connected by a path of length $n$.
Note that while the two anchor points have distance $s$, the edge lengths within the path can be as long as containment in $W$ in the original integral would allow: this is represented by the maximal length condition of $D$.
Using this notation, we have
\begin{align}\label{eq:chord}
\begin{split}
    \EE \c H_{n_1,n_2,n_3} 
    & \leq 12 t^{n-1} \int_{W^2} I_{n_1}(\| y_1- y_2\|) I_{n_2}(\| y_1- y_2\|) I_{n_3}(\| y_1- y_2\|) \dd y_1 \dd y_2\\
    & \leq 12 \lambda_{d-1}(S^{d-1})t^{n-1} \int_0^{D} I_{n_1}(s) I_{n_2}(s) I_{n_3}(s) s^{d-1} \dd s,
\end{split}
\end{align}
where in the second step $x_2$ is transformed into polar coordinates with respect to $x_1$, and its directional component, as well as $x_1$, is integrated out.


\begin{claim}\label{claim:path}
For fixed $\alpha<d$, 
there exist $c', c''>0$ depending only on $\alpha$, $d$ and $W$ such that 
$I_n(s)$ is upper bounded by the following quantities:
    \begin{center}
    \begin{tabular}{c||C{4.5cm}|C{6cm}}
     & $s_t \leq 2r_t$ & $s_t > 2r_t$\\
    \hline
    \hline
    $n <\frac{d}{d-\alpha} $ 
    & $c'r_t^{(n-1)d}$ 
    & $c'r_t^{n\alpha} s^{(n-1)d-n\alpha}$ \\
    \hline
    $n =\frac{d}{d-\alpha} $ 
    & $c' (c'' r_t^{\alpha})^n \ln \l(\frac{1}{r_t}\r)$ 
    & $c' (c'' r_t^{\alpha})^n \ln \l(\frac{1}{s}\r)$ \\
    \hline
    $ n >\frac{d}{d-\alpha}$ & \multicolumn{2}{c}{$c'(c''r_t^\alpha)^n$} 
    \end{tabular}
    \end{center}
\end{claim}

\begin{proof}
Denote the expressions in the table by $I_n^+(s)$; observe that for any $n$, this is a monotone decreasing function in $s$.

We show the upper bound using induction on $n$. 
Since $I_1(s)=\PP(y_0\con y_1(s)) = g(s) = \min\l\{1, r_t^\alpha s^{-\alpha}\r\}$ by definition, equality trivially follows for $n = 1 < \frac{d}{d-\alpha}$. 
For the inductive step, we claim the upper bound
\begin{equation*}
    I_{n+1}(s) \leq \lambda_{d-1}(S^{d-1})
    \l[\int_{s}^{D} g(x) I_{n}^+\l(\frac x2\r) x^{d-1}\dd x +\int_0^{s} \l( g(x) I_{n}^+\l(\frac s2\r) + g\l(\frac s2\r) I_{n}^+(x)\r) x^{d-1}\dd x\r]
\end{equation*}
holds for all $n\geq 1$. 
We show this as follows.
By the independence of edges, we have
\begin{align*}
    I_{n+1}(s) & = \int_{(\R^d)^{n}} \PP(y_0\con y_1) 
 \begin{multlined}[t] \PP(y_1\con y_2\con \ldots \con y_{n+1}(s)) \\ 
    \times \ind (\|y_i-y_j\|\leq D \ \forall i,j=0,\ldots, n+1)  \dd y_2\ldots y_{n} \dd y_1 
    \end{multlined}\\
    & = \int_{\R^d} g(\|y_0-y_1\|)I_{n}(\|y_1- y_{n+1}(s)\|)\ind (\|y_0-y_1\|, \|y_1-y_{n+1}(s)\|\leq D )\dd y_1.
\end{align*}
We get an upper bound by writing $I_n^+$ instead of $I_n$ in the integral.
Consider now the possible distance configurations: recall that $\|y_0-y_{n+1}(s)\| = s$ by construction, and write $x_-$ and $x^+$ for the minimum and maximum of $\|y_0-y_1\|, \|y_1-y_{n+1}(s)\|$, respectively.

First, if $x_-\geq s$, then by the triangle inequality, $x_-\geq x^+/2$.
As a consequence, if we apply a polar transform to $y_1$ with respect to $y_0$, (i.e. write $y_1=y_0+x u$ with $u\in S^{d-1}$ and $x\in [s/2,D]$), 
then the integrand is at most $g(x)I_{n}^+(x/2)$ by the monotonicity of $I_{n}^+$.

On the other hand, if $x_- < s$, we use that $x^+\geq s/2$ by the triangle inequality.
If $x_- = \|y_0-y_1\|$, write $y_1=y_0+xu$ with $x\in [0,s]$ and obtain the upper bound $g(x)I_{n}^+(s/2)$; if $x_- = \|y_1-y_{n+1}(s)\|$, let $y_1 = y_{n+1}+xu$ to get $g(s/2)I_{n}^+(x)$. This gives the second two terms of the expression, concluding the proof of the recurrence.

This geometrically obtained expression can be somewhat simplified, by observing first that changing the argument of $I_n^+$ changes the expression only by a constant factor, and $g(x)I_n(s)\gtrsim g(s) I_n(x)$ for $x\leq s$ by explicitly comparing their definitions; hence it is essentially sufficient to consider the first term of the second integral. 
Secondly, note that for any fixed $\alpha<d$, there are only finitely many integers $n$ below $\frac{d}{d-\alpha}$.
Consequently, we can choose $c_0$ such that 
\[I_{n+1}(s)\leq c_0 \l[ \int_{s}^D g(x)I_n^+(x)x^{d-1} \dd x  + I_n^+(s) \int_0^{s} g(x)x^{d-1} \dd x \r]\]
holds for all $n\leq \frac{d}{d-\alpha}$. 
Using this recurrence, \Cref{claim:path} then follows (for such $n$) by a standard integral computation; we omit the details, and note that a similar calculation was carried out in the proof of \Cref{thm:SRGGcliques}. 
This is applicable up to $n_0$ with $n_0\leq \frac{d}{d-\alpha}<n_0+1$, in which case the integral is of the order $r_t^{\alpha (n_0+1)}$.

For $n>\frac{d}{d-\alpha}$, the upper bound $I_n^+(s)\equiv I_n^+$ does not depend on $s$, and so we have
\begin{align*}
    I_{n+1}(s)  
    & \leq \int_{2s}^D g(x)I_n^+ x^{d-1}\dd x  + \int_{0}^{2s} (g(x)I_n^+ + g(x)4^\alpha I_n^+)x^{d-1}\dd x \\
    & \leq I_n^+ (4^\alpha+1)\int_{0}^D g(x)x^{d-1}\dd x.
\end{align*}
Since the last integral expression does not depend on $n$, and is of the order $r_t^\alpha$ (using that $\alpha<d$), we may choose $c''$ such that $I_{n+1}(s)\leq I_n^+  c'' r_t^\alpha$ holds for all $n > \frac{d}{d-\alpha}$. This is sufficient to conclude the proof of the claim.
\end{proof}

We return to finding an upper bound for $\EE \c H_{n_1,n_2,n_3}$.
Choose $k,l\in\{0,1,2,3\}$ such that 
\[1\leq n_1\leq \ldots \leq n_k < \frac{d}{d-\alpha} = n_{k+1} = \ldots = n_l < n_{l+1} \leq \ldots \leq n_3.\]
In other words, let us have $k$ indices below $\frac{d}{d-\alpha}$, $l-k$ equal to it, and $3-l$ above. 
The integral in \eqref{eq:chord} is then upper bounded by 
\[\Bigg [
    \begin{multlined}[t]
    \int_0^{2r_t} 
    \prod_{i=1}^k \bigg(c' r_t^{(n_i-1)d}\bigg) \prod_{i=k+1}^l \bigg( c'( c''r_t^{\alpha})^{n_i}\ln(1/r_t)\bigg) s^{d-1} \dd s 
    \\ +
    \int_{2r_t}^{D} \prod_{i=1}^k \bigg(c' r_t^{n_i \alpha} s^{(n_i-1)d-n_i\alpha} \bigg) \prod_{i=k+1}^l \bigg( c'( c''r_t^{\alpha})^{n_i}\ln(1/s)\bigg) s^{d-1} \dd s 
    \Bigg]
    \prod_{i=l+1}^3 \bigg(c' (c''r_t^\alpha)^{n_i}\bigg)
    \end{multlined}\]
according to \Cref{claim:path}.
For simplicity, we can use the bound $\ln(1/s)\leq \ln(1/r_t)$ for $s> 2r_t$. Setting $N\coloneq n_1+ \ldots + n_k$ and simplifying, this yields
\begin{multline}\label{eq:chord bound}
    \EE \c H_{n_1,n_2,n_3}\leq 12 \lambda_{d-1}(S^{d-1}) t^{n-1} \\
    \times (c')^3 \l[\int_0^{2r_t}  r_t^{(N-k)d}  s^{d-1}\dd s +   \int_{2r_t}^{D} r_t^{N\alpha}  s^{(N-k+1)d - N\alpha-1} \dd s\r] 
    (c''r_t^\alpha)^{n-N} \l(\ln\l(1/r_t\r)\r)^{l-k}.
\end{multline}
The order of the second integral depends on the sign of $m\coloneq N(d-\alpha)-(k-1)d$. 
For non-positive values, the sum of integrals is at most of the order $r_t^{(N-k+1)d} \ln(1/r_t)$ (with the logarithmic term only appearing for $m=0$).
As a consequence, 
\[\EE \c H_{n_1,n_2,n_3}\lesssim t^{-1} (tr_t^\alpha)^{n} r_t^{N(d-\alpha) - d(k-1)} \l(\ln\l(1/r_t\r)\r)^4\] 
whenever $m\leq 0$. 
Observe that since $k\leq 2$ and $d-\alpha>0$ is fixed, there are only finitely many triples $(n_1,n_2,n_3)$ with $m\leq 0$, thus we can disregard the constants.

Recall that $n_1\geq 1$, $n_1+n_2\geq 4$ and $n_1+n_2+n_3=n\geq 7$ by construction, and that $d-\alpha>0$. 
This yields that 
\[m\geq \min\{d,d-\alpha,3d-4\alpha,5d-7\alpha\} = \min\{d-\alpha,5d-7\alpha\}\]
for any triple $(n_1,n_2,n_3)$, where the expressions come from each potential value of $k$. This expression can then only be non-positive if $\alpha\geq\frac 57d$, in which case the minimum is exactly $5d-7\alpha \leq 0$.
The assumption $tr_t^{\frac 45 d} \lesssim 1$ of the theorem (or, more precisely, that of ii) at the beginning of the proof) yields $r_t\lesssim t^{-\frac {5}{4d}}$ and $tr_t^\alpha \lesssim t^{1-\frac{5\alpha}{4d}}$, and so for tuples satisfying $m\leq 0$, we have
\begin{align*}
    \EE \c H_{n_1,n_2,n_3}
    & \lesssim t^{-1} (tr_t^\alpha)^7 r_t^{\min\{d-\alpha,5d-7\alpha\}}\l(\ln(1/r_t)\r)^4 \\
    & \lesssim t^{-1+7(1-\frac{5\alpha}{4d})-\frac{5}{4d}(5d-7\alpha)} (\ln(t))^4
    = t^{-\frac 14} (\ln(t))^4.
\end{align*}
This converges to zero as $t\to \infty$, hence such graphs do not appear in the graph with high probability.

Lastly, consider triples with $m>0$, where now we have infinitely many such tuples.
One can see that the sum of integrals in \eqref{eq:chord bound} is of the order $r_t^{N\alpha}$, with the implied constant being at most exponential in $n$. This implies that there exists $c_1,c_2>0$ such that
\[\EE \c H_{n_1,n_2,n_3}\leq c_1 t^{-1}  (\ln(1/r_t))^4 (c_2 tr_t^\alpha)^n\]
for every triple satisfying $N(d-\alpha)-(k-1)d>0$.
To account for every tuple, observe that for any $n$, there are at most $n^2$ positive triples that sum up to $n$, hence
\[\sum_{N(d-\alpha)-(k-1)d>0} \EE \c H_{n_1,n_2,n_3} \lesssim t^{-1}(\ln(t))^4\cdot  \sum_{n=7}^\infty n^2 (c_2 t r_t^\alpha)^n.\]
Choosing $r_t$ such that $\lim_{t\to \infty} c_2 tr_t^\alpha<1$, the series converges and can be bounded by a constant, and convergence to zero follows.

\end{proof}

\section{The max-kernel vertex marked random geometric graph}\label{sec:RR}
{In this section, let $\hat \eta_t$ denote a $\Pareto(r_t,\alpha)$ marked Poisson point process on the convex body $W$: that is, let $\hat \eta_t$ be a Poisson point process on $\R^d \times [0,\infty)$ with intensity measure $t\lambda_d|_W\otimes \mu_{\alpha}$, where $\mu_{\alpha}$ has Lebesgue density given by $f_{\alpha}(x)\coloneq \alpha r_t^\alpha x^{-\alpha-1} \ind(x\geq r_t)$.

We use the following notation for marked point configurations and point processes. For a marked configuration $\hat \zeta\subset \R^d \times [0,\infty)$, write $\zeta$ for the corresponding spatial point configuration $\zeta(\cdot)\coloneq \hat \zeta(\cdot \times [0,\infty))$. 
In addition, for $\hat y \in \hat \zeta$, let $\hat y = (y,R_y)$, that is, denote by $y\in \R^d$ its spatial component, and by $R_y\in [0,\infty)$ the corresponding mark. 
Lastly, throughout the section, we write $\c G_t(\hat \zeta)$ for the graph with vertex set $\zeta$ and edges given by pairs $(y, z)\in \zeta_{\neq}^2$ with $\|y-z\| \leq \max \{R_y,R_z\}$.}

\subsection{Planarity of the max-kernel RGG}

For a graph $\Gamma$ and marked point configuration $\hat \zeta$, write $G_{\Gamma}(\hat \zeta)$ for the number of subgraphs isomorphic to $\Gamma$ in $\c G_t(\hat \zeta)$.
\begin{theorem}[Planarity of the max-kernel marked RGG]\label{thm:RRplanarity}
Let $\hat \eta_t$ be a $\Pareto(r_t,\alpha)$ marked Poisson point process in the convex body $W$, and $\c G_t(\hat \eta_t)$ the random geometric graph as described above.
Then
    \[\lim_{t\to \infty} \PP(\c G_t(\hat \eta_t) \text{ planar}) = 
        \begin{cases}
            1, & \text{if } L_t\to 0,\\            
            0, & \text{if } L_t\to \infty,\\
        \end{cases}\]
    with
    \[L_t \coloneq 
        \begin{cases}
            t (tr_t^d)^{\frac{\alpha}{d} + \ldots + \frac{\alpha^{3}}{d^{3}}},& \text{ if } \alpha< d ,\\
            t (tr_t^d)^{3} \l(\ln \l(\frac{1}{tr_t^d}\r)\r)^2 & \text{ if } \alpha=d,\\
            t (tr_t^d)^{\frac{3\alpha}{d}}, & \text{ if }  d < \alpha \leq \frac 43 d,\\
            t (tr_t^d)^{4}, & \text{ if } \alpha \geq \frac 43 d.\\
        \end{cases}\]
    Further, if $L_t\to \infty$ as $t\to \infty$, then $\PP(G_{K_{3,3}}(\hat \eta_t)\geq 1) \to 1$ if $\alpha \leq \frac 43 d$, and $\PP(G_{K_5}(\hat \eta_t)\geq 1) \to 1$ if $\alpha\geq \frac 43 d$.
\end{theorem}
\begin{remark}\label{remark:RR planarity}
    Additionally to the fact that a $K_{3,3}$ or a $K_5$ appears in the graph with high probability (depending on $\alpha$) whenever $L_t\to \infty$, we can actually show that if the convergence to infinity is slow, these are the only non-planar subgraphs. More precisely, there is a function $L_t^{(u)}$ such that if $L_t\to \infty$, but $L_t = O(L_t^{(u)})$, then no proper subdivision of $K_{3,3}$ or $K_5$ appears, nor a $K_5$ (if $\alpha<\frac 43d$) or $K_{3,3}$ (if $\alpha>\frac 43d$).
\end{remark}
The proof of the theorem is based on explicitly finding the existence probability of arbitrary subdivisions of $K_5$ and $K_{3,3}$. 

\begin{theorem}[Complete subgraphs in the max-kernel marked RGG]\label{thm:RR subgraph}
Let $\hat \eta_t$ be a $\Pareto(r_t,\alpha)$ marked Poisson point process in the convex body $W$. Then for $\Gamma=K_n$ or $\Gamma=K_{n,m}$ for any $1\leq n \leq m$,
    \[\lim_{t\to \infty} \PP(G_{\Gamma}(\hat \eta_t)\geq 1) =
        \begin{cases}
            0, & \caseif M_{t,\alpha}(\Gamma) \to 0,\\
            1, & \caseif M_{t,\alpha}(\Gamma) \to \infty
        \end{cases}
        \]
    holds, where    
    \begin{align*}
        M_{t,\alpha}(K_n) & = 
        \begin{cases}
            t (tr_t^d)^{\frac{\alpha}{d} + \ldots + \frac{\alpha^{n-1}}{d^{n-1}}}, & \caseif \alpha<d,\\
            t (tr_t^d)^{n-1} \l(\ln \l(\frac{1}{tr_t^d}\r)\r)^{n-1}, & \caseif \alpha = d,\\
            t (tr_t^d)^{n-1}, & \caseif \alpha>d,
        \end{cases}\\
    \intertext{and for $m\geq 2$,}
    M_{t,\alpha} (K_{n,m})&= 
        \begin{cases}
           t (tr_t^d)^{\frac{\alpha}{d} + \ldots + \frac{\alpha^{n}}{d^{n}}}, & \caseif \alpha<d,\\
           t (tr_t^d)^{n} \l(\ln \l(\frac{1}{tr_t^d}\r)\r)^{n-1}, & \caseif \alpha = d,\\
           t (tr_t^d)^{\frac{n\alpha}{d}}, & \caseif d < \alpha < d \cdot \frac{n+m-1}{n}, \\
           t (tr_t^d)^{n+m-1} \ln \l(\frac{1}{tr_t^d}\r), & \caseif \alpha = d\cdot \frac{n+m-1}{n},\\
           t (tr_t^d)^{n+m-1}, & \caseif \alpha>d\cdot \frac{n+m-1}{n}.
        \end{cases}
    \end{align*}
\end{theorem}
Note that for $n=m=1$, we have $M_{t,\alpha}(K_{1,1})=M_{t,\alpha}(K_2)$; it does not exactly fit into the general definition for $K_{n,m}$, because the cases in $\alpha$ collapse.

 
In addition, we have to control the proper subdivisions of $K_5$ and $K_{3,3}$.

\begin{lemma}\label{lemma:RR subdivisions} 
Let $P_n$ denote the path of length $2^n-2$ for some $n\geq 1$, in particular $P_1$ the one-vertex graph. Then $\EE G_{P_n}(\hat \eta_t)\lesssim M_{t,\alpha}(K_n)$. 
Further, if $\Gamma'$ is a proper subdivision of $\Gamma = K_n$ or $\Gamma= K_{n,m}$, then $\EE G_{\Gamma'}(\hat \eta_t) \lesssim M_{t,\alpha}(\Gamma)$.
\end{lemma}

\begin{proof}[Proof of \Cref{thm:RRplanarity}]
First, observe that 
\[L_t = 
    \begin{cases}
        M_{t,\alpha}(K_{3,3}),& \caseif \alpha\leq \frac 43d,\\
        M_{t,\alpha}(K_5),& \caseif \alpha \geq \frac 43d,
    \end{cases}\]
with $\alpha=\frac 43d$ indeed having $M_{t,\alpha}(K_5)=M_{t,\alpha}(K_{3,3})$.
For non-planarity, this immediately yields by \Cref{thm:RR subgraph} that if $L_t\to \infty$, the existence probability of a $K_{3,3}$, or $K_5$ subgraph (depending on $\alpha$) converges to one, and so the graph is non-planar with high probability.

Assume now that $L_t\to 0$. Then we must have $tr_t^d\to 0$, and so $M_{t,\alpha}(K_5)  = o_t(M_{t,\alpha}(K_{3,3}))$ if $\alpha < \frac 43d$, and similarly $M_{t,\alpha}(K_{3,3}) = o_t(M_{t,\alpha}(K_5))$ if $\alpha > \frac 43d$.
This yields that in the regime $L_t\to 0$, no $K_5$ or $K_{3,3}$ appears with high probability.
We additionally need to check all proper subdivisions of $K_5$ and $K_{3,3}$.
First, observe that if $L_t\to 0$, no path of length $2^5-2=30$ appears in the graph with high probability by \Cref{lemma:RR subdivisions}. Further, there are only finitely many proper subdivisions of $K_5$ and $K_{3,3}$ with no path of length $30$; for any such subdivision $\Gamma'$, we have $\EE G_{\Gamma'}(\hat \eta_t) = O(L_t)$ by \Cref{lemma:RR subdivisions}. This yields the statement of the theorem.

Note that \Cref{remark:RR planarity} can be made precise by showing a stronger statement than \Cref{lemma:RR subdivisions} for subdivisions, namely that for appropriate $\alpha$ and $r_t$, $\EE G_{\Gamma'}(\hat \eta_t)= o (M_{t,\alpha}(\Gamma))$ for any proper subdivision $\Gamma'$ of a complete (bipartite) graph $\Gamma$.
\end{proof}

\begin{proof}[Proof of \Cref{thm:RR subgraph} and \Cref{lemma:RR subdivisions}.]

Throughout the proof, we may assume that $tr_t^d, tr_t^\alpha\to 0$ as $t\to \infty$; if the threshold functions $M_{t,\alpha}(\cdot)= \Theta_t(1)$, this is satisfied, and the general statement (i.e. for larger $r_t$) follows from a monotonicity argument. 

First, let $\hat\xi_t$ be a $\Pareto(r_t,\alpha)$ marked homogeneous Poisson point process on $B^d$. This is comparable to $\hat \eta_t$, in a sense that the existence probability of a subgraph in $\c G_t(\hat \eta_t)$ and $\c G_t(\hat \xi_t)$ is determined by the same threshold function. This simply follows from the fact that $W$ is contained between two balls, and that the model is not sensitive to scaling by a constant; a short discussion was given in the proof of \Cref{thm:SRGGcliques} (see the argument following \eqref{eq:ball sandwich}).

Secondly, for $0<s\leq 1$, define the point process 
$\hat \xi_{t,s}\coloneq\hat \xi_t \big |_{sB^d\times [0,s]}$, that is, we restrict the location of the points, as well as the maximal value of the marks. We then aim to understand the probability $\PP(G_{\Gamma}(\xi_{t,s})\geq 1)$, and in particular, find threshold functions similarly to the overall results, now with a potential dependence on $s=s_t$.
This yields the original statement as follows.
On the one hand, $\PP(G_{\Gamma}(\hat \xi_t)\geq 1)\geq \PP(G_{\Gamma}(\hat \xi_{t,1})\geq 1)$. 
On the other hand, observe that $\hat \xi_t(B^d\times (1,\infty))$ is a Poisson distributed random variable with parameter $\kappa_d t\mu_\alpha((1,\infty)) = \kappa_dt r_t^\alpha$. As a consequence,
\begin{align}\label{eq:no large marks}
\begin{split}
\PP(G_{\Gamma}(\hat \xi_t)\geq 1) 
& \leq \PP(G_{\Gamma}(\hat \xi_{t,1})\geq 1) + \PP(\hat \xi_t \neq \hat \xi_{t,1}) 
= \PP(G_{\Gamma}(\hat \xi_{t,1})\geq 1)  +\PP(\exists \hat y\in \hat \xi_t\colon R_y>1) \\
&\leq \PP(G_{\Gamma}(\hat \xi_{t,1})\geq 1) + \kappa_d tr_t^\alpha =  \PP(G_{\Gamma}(\hat \xi_{t,1})\geq 1) +  o_t(1).
\end{split}
\end{align}
In other words, the threshold function of a subgraph existence for $\hat \xi_t$ and $\hat \xi_{t,1}$ is the same.

Now, one might consider simply counting subgraphs isomorphic to $\Gamma$, and finding the existence probability using expectation and variance results of $G_{\Gamma}(\hat \xi_{t,s})$. However, the expected number of subgraphs is highly distorted: this is because a few large marks, even if they appear with very low probability, can produce many subgraphs at once. We can eliminate this tendency by counting vertices with large marks contained in certain kind of subgraph, instead of counting the subgraphs themselves. 
Formally, consider the following.
For a graph $\Gamma$ of order $m$, real number $s > 0$, marked point configuration $\hat \zeta$, and $\hat y\in \hat \zeta$, let $X_{\Gamma}(\hat y,\hat \zeta)$ be the indicator of the event that there are points $\hat y_2,\ldots, \hat y_{m}\in \hat \zeta$ such that $\c G_t(\{\hat y, \hat y_2,\ldots, \hat y_{m}\})$ contains $\Gamma$ as a subgraph, and $R_{y_2},\ldots, R_{y_{m}}\leq R_{y}$. That is, $X_{\Gamma}(\hat y,\hat \zeta)=1$ whenever $\hat y$ is contained in a subgraph isomorphic to $\Gamma$ in $\c G_{t}(\hat \zeta)$, where the mark of $y$ is largest (among the vertices forming the $\Gamma$-subgraph). Note that we do not require that the points induce exactly $\Gamma$, i.e. additional edges are allowed.
Write 
\begin{equation}\label{eq:RR score}
F_\Gamma(\hat \zeta) \coloneq \sum_{\hat y\in \hat\zeta} X_{\Gamma}(\hat y,\hat \zeta)
\end{equation}
for the total number of such points $\hat y$, and in particular let $N_{\Gamma,s}\coloneq F_{\Gamma}(\hat \xi_{t,s})$.
By construction, $\PP(G_{\Gamma}(\hat \xi_{t,s})\geq 1 ) = \PP(N_{\Gamma,s}\geq 1)$. 
It turns out that unlike the subgraph count $G_{\Gamma}$, the vertex-type count $N_{\Gamma,1}$ is a good indicator for the existence of a subgraph.
Throughout the proof, we consider the probability $\PP(N_{\Gamma,s}\geq 1)$, and find the appropriate threshold function using first- and second moment bounds on $N_{\Gamma,s}$.

For the non-existence of a subgraph, it is sufficient to show, by Markov's inequality, 
that $\EE N_{\Gamma,1}\to 0$.
We begin by considering a general strategy to evaluate $\EE N_{\Gamma,s}$, using induction on the number of vertices. If $\Gamma$ is a single point, $N_{\Gamma,s}$ is simply the number of points in $\hat \xi_{t,s_t}$, which is a Poisson random variable by construction, with parameter $\kappa_d s_t^d \mu_{\alpha}([0,s_t])$.

Now, let $\Gamma$ be a connected graph on $k\geq 2$ vertices, and $0<s\leq 1$. By the Mecke formula, 
\[\EE N_{\Gamma, s}= t\int_{sB^d} \int_{r_t}^{s} f_\alpha(x) \PP\Big(X_{\Gamma} \big((y,x),\hat \xi_{t,s}\cup\{(y,x)\}\big)=1\Big) \dd x \dd y.\]

Set $\hat y \coloneq (y,x)$.
Now, if $\c G_t(\{\hat y, \hat y_2,\ldots, \hat y_k\})$ contains $\Gamma$ as a subgraph for some $\hat y_2,\ldots, \hat y_k\in \hat \xi_{t,s}$, then by the connectedness of $\Gamma$, the maximal distance between their spatial components is at most $k$ times the maximal mark among them. By definition of $X_{\Gamma}(\cdot)$, the maximal mark must be exactly $x$. As a consequence, if the indicator is one, then $\c G_t(\hat \xi_{t,s}\cap (y + kx B^d))$ contains a graph isomorphic to $\Gamma-v$ for some vertex $v\in V(\Gamma)$, where each of the marks is at most $x$. 
We obtain an upper bound by allowing all of the points to be in $y+kxB^d$, and have marks up to $kx$. This makes the existence probability independent of $y$ (but not the value of the mark $x$). 
For a lower bound, we can restrict $y$ to $\frac 12 s B^d$ and $x$ to $[2r_t, s/2]$, and only allow the rest of the points to be in $y+xB^d\subset s B^d$. This gives 
    \begin{equation}\label{eq:RR subgraph induction}
        \sum_{v\in V(\Gamma)} t s^d \int_{2r_t}^{s/2} f_\alpha(x) \PP(N_{\Gamma-v,x}\geq 1) \dd x 
        \lesssim \EE N_{\Gamma,s} 
        \lesssim \sum_{v\in V(\Gamma)} t s^d \int_{r_t}^{s} f_\alpha(x) \PP(N_{\Gamma-v,kx}\geq 1) \dd x.
    \end{equation}
    The reason for restricting $x\geq 2r_t$ in the lower bound will be explained later on.
    Throughout the proof, we will let $(s_t)_{t>0}$ be monotone decreasing with $s_t\in [r_t,1]$.
    
    This is the inductive relationship in its most general form. 
    For an upper bound, we can use Markov's inequality, or more precisely, $\PP(N_{\Gamma-v,kx}\geq 1)\leq \min \{1, \EE  N_{\Gamma-v,kx}\}$.
    To lower bound the existence probability, we additionally need to control the second moment of $N_{\Gamma-v,x}$: we use this approach for $\alpha\geq d$. 
    If $\alpha<d$, we give an alternative, more constructive proof for the lower bound.
    
    To make the computations more concise, introduce the following notation for the expressions appearing in $M_{t,\alpha}(\cdot)$:
    \begin{align}\label{eq:b}
    \begin{split}
        b_{1,\alpha}(n)& \coloneq t (tr_t^d)^{\frac{\alpha}{d} + \ldots + \frac{\alpha^n}{d^n}} = t^{\frac{\alpha^n}{d^n}} (tr_t^\alpha)^{1 + \frac{\alpha}{d} + \ldots + \frac{\alpha^{n-1}}{d^{n-1}}},\\
        b_{2,\alpha}(n) & \coloneq t (tr_t^d)^{\frac{n\alpha}{d}}, \\
        b_3(n) & \coloneq t(tr_t^d)^n
    \end{split}
    \end{align}
    for all $n\geq 0$, where $\frac{\alpha}{d}+\ldots \frac{\alpha^n}{d^n} \coloneq 0$ if $n=0$ by convention of the empty sum, and in particular $b_{1,\alpha}(0)=t$. Note that in order to not overwhelm the already busy notation, we omit the dependence on $t$ and $r_t$.

    \subsection*{The complete graph $K_n$}

    \paragraph{The upper bound.}
    
    We claim that $\EE N_{K_n,s_t}\lesssim I_{\alpha,s_t}(K_n)$ with $I_{\alpha,s_t}(K_n)$ given by the following table:
    \begin{equation}\label{eq:RR upper bound}
    \begin{tabular}{c||C{4.5cm}|C{4.5cm}}
     & $s_t \leq t^{-\frac 1d}$ & $s_t > t^{-\frac 1d}$\\
    \hline
    \hline
    $\alpha < d$ & \multicolumn{2}{c}{$s_t^d b_{1,\alpha}(n-1)$}   \\
    \hline
    $\alpha = d$ & $s_t^d \l(\ln \frac{s_t}{r_t}\r)^{n-1} b_3(n-1)$ & $s_t^d b_3(n-1) \l(\ln \frac{1}{tr_t^d}\r)^{n-1}$ \\
    \hline
    $\alpha > d $ & \multicolumn{2}{c}{$s_t^d b_3(n-1)$} 
    \end{tabular}
    \end{equation}
    For the base case, $N_{K_1,s_t}$ is Poisson with parameter $\kappa_d s_t^d \mu_{\alpha}([0,s_t])$. Thus we obtain $\EE N_{K_1,s_t}\lesssim ts_t^d$, where we upper bounded the term $\mu_{\alpha}([0,s_t])$ (coming from the restriction on the marks) trivially by $1$. 
    Now, fix $n\geq 1$, and assume the statement holds for $K_n$ and all $s_t$.
    Using \eqref{eq:RR subgraph induction} and the induction hypothesis, we have for $\alpha\neq d$ that
    \[\EE N_{K_{n+1},s_t}\lesssim t s_t^d \int_{r_t}^{s_t} f_\alpha(x) \min \{1, \EE N_{K_{n}, (n+1)x}\} \dd x \lesssim s_t^d t r_t^\alpha \int_{r_t}^{s_t} x^{-\alpha-1} \min \{1, x^d b(n-1)\} \dd x\]
    where $b=b_{1,\alpha}$ if $\alpha < d$ and $b=b_3$ if $\alpha>d$. 
    For $\alpha<d$, split the integral into two parts at $(b_{1,\alpha}(n-1))^{-\frac 1d}$, where the upper bound for the expectation surpasses one.
    As we a priori do not know where $r_t$ and $s_t$ fall compared to this value, we expand the integral to the entirety of the real line, and obtain
    \begin{align*}
        \EE N_{K_{n+1},s_t}
        & \lesssim s_t^d t r_t^\alpha \l[ \int_{0}^{(b_{1,\alpha}(n-1))^{-\frac 1d}} x^{d-\alpha-1} b_{1,\alpha}(n-1) \dd x + \int_{(b_{1,\alpha}(n-1))^{-\frac 1d}}^\infty x^{-\alpha-1} \dd x \r] \\
        & \lesssim s_t^d t r_t^\alpha \l[ b_{1,\alpha}(n-1) \big[ x^{d-\alpha} \big]_0^{(b_{1,\alpha}(n-1))^{-\frac 1d}}  - \big[ x^{-\alpha}\big]_{(b_{1,\alpha}(n-1))^{-\frac 1d}}^\infty \r] \\
        & \lesssim s_t^d t r_t^\alpha \, (b_{1,\alpha}(n-1))^{\frac{\alpha}{d}}.
    \end{align*}
    Note that since $d-\alpha>0$, the first integral is dominated by contributions around the upper limit, and the second integral around its lower limit. Hence the main contribution overall comes from marks in the neighbourhood of $(b_{1,\alpha}(n-1))^{-\frac 1d}$. 
    Now, observe that the sequence of functions $b_{1,\alpha}(\cdot)$ given in \eqref{eq:b} satisfies the recurrence relation $b_{1,\alpha}(n) = tr_t^\alpha(b_{1,\alpha}(n-1))^{\frac{\alpha}{d}} $, and is in fact defined by it together with $b_{1,\alpha}(0)=t$. As a consequence, $\EE N_{K_{n+1},s_t}\lesssim s_t^d b_{1,\alpha}(n)$.
    
    For $\alpha>d$, upper bound by $I_{\alpha,x}(K_n)$ for all values of $x$, and obtain
    \[ \EE N_{K_{n+1},s_t} (\hat \xi_{t,s_t}) \lesssim s_t^d t r_t^\alpha \int_{r_t}^{s_t} b_3(n-1) x^{d-\alpha-1} \dd x\lesssim s_t^d tr_t^\alpha b_3(n-1) r_t^{d-\alpha} = s_t^d tr_t^db_3(n-1),\]
    as leading contribution comes from small values of $x$ (namely, around $r_t$).
    The sequence of functions $b_3(\cdot)$ given in \eqref{eq:b} is determined by the recurrence relation $b_3(n) = t r_t^db_3(n-1) $, together with the initial case $b_3(0) = t$. The expectation is thus upper bounded by $s_t^d b_3(n)$.
    
    Due to the logarithmic factor, the case of $\alpha = d$ requires a little more work. For $r_t\leq s_t\leq t^{-\frac 1d}$, 
    \begin{align*}
        \EE N_{K_{n+1},s_t}
        & \lesssim s_t^d tr_t^\alpha \int_{r_t}^{s_t} \frac 1x \l( \ln \l(\frac{x}{r_t}\r)\r)^{n-1} b_3(n-1)\dd x \approx   s_t^d tr_t^d b_3(n-1) \l[\l( \ln \l(\frac{x}{r_t}\r)\r)^{n}\r]_{r_t}^{s_t} \\
        & = s_t^d tr_t^d b_3(n-1) \l( \ln \l(\frac{s_t}{r_t}\r)\r)^{n} =  s_t^d b_3(n) \l( \ln \l(\frac{s_t}{r_t}\r)\r)^{n}.
    \end{align*}
    For $s>t^{-\frac 1d}$, set $\beta=\beta(n-1)\coloneq b_3(n-1) \l(\ln \l(\frac{1}{tr_t^d}\r)\r)^{n-1}$ 
    to obtain
    \begin{align*}
        \EE &N_{K_{n+1},s_t}\\
        &\lesssim s_t^d tr_t^d \l[b_3(n-1) \int_{r_t}^{t^{-\frac 1d}} \frac 1x \l( \ln \l(\frac{x}{r_t}\r)\r)^{n-1} \dd x  + \int_{t^{-\frac 1d}}^{s_t} x^{-d-1} \min\l\{1, \beta(n-1) x^d\r\} \dd x\r]\\
        & \lesssim s_t^d tr_t^d \l[b_3(n-1) \l(\ln \l(\frac{1}{tr_t^d}\r)\r)^n  + \beta(n-1) \int_{t^{-\frac 1d}}^{\beta^{-\frac 1d}} x^{-1} \dd x  + \int_{\beta^{-\frac 1d}}^\infty x^{-d-1} \dd x \r]\\
        & \lesssim s_t^d tr_t^d \l[ \beta (n-1) \l(\ln \l(\frac{1}{tr_t^d}\r)\r) + \beta(n-1) \ln \l(\frac{t}{\beta(n-1)}\r) + \beta(n-1)\r].
    \end{align*}
    Note that since $tr_t^d\to 0$, we have $\beta(n-1)/t = (tr_t^d)^{n-1} (\ln(1/(tr_t^d))^{n-1}\to 0$, and hence the way the integral was split is sensible. Additionally, this yields that $\ln(t/\beta(n-1))\approx \ln(1/(tr_t^d))$. It follows that the second summand is of the same order as the first, and the third summand is smaller.
    We thus have
    \[\EE N_{K_{n+1},s_t} 
    \lesssim s_t^d tr_t^d \beta(n-1) \ln \l(\frac{1}{tr_t^d}\r) 
    = s_t^d b_3(n) \l(\ln \l(\frac{1}{tr_t^d}\r)\r)^n.\]
    This completes the proof of the upper bound for the complete graph.

    \paragraph{The lower bound.}

    To prove the existence of a given complete graph in $\c G_t$, we showcase two different approaches. One is a standard second moment method proof, the other, which can be used when the model is dominated by large marks, is more constructive. We start with the latter; for this approach to be applicable, let $\alpha<d$.
    The idea is the following: we work inductively, showing that there exists a vertex in $\hat \xi_{t,s_t}$ that has a sufficiently large mark, and that a vertex with such a large mark is contained in a complete $n$-subgraph with high probability.
    
    To be precise, we claim that if $s_t^d b_{1,\alpha} (n-1)\to \infty$, then $\PP(N_{K_n,s_t}\geq 1)\to 1$; since $b_{1,\alpha}(n-1)=M_{t,\alpha}(K_n)$, this implies the claim of the theorem with $s_t=1$. 
    For the base case $n=1$, we use that $N_{K_n,s_t}$ is $\Po(\kappa_d ts_t^d \mu_{\alpha}([r_t,s_t]))$. If $r_t=o_t(s_t)$, we have $\mu_{\alpha}([r_t,s_t]) = r_t^\alpha [s_t^{-\alpha}-r_t^{-\alpha}] \to 1$, hence the parameter is $\Theta_t(ts_t^d)=\Theta_t(s_t^d b_{1,\alpha}(0))$. Consequently, if $s_t^d b_{1,\alpha}(0)\to \infty$, then $N_{K_1,s_t}\geq 1$ with high probability. 
    Observe that since $tr_t^d = b_{1,\alpha}(0) r_t^d \to 0$, then $s_t^d b_{1,\alpha}(0)\to \infty$ indeed implies that $s_t/r_t\to \infty$, which is necessary for the stated order of the parameter.
    Note that if we allowed $s_t$ to be arbitrarily close to $r_t$, the existence claim would no longer hold: the condition on the mark would be too restrictive to guarantee existence.

    Now, let $n\geq 1$ be arbitrary, and assume the statement for existence holds for $K_n$.
    Consider the following lower bound: for any $\rho_t>0$,
    \begin{align*}
        \PP(N_{K_{n+1},s_t}\geq 1) 
        & \geq \PP(\exists \hat y\in \hat \xi_{t,s_t}\colon y\in (s_t/2)B^d,  R_y\geq \rho_t,\,  X_{K_{n+1}} (\hat y, \hat \xi_{t,s_t})=1)\\
        & \geq  \PP\big(\exists \hat y\in \hat \xi_{t,s_t}\colon y\in (s_t/2) B^d,  R_y\geq \rho_t,\, F_{K_n}\big(\hat \xi_{t,s_t}\big|_{(y+\rho_t B^d)\times [0,\rho_t]}\big) \geq 1\big),
    \end{align*}
    with $F_{K_n}$ given in \eqref{eq:RR score} and $X_{K_n}$ in the preceding paragraph.
    Write $H'$ for the number of points in $\hat \xi_{t,s_t}\cap A_{\rho_t}$ with $A_{\rho_t}\coloneq (s_t/2)B^d\times [\rho_t,\infty)$, and let
    \[ H \coloneq \sum_{\hat y \in \hat \xi_{t,s_t}} \ind \big(\hat y \in A_{\rho_t},\, F_{K_n}\big(\hat \xi_{t,s_t}\big|_{(y+\rho_t B^d)\times [0,\rho_t]}\big) \geq 1\big).\]
    With this notation, the previous bound is equivalent to $\PP(N_{K_{n+1,s_t}}\geq 1)\geq \PP(H\geq 1)$.
    
    Assume now that $\rho_t = o_t(s_t)$. 
    Then, for sufficiently large $t$, we have $y+\rho_t B^d \subset s_tB^d$, and so $\hat \xi_{t,s_t}\big|_{(y+\rho_t B^d)\times [0,\rho_t]}$ has the same law as $\hat \xi_{t,s_t}\big|_{\rho_t B^d\times [0,\rho_t]}=\hat \xi_{t,\rho_t}$. 
    Using this observation together with the Mecke formula,
    \begin{align*}
        \EE H 
        & = \sum_{\hat y \in \hat \xi_{t,s_t}} \ind(\hat y \in A_{\rho_t}) \ind (F_{K_n}(\hat \xi_{t,s_t}\big|_{(y+\rho_t B^d)\times [0,\rho_t]})\geq 1) \\
        & = \int_{B^d\times[0,\infty)} \ind ((y,x)\in A_{\rho_t})\PP(N_{K_n,\rho_t}(\hat \xi_{t,\rho_t})\geq 1)  f(x)\dd y \dd x \\
        & =  \PP(N_{K_n,\rho_t}(\hat \xi_{t,\rho_t})\geq 1) \EE H',
    \end{align*}
    and similarly, $\EE H^2 \leq \EE H + \EE H \EE H'$. Consequently, by the second moment method,
    \[\PP(N_{K_{n+1},s_t}\geq 1) 
    \geq  \PP(H\geq 1) \geq \frac{(\EE H)^2}{\EE H^2}
     \geq \frac{\EE H'}{1+ \EE H'}\cdot \PP(N_{K_{n},\rho_t}(\hat \xi_{t,\rho_t})\geq 1).\]
    By definition, $H'$ is a Poisson random variable with $\EE H' = \Theta(ts_t^d r_t^\alpha \rho_t^{-\alpha})$; if this expression tends to infinity, the first factor tends to one.     
    Additionally, by the induction hypothesis, the second factor tends to one for any $\rho_t$ with $\rho_t^d b_{1,\alpha}(n-1)\to \infty$.
    We show that $\rho_t = o(s_t)$ exists such that both of these conditions are fulfilled, which is sufficient to complete the inductive step; recall that the condition $\rho_t = o(s_t)$ is necessary to obtain the geometric containment necessary for the lower bound.
   First, notice that all solutions $\rho_t$ will indeed have $\rho_t = o_t(s_t)$: if conversely, $\rho_t\gtrsim s_t$ would hold, we would have $ts_t^d r_t^\alpha \rho_t^{-\alpha} \lesssim tr_t^\alpha s_t^{d-\alpha}\to 0$, since $tr_t^\alpha\to 0$, $s_t\leq 1$, and $d-\alpha>0$.
    Hence it is sufficient to find $\rho_t$ such that 
    \[
    \left\{ 
    \begin{aligned}
        \rho_t^\alpha (s_t^d t r_t^\alpha)^{-1} &\to 0\\
        \rho_t^\alpha b_{1,\alpha}^{\frac{\alpha}{d}}(n-1)&\to \infty.
    \end{aligned} 
    \right.
    \]
    This can be solved for $\rho_t$ if and only if $s_t^d (tr_t^\alpha) b_{1,\alpha}^{\frac{\alpha}{d}}(n-1)\to \infty$. By using the recurrence relation, this is equivalent to $s_t^d b_{1,\alpha}(n)\to \infty$, which was exactly our assumption for the existence of a $K_{n+1}$. This completes the proof of existence for $\alpha<d$.

    For $\alpha\geq d$, the graph is not dominated by vertices with large marks in the manner above, and we rely on a second moment method argument.
    Recall that our aim is to show that if $M_{t,\alpha}(K_n)\to \infty$ (as defined in the statement of the theorem), $K_n$ appears in the graph with high probability. 
    We show, using the second moment method, that this holds even if we only consider the vertices with mark at most $t^{-\frac 1d}$. From the integrals computed in the upper bound, one can see that for $\alpha\leq d$, the contribution coming from marks in $[r_t,t^{-\frac 1d}]$ indeed has the same order as the entire integral. The restriction is necessary here because large marks can distort the variance, even if they do not distort the expectation. 

    Formally, consider the marked process $\hat \theta_{t}\coloneq \hat \xi_t\big |_{B^d\times [2r_t,2t^{-1/d}]}$. 
    Note how compared to $\hat \xi_{t,s}$, the marks are restricted, but not the spatial component. 
    Write $N_{K_n}' \coloneq F_{K_n}(\hat \theta_t) = \sum_{\hat y \in \hat\theta_t} X_{K_n}(\hat y,\hat \theta_t)$ for the number of points in $\hat \theta_t$ that are in a $K_n$ subgraph where their mark is the largest.
    Note that this is analogous to $N_{K_n,s}$, only now using $\hat \theta_{t}$ instead of $\hat \xi_{t,s}$.
    By construction, $G_{K_n}(\hat \xi_t)\geq N_{K_n}'$, and as a consequence, it is sufficient to show $\PP(N_{K_n}'\geq 1)\to 1$ whenever $M_{t,\alpha}(K_n)\to \infty$. 
    We do this by showing that $\EE N_{K_n}'\gtrsim M_{t,\alpha}(K_n)$, and giving an upper bound on the second moment. 

    Applying the same strategy as when obtaining \eqref{eq:RR subgraph induction}, we have 
    \[\EE N_{K_n}' \gtrsim tr_t^\alpha \int_{2r_t}^{t^{-\frac 1d}} x^{-\alpha-1} \PP(N_{K_{n-1},x}\geq 1) \dd x.\]
    To obtain a lower bound on this expression, we use a reverse Markov-type inequality for the existence probability. 
    The following lemma states the necessary results.
    \begin{lemma}\label{lemma:RR complete lower}
    Let $s_t\in [2r_t,t^{-\frac 1d}]$. Then, with $I_{\alpha,s_t}(K_n)$ given by \eqref{eq:RR upper bound} on p.~\pageref{eq:RR upper bound},
    \begin{enumerate}[label=\roman*),ref=\roman*)]
        \item\label{item:one}  $\EE N_{K_{n}, s_t} \gtrsim I_{\alpha, s_t}(K_n)$
        \item\label{item:two} $\EE N_{K_{n}, s_t}^2\lesssim I_{\alpha, s_t}(K_n)$
        \item\label{item:three} $\PP(N_{K_n,s_t}\geq 1) \gtrsim I_{\alpha,s_t}(K_n)$
        \item\label{item:four} $\EE N_{K_n}'\gtrsim M_{t,\alpha}(K_n)$
        \item\label{item:five} $\EE (N_{K_n}')^2 = (\EE N_{K_n}')^2(1+o_t(1))$ whenever $\EE N_{K_n}'\to \infty$ as $t\to \infty$.
    \end{enumerate}
    \end{lemma}
    Observe that for $s_t\leq t^{-\frac 1d}$, we have $s_t^d b_3(n-1)\leq t^{-1} b_3(n-1)=(tr_t^d)^{n-1}\to 0$, hence $I_{\alpha,s_t}(K_n)\to 0$ for all $\alpha \geq d$, and so it is reasonable to lower bound a probability by it. 
    The lemma yields the existence of a $K_n$ as follows: from \ref{item:four} of the lemma, $\EE N_{K_n}'\to \infty$ whenever $M_{t,\alpha}(K_n)\to \infty$, and by the second moment method,
    \[\PP(N_{K_n}'\geq 1) = \PP(N_{K_n}'>0)\geq \frac{(\EE N_{K_n}')^2}{\EE (N_{K_n}')^2}.\]
    From \ref{item:five}, the last expression converges to one, which concludes the proof of the theorem for $K_n$.
    \begin{proof}[Proof of \Cref{lemma:RR complete lower}.]
    First, we show that \ref{item:three} follows from \ref{item:one} and \ref{item:two}: by the second moment method,
    \[\PP(N_{K_n,s_t}\geq 1) = \PP(N_{K_n,s_t} >  0) \geq \frac{(\EE N_{K_{n}, s_t})^2}{\EE N_{K_{n}, s_t}^2}.\]
    Bounding the expressions by the claim of \ref{item:one} and \ref{item:two}, the statement of \ref{item:three} follows.
    
    The rest of the proof is carried out via induction. 
    For $n=1$, we again use that $N_{K_1,s_t}$ is Poisson with parameter $\kappa_d ts_t^d \mu_{\alpha}([r_t,s_t])\gtrsim ts_t^d (2^\alpha-1)\gtrsim ts_t^d = I_{\alpha,s_t}(K_1)$, hence i) and ii) follow for $n=1$.
    
    Assume now that all statements hold for a fixed $n\geq 1$, and let us show them for $n+1$.
    Start with inequality \ref{item:one}. By \eqref{eq:RR subgraph induction} and the induction hypothesis for \ref{item:three}, 
    \[\EE N_{K_{n+1}, s_t}\gtrsim ts_t^d \int_{2r_t}^{s_t/2} f_\alpha(x) \PP(N_{K_{n}, x}\geq 1)\dd x\gtrsim ts_t^d \int_{2r_t}^{s_t/2} f_\alpha(x) I_{\alpha,x}(K_n)\dd x.\]
    This is, up to multiplicative constants in the interval endpoints, the same expression that was used in the upper bound. It is easy to see in the computations for the upper bound that the modified constants do not change the order of the integral; in particular, the resulting expression is $I_{\alpha,s_t}(K_{n+1})$, and the claim follows.

    For the second moment of $N_{K_{n+1,s_t}}$, we have by definition
    \[\EE N_{K_{n+1},s_t}^2  = \EE \sum_{\hat y_1 \in \hat \xi_{t,s_t}} \sum_{\hat y_2 \in \hat \xi_{t,s_t}} \ind (X_{K_{n+1}}(\hat y_1,\hat \xi_{t,s_t})=1) \cdot \ind (X_{K_{n+1}}(\hat y_2,\hat \xi_{t,s_t})=1). \]
    If $\hat y_1 = \hat y_2$, we obtain the expectation. For $\hat y_1\neq \hat y_2$, the sum is a Poisson $U$-statistic of order two. Applying the Mecke formula, we obtain
    \begin{align*}
        \EE N_{K_{n+1},s_t}^2   & =  \EE N_{K_{n+1},s_t} \\ & + t^2 \int_{(s_tB^d\times [r_t,s_t])^2}  
        \begin{multlined}[t]
        \PP\big(X_{K_{n+1}}(\hat y_1, \hat \xi_{t,s_t}\cup\{\hat y_1,\hat y_2\})=1, \\ X_{K_{n+1}}(\hat y_2, \hat \xi_{t,s_t}\cup\{\hat y_1,\hat y_2\})=1\big )
        f_\alpha(x_1) f_\alpha(x_2)
        \dd x_2 \dd y_2 \dd x_1 \dd y_1,
        \end{multlined}
    \end{align*}
    where now $\hat y_i = (y_i,x_i)$ for $i=1,2$.  
    Observe that if $\|y_1-y_2\|\geq x_1+x_2$, then the two indicators are independent (as the relevant points are contained in disjoint balls), and we can factor the probability. This yields
    \[\EE N_{K_{n+1},s_t}^2 \leq \EE N_{K_{n+1},s_t} + (\EE N_{K_{n+1},s_t})^2 + J_{s_t}\]
    with
    \begin{multline*}
        J_{s_t} \coloneq t^2 \int_{s_tB^d} \int_{r_t}^{s_t} \int_{y_1+(x_1+x_2)B^d} \int_{r_t}^{s_t} f_\alpha(x_1) f_\alpha(x_2) \\
        \times \PP(X_{K_{n+1}}(\hat y_1, \hat \xi_{t,s_t}\cup\{\hat y_1,\hat y_2\}) = X_{K_{n+1}}(\hat y_2, \hat \xi_{t,s_t}\cup\{\hat y_1,\hat y_2\})=1) \dd x_2 \dd y_2 \dd x_1 \dd y_1.
    \end{multline*}    
    Since $\EE N_{K_n,s_t}\lesssim I_{t,s_t}(K_n)\to 0$, we have $\EE N_{K_{n+1},s_t}^2 = o_t(\EE N_{K_{n+1},s_t})$, and it remains to be shown that $J_{s_t}\lesssim I_{\alpha,s_t}(K_{n+1})$. 
    The key observation is the following: assuming that $x_1\geq x_2$ in the integral (which we may do for the price of a constant factor of two), 
    \[X_{K_{n+1}}(\hat y_2, \hat \xi_{t,s_t}\cup\{\hat y_1,\hat y_2\})=1 \ \Leftrightarrow \ X_{K_{n+1}}(\hat y_2, \hat \xi_{t,s_t}\cup\{\hat y_2\})=1.\]
    This is because if the mark of $\hat y_2$ is smaller than that of $\hat y_1$, adding the latter to the process does not help in creating a $K_{n+1}$ where $\hat y_2$ has the largest mark.
    As a consequence,
    \begin{align*}
        J_{s_t}
        & \lesssim t \int_{s_tB^d} \int_{r_t}^{s_t} f_\alpha(x_1) \l[ t\int_{2x_1B^d} \int_{r_t}^{x_1} f_\alpha(x_2) \PP(X_{K_{n+1}}(\hat y_2, \hat \xi_{t,x_2}\cup\{\hat y_2\})=1)\dd x_2 \dd y_2\r] \dd x_1 \dd y_1 \\
        & \lesssim t s_t^d \int_{r_t}^{s_t} f_\alpha(x_1) \EE N_{K_{n+1},2x_1} \dd x_1\lesssim ts_t^d \int_{r_t}^{s_t} f_{\alpha}(x_1) I_{\alpha,x_1}(K_{n+1})\dd x_1.
    \end{align*}
    As before, this is essentially the same expression already encountered for the expectation upper bound, and we obtain $J_{s_t}\lesssim I_{\alpha,s_t}(K_{n+2}) = o_t(I_{\alpha,s_t}(K_{n+1}))$, where the second statement comes from direct comparison of the definitions of $I_{\alpha,s_t}$. This completes the proof of \ref{item:two}.

    The proof of \ref{item:four} and \ref{item:five} follow analogously. In particular, similarly to the induction for \ref{item:one}, 
    \[\EE N_{K_{n+1}}'\gtrsim  t\int_{2r_t}^{t^{-\frac 1d}} f_\alpha(x) I_{\alpha,x}(K_n)\dd x.\]
    Observe that the integral itself coincides with the integral in the lower bound for $\EE N_{K_{n+1},s_t}$ with $s_t=2t^{-\frac 1d}$, and the entire expression is only missing the factor $s_t^d = (2t^{-1/d})^d$. As a consequence, 
    \[\EE N_{K_{n+1}}'\gtrsim  \frac{I_{\alpha,t^{-1/d}}(K_{n+1})}{(t^{-1/d})^d} = tI_{\alpha,t^{-1/d}} (K_{n+1}).\]
    Comparing the definitions of $I_{\alpha,s_t}$ and $M_{t,\alpha}$, the last expression is exactly $M_{t,\alpha}(K_{n+1})$. This completes the proof of \ref{item:four}.
    For $K_{n+1}'$, we have $\EE (N_{K_{n+1}}')^2 \leq \EE N_{K_{n+1}}' + (\EE N_{K_{n+1}}')^2+ J_{s_t}'$, where $J_{s_t}'$ is defined similarly to $J_{s_t}$, with the first integral running over $B^d$ instead of $s_tB^d$. 
    If $\EE N_{K_{n+1}}'\to \infty$, the first summand (expectation) is negligible compared to its square; it remains to be shown that $J_{s_t}' = o_t((\EE N_{K_{n+1}}')^2)$.

    By the same argument as before,
    \[J_{s_t}'\lesssim t \int_{r_t}^{t^{-\frac 1d}} f_{\alpha}(x_1)  \EE N_{K_{n},2x_1} \dd x_1\lesssim \frac{I_{\alpha,t^{-1/d}}(K_{n+2})}{(t^{-1/d})^d} = M_{t,\alpha}(K_{n+2}) = o_t(M_{t,\alpha}^2(K_{n+1})).\]

    This completes the proof of \ref{item:five}, and thus the lemma.
    \end{proof}

    \subsection*{The complete bipartite graph $K_{n,m}$}

    The overall proof outline for $K_{n,m}$ is essentially the same as for $K_n$. However, the details of the integrals get more complex, and an additional change in behaviour occurs in terms of $\alpha$.
    Let $1\leq n \leq m$ with $m\geq 2$; since $K_{1,1} = K_2$, the $n=m=1$ case has already been handled.

    \paragraph{The upper bound.}
    We claim that $\EE N_{K_{n,m},s_t}\lesssim I_{\alpha,s_t}(K_{n,m})$ with $I_{\alpha,s_t}(K_{n,m})$ given by the following table: 
    \begin{center}
    \begin{tabular}{c||C{4.2cm}|C{4.8cm}}
     & $s_t \leq t^{-\frac 1d}$ & $s_t > t^{-\frac 1d}$\\
    \hline
    \hline
    $\alpha < d$ & \multicolumn{2}{c}{$s_t^d b_{1,\alpha}(n)$}   \\
    \hline
    $\alpha = d$ & 
    $s_t^d \l(\ln \l(\frac{s_t}{r_t}\r)\r)^{n-1} b_{1,\alpha}(n)$ 
    & $s_t^d  \l(\ln \l(\frac{1}{tr_t^d}\r)\r)^{n-1}b_{1,\alpha}(n)$ \\
    \hline
    $d < \alpha < d \cdot \frac{m+n-1}{n} $ & $s_t^{d(m+n)-n\alpha} a_{\alpha}(n,m)$ & $s_t^d b_{2,\alpha}(n)$ \\
    \hline
    $\alpha = d \cdot \frac{m+n-1}{n}$ & $s_t^d \ln \l(\frac{s_t}{r_t}\r) b_3(n+m-1)$ & $s_t^d  \ln \l(\frac{1}{tr_t^d}\r)b_3(n+m-1)$ \\
    \hline
    $\alpha > d \cdot \frac{m+n-1}{n}$ & \multicolumn{2}{c}{$s_t^d b_3(n+m-1)$} 
    \end{tabular}
    \end{center}
    with $a_{\alpha}(n,m)=t^{n+m} r_t^{n\alpha}$. 
    In addition, we claim that for $d < \alpha < d \cdot \frac{m+n-1}{n}$, both bounds in the table hold for all $s_t$; the columns indicate which is optimal, but for the induction, it is helpful to have both for all values of $s_t$.
    
    Observe that since $m\geq 2$ is assumed, there are two distinct changes in behaviour depending on $\alpha$; further, the larger transition point changes depending on the parameters.
    The contrast to the case of $K_n$ already shows with the initial case $K_{1,m}$. 
    Let $L_m$ be the graph with $m$ isolated points (no edges).
    For arbitrary vertex $v$ of $K_{1,m}$, the graph $K_{1,m}-v$ contains $L_m$ as a subgraph. Recalling the definition of $N_{L_m,s}$ on page~\pageref{eq:RR score}, observe that $N_{L_m,s}\geq 1$ is equivalent to $\hat \xi_{t,s}$ containing at least $m$ points. Consequently, using the induction \eqref{eq:RR subgraph induction}, we have
    \begin{align*}
        \EE N_{K_{1,m},s_t} 
        & \lesssim ts_t^d \int_{r_t}^{s_t} f_{\alpha}(x) \PP\big(\hat \xi_{t,s_t}((m+1)xB^d)\geq m \big)\dd x \\
        & \lesssim ts_t^d \int_{r_t}^{s_t} f_{\alpha}(x) \PP\big( \Po(t((m+1)x)^d)\geq m \big)\dd x\\
        &\lesssim s_t^d t r_t^\alpha \int_{r_t}^{s_t} x^{-\alpha-1} \min \{1, t^m x^{dm}\} \dd x,
    \end{align*}
    where in the last step we used the standard Chernoff-Hoeffding bound (see e.g. \cite[Lemma~1.2]{RGGbook}) for the tail of a Poisson random variable.
    For $\alpha>dm$, the leading term comes from the lower endpoint, and we obtain $s_t^d tr_t^\alpha t^ m r_t^{dm-\alpha} = s_t^d b_3(m)$.
    For $\alpha =dm$, we obtain a logarithmic integral, handled as in the case of $K_n$, depending on whether $s_t$ is above or below the threshold $t^{-\frac 1d}$.
    For $\alpha < dm$, we show that both bounds hold for all $s_t$. 
    First, 
    \[\EE N_{K_{1,m},s_t} \lesssim s_t^d tr_t^\alpha \int_{r_t}^{s_t} t^m x^{dm-\alpha-1} \dd x \lesssim s_t^d tr_t^\alpha t^m s_t^{dm-\alpha} =  s_t^{d(m+1)-\alpha} a_\alpha(1,m).\]
    On the other hand,
    \[\EE N_{K_{1,m},s_t} \lesssim s_t^d tr_t^\alpha \l[\int_{0}^{t^{-\frac 1d}} t^m x^{dm-\alpha-1} \dd x + \int_{t^{-\frac 1d}}^{\infty} x^{-\alpha-1} \dd x \r]\lesssim s_t^d tr_t^\alpha t^{\frac{\alpha}{d}} =  s_t^d b_{1,\alpha}(1).\]
    This completes the bound for the initial case $K_{1,m}$.

    Now, consider the expectation for $K_{n+1,m}$. One can see that $I_{\alpha,s_t}(K_{n+1,m-1})\lesssim I_{\alpha,s_t}(K_{n,m})$ by comparing the expressions explicitly;
    e.g. for $d < \alpha < d \cdot \frac{m+n-1}{n+1}$ and $r_t\leq s_t\leq t^{-\frac 1d}$,
    \[I_{\alpha,s_t}(K_{n+1,m-1}) = s_t^{d(m+n)-(n+1)\alpha} t^{n+m} r_t^{(n+1)\alpha} = I_{\alpha,s_t}(K_{n,m}) \l(\frac{r_t}{s_t}\r)^\alpha \leq I_{\alpha,s_t}(K_{n,m}).\]
    The rest of the cases can be verified similarly.
    As a consequence, when choosing the vertex $v$ in \eqref{eq:RR subgraph induction}, the summand with $\Gamma-v = K_{n,m}$ dominates the one with $\Gamma-v = K_{n+1,m-1}$  (more precisely, has at least the same order).
    
    Observe that for $\alpha<d$ and $\alpha > d \cdot \frac{m+n-1}{n}$, the bounds for $K_{n,m}$ correspond closely to the ones for the complete graph of a certain size. One can show that the same recurrence holds here as for the complete graph (i.e. the defining relation of $b_{1,\alpha}$ and $b_3$), as $n$ decreases and $m$ remains fixed, i.e. the size of the smaller group decreases. 
    
    This yields the bounds in the same way as before. 
    For $\alpha = d$, we use the same approach as well, noting that here, the initial case $K_{1,m}$ does not contain a logarithmic term, while the complete graph $K_2$ does: this explains why $I_{d,s_t}(K_{n,m})$ is missing a logarithmic factor compared to $I_{d,s_t}(K_{n+1})$.

    For $d < \alpha < d \cdot \frac{m+n-1}{n}$, we have
    \[\EE N_{K_{n+1,m},s_t}\lesssim s_t^d \l[ tr_t^\alpha \int_{r_t}^{t^{-\frac 1d}} a_\alpha(n,m) x^{d(m+n)-(n+1)\alpha-1}\dd x + tr_t^\alpha \int_{t^{-\frac 1d}}^\infty b_{2,\alpha}(n) x^{d-\alpha-1} \dd x \r].\]
    The first integral $I_1$ depends on $\alpha$:
    \[I_1\lesssim
        \begin{cases}
            tr_t^\alpha a_{\alpha}(n,m) t^{-\frac 1d(d(m+n)-(n+1)\alpha)} = b_{2,\alpha}(n+1),
            & \caseif d<\alpha< d\cdot \frac{n+m}{n+1},\\
            tr_t^\alpha a_{\alpha}(n,m) \ln \l(\frac{1}{tr_t^d}\r) = b_3(n+m) \ln \l(\frac{1}{tr_t^d}\r) ,
            & \caseif \alpha=d \cdot \frac{n+m}{n+1},\\
            tr_t^\alpha a_{\alpha}(n,m) r_t^{d(m+n)-(n+1)\alpha} = b_3(n+m),
            & \caseif d\cdot \frac{n+m}{n+1} < \alpha< d \cdot \frac{n+m-1}{n}.\\
        \end{cases}\]
    Observe that together with the $s_t^d$ prefactor, these are exactly the bounds stated for $K_{n+1,m}$ in the second column of the table above.
    
    Due to $d-\alpha<0$, the second integral has order $(tr_t^d)^{\frac{\alpha}{d}} b_{2,\alpha}(n)=b_{2,\alpha}(n+1)$. This recurrence easily follows from \eqref{eq:b}, and defines the sequence of functions $b_{2,\alpha}(\cdot)$ together with $b_{2,\alpha}(0) = t$. 
    We can see that $b_{2,\alpha}(n+1)$ is either the same order as the bound on $I_1$, or negligible compared to it; this completes the proof of the bound stated in the second column.
    
    For $ d< \alpha < d \cdot \frac{m+n}{n+1}$, we additionally have
    \[\EE N_{K_{n+1,m},s_t} \lesssim s_t^d tr_t^\alpha \int_{r_t}^{s_t} a_\alpha(n,m) x^{d(m+n)-(n+1)\alpha-1}\dd x \lesssim s_t^d tr_t^\alpha a_\alpha(n,m) s_t^{d(m+n)-(n+1)\alpha},\]
    which yields the bound in the first column, as $a_{\alpha}(n+1,m)=tr_t^\alpha a_{\alpha}(n,m)$ exactly.
    Lastly, for $\alpha = d \cdot \frac{m+n}{n+1}$ and $s_t\leq t^{-\frac 1d}$, we have
    \[\EE N_{K_{n+1,m},s_t} \lesssim s_t^d tr_t^\alpha \int_{r_t}^{s_t} a_\alpha(n,m) x^{-1}\dd x \lesssim s_t^d tr_t^\alpha a_\alpha(n,m) \ln \l(\frac{s_t}{r_t}\r) = s_t^d b_3(n+m) \ln \l(\frac{s_t}{r_t}\r),\]
    completing the proof of the upper bound for $K_{n,m}$.

    \paragraph{The lower bound.}
    To show the existence of a $K_{n,m}$ subgraph, we use similar tools as for the complete graph.
    In particular, for $\alpha<d$, we use the same constructive approach as for $K_n$, taking into account that $t\rho_t^d\to \infty$ implies the existence of $m$ points in a $\rho_t$-ball w.h.p. 
    For $\alpha>d$, we use the second moment method, and the analogue of \Cref{lemma:RR complete lower}. 
    For the initial case here, observe that removing the vertex in the 1-group of $K_{1,m}$ yields the graph $L_m$ with $m$ isolated vertices. As noted before when finding the upper bound, $N_{L_m,x}\geq 1$ is equivalent to $\hat \xi_{t,x}$ containing at least $m$ points, and so the lower bound in \eqref{eq:RR subgraph induction} yields
    \[\EE N_{K_{1,m},s_t} \gtrsim ts_t^d \int_{2r_t}^{s_t/2} f_\alpha(x) \PP(\Po(tx^d)\geq m)\dd x.\]
    By the definition of a Poisson random variable, the probability in the integral can be lower bounded by $e^{-tx^d}(tx^d)^m/m!\gtrsim (tx^d)^m$ whenever $tx^d\lesssim 1$. This yields the analogue statements to \Cref{lemma:RR complete lower} for $K_{1,m}$. For general $K_{n,m}$, the same inductive approach can be used as before. 

    For $\alpha=d$, we have to modify the second moment approach, as the main contribution to the integral does not come from marks in $[r_t,t^{-\frac 1d}]$. 
    Let $\beta(0)\coloneq t$ and $\beta(n) \coloneq b_{1,\alpha}(n) (\ln( 1/(tr_t^d)))^{n-1}$ for $n\geq 1$. Define the interval $A_n = [(\beta(n-1))^{-\frac 1d}, (\beta(n))^{-\frac 1d}]$.
    The main contribution to the number of $K_{n,m}$-s comes from marks in $A_{n-1}$. 
    We can thus show that for $s_t\in A_{n-1}$, $\EE N_{K_{n,m},s_t}\gtrsim I_{\alpha,s_t}(K_{n,m}) = s_t^d \beta(n)$, and the same lower bound for $\PP(N_{K_{n,m},s_t}\geq 1)$ with the second moment method, as in \Cref{lemma:RR complete lower}.
    
    This completes the proof of \Cref{thm:RR subgraph}.
    
\subsection*{Paths and proper subdivisions}
    
    Consider the paths and proper subdivisions of \Cref{lemma:RR subdivisions}.
    For paths, note that $P_1=K_1$, and for any vertex $v\in P_{n+1}$, $P_{n+1}-v$ contains a $P_n$: removing a vertex from a path of length $2^{n+1}-2$, the larger remaining connected path has at least $(2^{n+1}-2-2)/2=2^n-2$ edges. This argument gives $\EE G_{P_n}(\hat \xi_t) \lesssim M_{t,\alpha}(K_n)$ the same way as before, and the first statement in \Cref{lemma:RR subdivisions} follows.
    Similarly, if $\Gamma$ is a proper subdivision of some $K_{n+1}$, it is easy to see that $\EE G_{\Gamma}(\hat \xi_t)\lesssim M_{t,\alpha}(K_{n+1})$. In particular, for any vertex $v\in V(\Gamma)$,  $\Gamma-v$ contains a subdivision (not necessarily proper) of $K_{n}$, and subdivisions of $K_{n,m}$ are handled analogously.

    Finally, we give a short sketch of how to approach showing a stronger statement for subdivisions, namely that in the correct regime, they are strictly less likely to occur than the complete (bipartite) graphs themselves.
    Here, the structure of the induction is more intricate, and various cases have to be considered. For an example, see \Cref{fig:subdivision removal}; the details of potential progressions are given below.
    \begin{figure}[ht]
    \centering
\begin{tikzpicture}[scale = 0.9, line width =0.8,
    node/.style={circle, fill=black, minimum size=1.1pt, inner sep=1pt},
    label style/.style={font=\tiny}
]

    \node[node, label=below left:$v_2$] (n2) at (0, 0) {};
    \node[node, label=above:$v_1$] (n1) at (0.5, 1.6) {};
    \node[node, label=below right: $v_3$] (n3) at (1, 0) {};
    \node[node, label={[inner sep = 8pt]right:$v_4$}] (n4) at (0.5, 0.5) {};
    \node[node, label={[inner sep = 7pt]right:$v_5$}] (n5) at (0.5, 1) {};
    \node at (0.5,-1) {$\Gamma$};
    
    \draw (n1) -- (n2) -- (n3) -- (n1);
    \draw (n4) -- (n2);
    \draw (n4) -- (n3);
    \draw (n4) -- (n5);
    \draw (n5) -- (n1);

    \draw[-{Stealth}, thick] (3, 1) -- (4, 1) node[midway, above] {\small $-v_4$};
    \draw[-{Stealth}, thick] (8.5, 1) -- (9.5, 1)  node[midway, above] {\small $-v_1$};

    
    \begin{scope}[shift={(6, 0)}]
    \node[node, label=below left:$v_2$] (r2) at (0, 0) {};
    \node[node, label=above:$v_1$] (r1) at (0.5, 1.6) {};
    \node[node, label=below right:$v_3$] (r3) at (1, 0) {};
    \node[node, label=below:$v_5$] (r5) at (0.5, 0.7) {};
        
        \draw (r1) -- (r2) -- (r3) -- (r1);
        \draw (r1) -- (r5);
        \node at (0.5, -1) {$K_3'$};
    \end{scope}

    \begin{scope}[shift={(11, 0)}]
        \node[node, label=below left:$v_2$] (br2) at (0, 0) {};
        \node[node, label=below right:$v_3$] (br3) at (1, 0) {};
        \node[node, label=above:$v_5$] (br5) at (0.5, 0.7) {};
        
        \draw (br2) -- (br3);
        \node at (0.5,-1) {$K_2^+$};
    \end{scope}
\end{tikzpicture}
    \caption{$\Gamma$ is a proper subdivision of $K_4$. Removing the vertex $v_4$ gives a $K_3$ with an additional leaf. Subsequently removing the vertex $v_1$ yields a $K_2$ with an additional isolated vertex. If we instead remove $v_3$, then $v_1$, we obtain proper subdivisions of $K_3$, then $K_2$.}
    \label{fig:subdivision removal}
\end{figure}
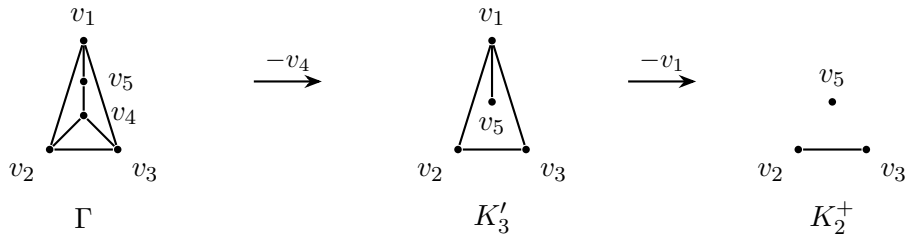
    Let $K_n'$ denote the complete $n$-graph with one additional leaf (a vertex with degree one, joined to an arbitrary point of the $K_n$), and $K_n^+$ the complete $n$-graph with an added isolated point. 
    With this notation, the following statements hold:
    \begin{itemize}
        \item[-] If $\Gamma$ is a proper subdivision of $K_{n+1}$, then for any vertex $v\in V(\Gamma)$, $\Gamma-v$ contains either a proper subdivision of $K_n$, or a $K_n'$.
        \item[-] If $\Gamma = K_n'$, then $K_n'-v$ contains a $K_{n-1}^+$.
    \end{itemize}
    Note that the removal of a vertex can result in a graph with more edges (and vertices) than what is stated above: our claim is the minimal subgraph sufficient for the target statement.
    
    Lastly, if $\Gamma = K_n^+$, then \[\EE N_{\Gamma,s_t}\lesssim \min\{1,ts_t^d\} \cdot I_{\alpha,s_t}(K_n).\]
    To prove this statement, we consider induction as before, but take into account the one additional point of $K_n^+$: such a point exists with probability $\Theta_t(\min\{1, t s_t^d\})$. Though this is not independent from the existence of the $K_n$ itself, it is sufficient (and can be made precise) for an upper bound. 
    Using this specific bound, as well as the first two observations together with the induction argument, we can derive the stronger (quantitative) upper bound for proper subdivisions. 
\end{proof}

\subsection{Crossings in the max-kernel RGG}

In this final section about the max-kernel vertex marked random geometric graph, we consider the number of edge crossings in a projection.
\begin{theorem}[Crossing in the max-kernel marked RGG]\label{thm:RR crossings}
    Let $\hat \eta_t$ be a $\Pareto(r_t,\alpha)$ marked Poisson point process in the convex body $W$, and denote by $X_t(L)$ the number of crossings of $\c G_t(\hat \eta_t)$ when projected onto a plane $L$.
    Then
    \[
    \EE X_t(L) = 
    \begin{cases}
        \Theta_t\l(t^4 r_t^{2d+2}\r), 
        &\caseif \alpha>d+1,\\
        \Theta_t\l(t^4 r_t^{2d+2} \ln \l(\frac{1}{r_t}\r) \r), 
        & \caseif \alpha=d+1,\\
        \Theta_t \l(t^4 r_t^{2\alpha}\r), 
        & \caseif \alpha<d+1.
    \end{cases}
    \]
    Further, if $\alpha<d$, then $\VV X_t(L) = \Theta_t(t^6 r_t^{2\alpha})$.
\end{theorem}
\begin{remark}
    Observe that the expected number of crossings shows the same asymptotic behaviour as the Pareto soft random geometric graph (\Cref{thm:SRGG crossings}). This is due to the fact (by \Cref{thm:A}) that for the expectation, only the connection probability $g(x)$ is significant, and the two models, up to a constant, have the same $g(x)$. 
    However, due to the dependence between the edges, the variance here is of a different order, and the overall behaviour is fundamentally different. 
\end{remark}
\begin{proof}
    For the expectation, we apply Theorem~\ref{thm:A} with $g(x) = \Theta (\min\{1, r_t^\alpha x^{-\alpha}\})$. 
    This yields the same computation and same orders as \Cref{thm:SRGG crossings}.
    For the variance, we can use the upcoming \Cref{sec:crossings variance}. In particular, for $\alpha<d$, Configuration 2.3 therein dominates the model, giving the order stated. 
\end{proof}
Observe that if $\alpha<d$ and $tr_t^\alpha\to 0$, we have $\EE X_t(L)= o_t(\sqrt{\VV X_t(L)})$, that is, the standard deviation has much larger order than the expectation, and so the moments are not indicative of $X_t(L)$. This is simply because two large marks will create many crossings simultaneously (which we see computationally in Configuration 2.3).
If we would like to understand the existence probability of a crossing (or more precisely, obtain the threshold function for the existence), one can correct this bias with the same strategy as in \Cref{thm:RR subgraph}: count large-mark vertices involved in creating a crossing instead of the crossings themselves.


While we believe this approach is viable, the additional geometry, and the fact that each step of an induction-like approach would be fundamentally different, makes the cases difficult to navigate. We do not attempt a meaningful contribution to this here.


\section{The number of crossings in general marked models}\label{sec:crossings}
{In this section, we consider general models of marked random geometric graphs, as discussed in the Preliminaries; recall that each point $y\in \eta_t$ possesses an i.i.d. vertex mark $R_y$, and each distinct, unordered pair $y,z\in \eta_t$ an i.i.d. edge mark $W_{y,z}$. 
Throughout the section, $\c G_t$ denotes the random graph with edges given by pairs $(y,z)$ with $\|y-z\| \leq \varphi(R_y,R_z,W_{y,z})$. 
We also write $y\con z$ for the event that the two points are connected by an edge, and $g(x)$ for the connection probability of two points at distance $x>0$.

Let $L$ be an arbitrary plane in $\R^d$. 
Denote by $X_t(L)$ the number of crossings of $\c G_t$ when projecting onto the plane $L$.
To be precise, defining
\begin{align}
\begin{split}\label{eq:kernel}
    h_L(v_1,v_2,v_3,v_4) 
    \coloneq & \ind \l(v_1 \con v_2,\, v_3\con v_4,\, [v_1,v_2]|_{L}\cap [v_3,v_4]|_{L} \neq \emptyset\r)\\
    + & \ind \l( v_1 \con v_3,\,  v_2 \con v_4,\, [v_1,v_3]|_{L}\cap [v_2,v_4]|_{L} \neq \emptyset\r)\\
    + & \ind \l( v_1\con v_4,\, v_2\con v_3,\, [v_1,v_4]|_{L}\cap [v_2,v_3]|_{L}\neq \emptyset\r),
\end{split}
\end{align}
let
\[X_t(L)\coloneq \frac{1}{24}\sum_{(v_1,v_2,v_3,v_4)\in \eta_{t,\neq}^4} h_L(v_1,v_2,v_3,v_4).\]
For points in general position, we have $h_L\in \{0,1\}$: geometrically, only one of the three possible pairs of edges can form a crossing. In our case, this implies that in the definition of $X_t(L)$, $h_L\in \{0,1\}$ almost surely.
Note also that $h_L$ implicitly depends on the vertex- and edge markings.
}

\subsection{The expected number of crossings}

First, we restate \Cref{thm:A} of the Introduction.
\begin{theorem}[Expected number of crossings]\label{thm:crossing expectation}
    Let $\c G_t$ denote any model of random geometric graph with vertex- and edge markings. Write $g(x)$ for the probability that two points at distance $x$ are connected by an edge.
    Then there exist constants $c_1,c_1'$ and $c_2,c_2'$ depending only on $W$ such that \[c_1' t^4 \l[ \int_0^{c_1} g(x) x^{d} \dd x \r]^2 \leq \EE X_t\leq c_2' t^4 \l[ \int_0^{c_2} g(x) x^{d} \dd x\r]^2.\]
    Further, if there exists $(R_t)_{t>0}$ with $R_t\to 0$ as $t\to \infty$ such that 
    \[\int_{R_t}^{c_2} g(x) x^{d} \dd x = o_t\l(\int_0^{R_t} g(x) x^{d} \dd x\r),\]
    then 
    \[
    \EE X_t(L) = t^4 \l[\int_0^{R_t} g(x) x^{d}\dd x\r]^2 \frac 18 
    \begin{multlined}[t]
    c_d (d+1)^2 I_W(L)\\ 
        \times \l(1 +\c O_t(R_t)\r) \l(1+\c O_t\l(\frac{\int_{R_t}^{c_2}g(x) x^{d}\dd x}{\int_0^{R_t}g(x) x^{d}\dd x}\r)\r)
    \end{multlined}
    \]
    with
    \begin{equation}\label{eq:crossing constant}
        c_d\coloneq 2\pi\kappa_{d-2}^2 \mathbf B\l(\frac 32, \frac d2\r)^2 
        \ \text{ and }\  
        I_W(L)\coloneq \int_{L}\lambda_{d-2}(W\cap (v+L^\perp))^2\dd v,
    \end{equation}
    where $\kappa_d$ is the volume of the $d$-dimensional unit ball, and $\mathbf B(\cdot, \cdot)$ stands for the Beta function $\mathbf B(a,b)=\int_0^1 t^{a-1} (1-t)^{b-1} \dd t$.
\end{theorem}

\begin{remark}$\ $
\begin{enumerate}[label = \roman*)]
\item Our constant $c_d$ differs from the one stated in the corresponding result for the Gilbert graph; see \Cref{remark:RGG constant}.
\item If the tail of $g$ is very heavy, it is possible that the gap between the two sides of the first inequality is very large;  however, in many interesting cases, we obtain matching bounds (that is, they are of the same order of magnitude in $t$), as seen e.g. in \Cref{thm:SRGG crossings}.
\item The condition for the second part of the theorem corresponds to the expected number of crossings being dominated by ones coming from ‘very short' edges i.e. with length at most $R_t$. In this case, the approximation can be made sufficiently precise, with the error term depending on the magnitude $R_t$. 
If the condition does not hold, ‘long edges' carry significant weight in the integral, and finding exact asymptotics would require significant geometric tools. We have not attempted to resolve this here.
\item Similarly to \cite[Lemma~1]{DdJ25}, it is possible to find the expected number of crossings in any Borel set $A\subset L$, with only the constant $I_W(L)$ changing in the exact asymptotic formula.
\end{enumerate}
\end{remark}

\begin{proof}
Write
\[\c C_L\coloneq \{(v_1,v_2,v_3,v_4)\in W^4:[v_1,v_2]\vert_L\cap[v_3,v_4]\vert_L \neq \emptyset\}\]
for tuples that geometrically form a crossing on $L$, in the stated order.
By the Mecke formula, and using the symmetry of $h_L$,
\begin{align*}
    \EE  X_t(L)
    & = \frac{1}{24} t^4 \int_{W^4} \PP(h_L(v_1,v_2,v_3,v_4)=1) \dd v_1 \dd v_2 \dd v_3 \dd v_3 \\
    & = \frac 18 t^4 \int_{W^4} \PP(v_1\con v_2, v_3\con v_4)\ind ((v_1,v_2,v_3,v_4)\in \c C_L) \dd v_1 \dd v_2 \dd v_3 \dd v_4.
\end{align*}
By construction, the existence of the two edges is independent, hence
\[ \EE X_t(L) = \frac 18 t^4 \int_{W^4} g(\Vert v_1-v_2\Vert)g(\Vert v_3-v_4\Vert)\ind ((v_1,v_2,v_3,v_4)\in \c C_L) \dd v_1 \dd v_2 \dd v_3 \dd v_4.\]
Defining the measure $\mu_L$ on $([0,\infty)^2,\c B([0,\infty)^2))$ as
\begin{equation*}
\mu_L(B)\coloneq \int_{W^4}\ind ((\|v_1-v_2\|,\|v_3-v_4\|)\in B) 
\ind((v_1,v_2,v_3,v_4)\in \c C_L) \dd v_1 \dd v_2 \dd v_3 \dd v_4,
\end{equation*}
the expected number of crossings is 
\begin{equation}\label{eq:Etransform}
    \EE X_t (L)= \frac 18 t^4 \int_0^\infty \int_0^\infty g(s_1)g(s_2) \dd \mu_L(s_1, s_2).
\end{equation}
To evaluate this integral, we show the following lemma.
\begin{lemma}\label{lemma:expectation}
    The measure $\mu_L$ is absolutely continuous with respect to the Lebesgue measure, and there exist constants $c_1'',c_2'' > 0$ and $0<c_1<c_2$ depending only on $W$ such that 
    \begin{align*}
      \frac{\dd \mu_L}{\dd s_1 \dd s_2} (s_1, s_2) &\leq c_2'' \cdot  s_1^d s_2^d \text{ for all } s_1,s_2\leq c_2(W), \\
     \frac{\dd \mu_L}{\dd s_1 \dd s_2} (s_1, s_2) & \geq c_1''\cdot s_1^d s_2^d \text{ for all } s_1,s_2\leq c_1(W),
    \end{align*}
    and the density vanishes whenever $s_1>c_2(W)$ or $s_2>c_2(W)$.
    Further, given $(R_t)_{t>0}$ with $R_t\to 0$ as $t\to \infty$, 
    \[\frac{\dd \mu_L}{\dd s_1 \dd s_2} (s_1, s_2) = s_1^d s_2^d \cdot c_d(d+1)^2 I_W(L) \big(1+ \c O_t(R_t)\big)\]
    for all $s_1,s_2\leq R_t$, where the implied constant in $\c O_t(R_t)$ does not depend on $s_1$ and $s_2$.
\end{lemma}

We first show how the lemma implies the theorem.
In the general case, applying \Cref{lemma:expectation} to \eqref{eq:Etransform} immediately yields the bounds stated in the theorem.
To prove the precise asymptotic statement, split the integral in \eqref{eq:Etransform} into two parts:
\[\EE X_t(L)  =   \frac 18 t^4 \int_0^{R_t} \int_0^{R_t} g(s_1)g(s_2) \dd \mu_{L}(s_1,s_2)  +  \frac 18 t^4 \iint_{[0,c'']^2\setminus [0,R_t]^2} g(s_1)g(s_2) \dd \mu_{L}(s_1,s_2).\]
Applying the exact asymptotics to the first part, and the crude bounds to the second, we have
\begin{align*}
    \EE X_t(L) 
    & \leq   \frac 18 t^4 c_d (d+1)^2 I_W(L)(1+\c O_t(R_t))
    \begin{multlined}[t] 
    \int_0^{R_t} \int_0^{R_t} g(s_1)g(s_2) s_1^d s_2^d \dd s_1 \dd s_2   \\+ \frac 18c'' t^4 \iint_{[0,c_2]^2\setminus [0,R_t]^2} g(s_1)g(s_2) s_1^d s_2^d \dd s_1 \dd s_2 
    \end{multlined}\\
    & \leq \frac 18 t^4 c_d (d+1)^2 I_W(L)(1+\c O_t(R_t)) \l(\int_0^{R_t} g(s) s^d \dd s\r)^2 \l(1+\c O_t\l(\frac{\int_{R_t}^{c_2}g(x) x^{d}\dd x}{\int_0^{R_t}g(x) x^{d}\dd x}\r)\r).
\end{align*}
By the assumptions, the last parentheses indeed converges to one. A lower bound is obtained similarly.

\end{proof}

\begin{proof}[Proof of \Cref{lemma:expectation}]
    Fix $s_1,s_2> 0$, and $0<h_1,h_2<s_1,s_2$ small, and write $s_i'\coloneq s_i+h_i$  for $i=1,2$. Assume without loss of generality that $s_1'\leq s_2'$.
We consider the value of 
\[I\coloneq \mu_L([s_1,s_1']\times [s_2,s_2']),\]
and in particular aim to determine its asymptotic behaviour as $h_1,h_2\to 0$. In what follows, we assume that $s_1,s_2\leq \diam (W)$; otherwise, $I=0$ trivially. 

For simplicity, since $L$ is fixed, we identify it with $\R^2$, and also write $B^2$ for $L\cap B^d$.
Additionally, we write $R(w,a,b) = w+(bB^d\setminus aB^d)$ for the annulus centred at $w$ between the radii $a<b$, and in particular $R(a,b)\coloneq R(\mathbf 0,a,b)$. Then
\begin{equation}\label{eq:Eintegral} 
    I = \int_W\int_{R(v_1,s_1,s_1')\cap W}\int_W \int_{R(v_3,s_2,s_2')\cap W}\ind([v_1,v_2]|_L\cap[v_3,v_4]|_L\neq \emptyset)\dd v_4 \dd v_3 \dd v_2 \dd v_1.
\end{equation}
The following steps of evaluating this expression is similar to the proof of \cite[Theorem~12]{CDRarxiv}. 

To upper bound $I$, we may drop the condition of the containment in $W$ for $v_2$ and $v_4$, and obtain by translation
\[I \leq \int_W\int_{R(s_1, s_1')} \int_{W-v_1} \int_{R(s_2, s_2')}\ind (v_1|_L+[0,v_2|_L])\cap(v_1|_L + v_3|_L+[0,v_4|_L])\neq \emptyset)\dd v_4 \dd v_3 \dd v_2 \dd v_1.\]
Denote the inner triple integral by $J^+(v_1)$.
As the integrand does not depend on the components orthogonal to $L$, we may decompose the points as $v_i=u_i+u_i'$, where $u_i\in L$ and $u_i'\in L^\perp$ ($i=2,3,4$). Integrating with respect to the orthogonal components $u_i'$, we obtain
\begin{equation*}
    J^+(v) 
    \begin{multlined}[t] 
    \leq \int_{s_1'B^2} \int_{W_L-v|_L} 
    \int_{s_2'B^2} \ind ([0,u_2]) \cap(u_3+[0,u_4])\neq \emptyset) \\
    \times\lambda_{d-2}(\{u\in W-v\colon u|_L = u_3  \})
    \prod_{i=2,4}\lambda_{d-2}(\{u\in R(s_i,s_i')\colon u|_L = u_i  \}) \dd u_4 \dd u_3\dd u_2.
    \end{multlined}
\end{equation*}
For the volume of points getting projected onto $u_3$, observe that by the edge lengths, $\|u_3\|\leq 2s_2'\leq 4s_2$, hence
\begin{align*}
    \lambda_{d-2}(\{u\in W-v\colon u|_L = u_3  \})
    & \leq \sup_{z\in 4s_2B^2} \lambda_{d-2}(\{u\in W-v\colon u|_L = z\}) \\
    & = \sup_{z\in v+4s_2B^2} \lambda_{d-2}(W\cap (z+L^\perp))\eqcolon q_v^+.
\end{align*}
For general values of $s_i$, we can upper bound this by a universal constant (independently of $v$). If the radii $s_i$ are small, $q_v^+$ is sufficiently close to $\lambda_{d-2}(W\cap (v+L^\perp))$, which is exactly the expression appearing in $I_W(L)$; we give the precise statement at the end of the proof.

For the projections onto $u_2$ and $u_4$, consider the following.
Let $u_0\in L$ be an arbitrary point of distance $x>0$ from the origin, and let
\[p_{a,b}(x) \coloneq \lambda_{d-2} \big( \{u\in R(a,b): u|_L=u_0\} \big)= \lambda_{d-2} \big(R(a,b)\cap(u_0+L^\perp)\big)\]
be the ($(d-2)$-dimensional) volume of the points of the annulus whose projection onto $L$ is $u_0$. 
Geometrically, the set of these points is a $(d-2)$-dimensional ball of radius $\sqrt{b^2-x^2}$ if $a \leq x \leq b$, and an annulus between the radii $\sqrt{a^2-x^2}$ and $\sqrt{b^2-x^2}$ for $x<a$, yielding that
\[p_{a,b}(x)= \l[(b^2-x^2)^{\frac{d-2}{2}}\ind(x\leq b)-(a^2-x^2)^{\frac{d-2}{2}}\ind (x\leq a)\r] \cdot \kappa_{d-2}.\]
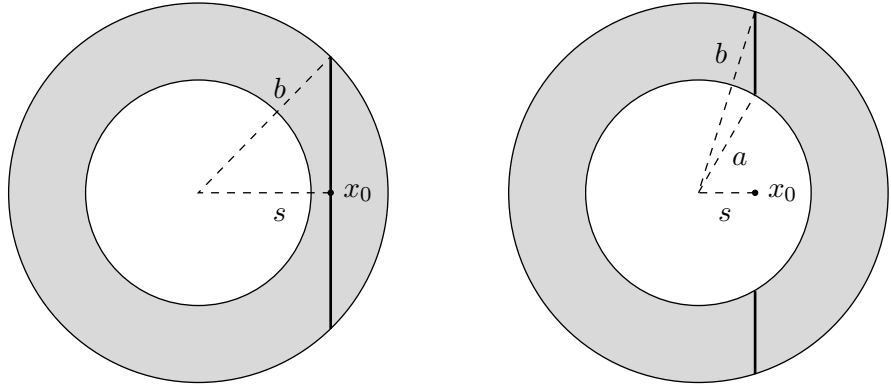
\begin{figure}[h]
    \centering
    \begin{subfigure}{0.4\textwidth}
    \hfill
    \centering
    \begin{tikzpicture}[scale=0.5]
        \coordinate (x) at (0,0);

        \draw [line width = 1pt] (x) circle (5);
        \draw [line width = 1pt] (x) circle (3);
        \fill [black!15, even odd rule] 
              (x) circle (5)
              (x) circle (3);
  
        \coordinate (x1) at (3.5,0);
        \draw [fill=black] (x1) circle (2pt);
        \node at (x1) [right = 1pt]{$x_0$};      
        \draw [line width = 1pt] (3.5,-3.6)--(3.5,3.6);
        \draw [line width = 0.5pt, dashed] (x1)--(x)--(3.5,3.6);
        \node at (2.15,0) [below=2pt] {$s$};
        \node at (2.15,2.77) {$b$};
    \end{tikzpicture}
    \end{subfigure}
    \hspace{20pt}
    \begin{subfigure}{0.4\textwidth}
    \centering
    \begin{tikzpicture}[scale=0.5]
        \coordinate (x) at (0,0);

        \draw [line width = 1pt] (x) circle (5);
        \draw [line width = 1pt] (x) circle (3);
        \fill [black!15, even odd rule] 
              (x) circle (5)
              (x) circle (3);
              
        \coordinate (x2) at (1.5,0);
        \draw [fill=black] (x2) circle (2pt);
        \node at (x2) [right = 1pt]{$x_0$};
        \draw [line width = 1pt] (1.5,-4.8)--(1.5,-2.6); 
        \draw [line width = 1pt] (1.5,4.8)--(1.5,2.6); 
        \draw [line width = 0.5pt, dashed] (x2)--(x)--(1.5,4.8);
        \draw [line width = 0.5pt, dashed] (x)--(1.5,2.6);
        \node at (0.7,0) [below=2pt] {$s$};
        \node at (0.6,3.7) {$b$};
        \node at (1.1,0.9) {$a$};
    \end{tikzpicture}
    \hfill
    \end{subfigure}
    
    \caption{The cross-section of the points of an annulus projected onto $x_0$.}
    \label{fig:annulus}
\end{figure}
As a consequence,
\[ J^+(v) \leq q_v^+\int_{s_1'B^2}\int_{2s_2'B^2} \int_{s_2'B^2} 
\ind ([0,u_2])\cap
\begin{multlined}[t]
(u_3+[0,u_4])\neq \emptyset)  \\p_{s_1,s_1'}(\|u_2\|) p_{s_2,s_2'} (\|u_4\|)\dd u_4 \dd u_3\dd u_2.
\end{multlined}
\]
For fixed $u_2$ and $u_4$, the indicator of the crossing is equivalent to the condition $u_3\in [0,u_2]+[0,-u_4]$. Integrating with respect to $u_3$, 
\begin{align*}
    J^+(v) &\leq q_v^+ \int_{s_1'B^2}\int_{s_2'B^2} \lambda_2([0,u_2]+[0,-u_4]) p_{s_1,s_1'}(\|u_2\|) p_{s_2,s_2'} (\|u_4\|)\dd u_4 \dd u_2 \\
    &\leq q_v^+ \int_{s_1'B^2}\int_{s_2'B^2} \lambda_2\bigg(\l[0, \frac{u_2}{\|u_2\|}\r]+
        \begin{multlined}[t]
        \l[0,-\frac{u_4}{\|u_4\|}\r]\bigg) \\ \times \|u_2\|\|u_4\| p_{s_1,s_1'}(\|u_2\|) p_{s_2,s_2'} (\|u_4\|)\dd u_4 \dd u_2.
        \end{multlined}
\end{align*}
Transforming into polar coordinates with $u_2=x_1z_1$ and $u_4=x_2z_2$ with $z_1,z_2\in S^1$, 
\[J^+(v)\leq q_v^+ \int_0^{s_1'} x_1^2 p_{s_1,s_1'}(x_1)  \dd x_1 \int_0^{s_2'}  x_2^2p_{s_2,s_2'}(x_2) \dd x_2 \cdot \int_{(S^1)^2} \lambda_2([0,z_1]+[0,-z_2]) \dd z_1 \dd z_2,\]
where one of the $x_i^2$ comes from the edge length, and one from the Jacobian of the transformation.
The last integrand, the area of a parallelogram spanned by two vectors, is the absolute sine of the angle between the two vectors. By symmetry, the integral is then $2\pi \int_{0}^{2\pi} |\sin \varphi|\dd \varphi = 8\pi$.
Lastly, the integrals with respect to $x_1$ and $x_2$ can be explicitly evaluated. In particular,
\[\int_0^b  x^2 (b^2-x^2)^{\frac{d-2}{2}}\dd x 
     = \frac 12\int_0^1  b^3 z^{\frac 12} (b^2-b^2z)^{\frac{d-2}{2}}\dd y 
    = \frac 12 b^{d+1}\int_0^1 (1-z)^{\frac d2-1} z^{\frac 12} \dd z,
    \]
which is exactly $\frac 12 b^{d+1} \mathbf B\l(\frac 32, \frac d2\r)$ by definition of the Beta function, hence
\[\int_0^{s_i'} p_{s_i,s_i'}(x_i) x_i^2 \dd x_i =  \frac 12\kappa_{d-2}\mathbf B\l(\frac 32, \frac d2\r) \big((s_i')^{d+1}-s_i^{d+1}\big)\sim h_i \frac 12(d+1) \kappa_{d-2}\mathbf B\l(\frac 32, \frac d2\r)  s_i^d \]
as $h_i\to 0$.
Altogether, we have
\[\limsup_{(h_1,h_2)\to (0,0)} \frac{\mu_L([s_1,s_1'] \times [s_2,s_2'])}{h_1h_2}
\leq 2\pi (d+1)^2 \kappa_{d-2}^2 \mathbf B^2\l(\frac 32, \frac d2\r) s_1^d s_2^d \int_W q_v^+\dd v.\]
Integrating out the component of $v$ orthogonal to $L$, we obtain $\int_L q_u^+ \lambda_{d-2}(W\cap(u+L^\perp))\dd u\leq \int_L(q_u^+)^2 \dd u$.
It is easy to see that the expression without the $s_1^ds_2^d$ factor can be universally upper bounded, which yields the crude upper bound of the lemma.

Lastly, we need a few modifications to obtain a lower bound.
In the first step \eqref{eq:Eintegral}, restrict $v_1$ and $v_3$ to the inner parallel body $W'\coloneq W_{-2s_2}= \{y\in K\colon y+2s_2B^d\subset K\}$. Then $R(v_1, s_1,s_1'), R(v_3,s_2,s_2')\subset W$, and we obtain $I\geq \int_{W'} J^-(v)\dd v$ with $J^-(v)$ as $J^+(v)$, only changing $v_3$ to range over $W'-v$.
The rest of the computations essentially remain the same, and we obtain
\[\liminf_{(h_1,h_2)\to (0,0)} \frac{\mu_L([s_1,s_1'] \times [s_2,s_2'])}{h_1h_2}
\geq 2\pi (d+1)^2 \kappa_{d-2}^2 \mathbf B^2\l(\frac 32, \frac d2\r) s_1^d s_2^d \int_{W'} q_v^-\dd v\]
with $q_v^- \coloneq \inf_{z\in v+4s_2B^2} \lambda_{d-2}(W'\cap(z+L^\perp))$. The remaining integral is lower bounded by $\int_L (q_u^-)^2\dd u$ similarly to the last step of the upper bound.
Now, if $s_2$ is too large, then $W'$ can be empty, or $q_v^-=0$ for all $v\in W'$. 
However, since $W$ has nonempty interior, there is a constant $c_1(W)$ such that if we choose $s_1,s_2\leq c_1(W)$, then the projected volume is positive for some $v$, and the integral does not vanish; this gives the crude lower bound of the lemma.

Lastly, if $s_1,s_2\leq R_t$, where $R_t\to 0$ as $t\to \infty$, then by \cite[Lemma~4]{DdJ25},
\[\int_W q_v^+\dd v - \int_{W_{-2r_t}} q_v^- \dd v 
    = \int_W \Bigg[ 
    \begin{multlined}[t]
    \sup_{x\in v+4s_2B^2} 
    \lambda_{d-2}\big(W \cap (x+L^\perp)\big) \\
     -\inf_{x\in v+4s_2B^2} \lambda_{d-2}\big( W_{-2s_2}\cap (x+L^\perp)\big)\Bigg] \dd v = \c O_t(R_t).
     \end{multlined}
\]
Since $q_v^-\leq\lambda_{d-2}(W\cap(v+L^\perp))\leq q_v^+$, this proves the precise asymptotic stated in the lemma.

\end{proof}

\subsection{The variance of the number of crossings}
\label{sec:crossings variance}
In this section, we derive variance results for the number of crossings. Our goal is to give asymptotics in the most general form possible, which includes many different dependence structures. This makes it difficult to give a concise, but comprehensive statement without an overwhelming amount of different cases and notation. 
This section thus contains all of the tools necessary to find the variance in any of the models described in the introduction, but not in the form of a single result.

For the second moment of the number of crossings, 
\[\EE X_t^2(L) = \frac{1}{24^2}\cdot \EE \sum_{\substack{(v_1,\dots,v_4)\in \eta_{t,\neq}^4 \\ (w_1,\dots,w_4)\in \eta_{t,\neq}^4}} h_{L}(v_1,v_2,v_3,v_4) h_{L}(w_1,w_2,w_3,w_4),\]
with $h_L$ defined in \eqref{eq:kernel}.
If $\{v_1,\ldots, v_4\} \cap\{w_1,\ldots, w_4\}=\emptyset$, we can apply the Mecke formula independently to obtain $(\EE X_t(L))^2$; this yields that the variance is given by the above sum over two $4$-tuples with nonempty intersection. Also note that if $\{v_1,\ldots, v_4\}=\{w_1,\ldots, w_4\}$, then we get exactly the expectation $\EE X_t(L)$.


Going forward, without loss of generality, we may let $v_i=w_i$ for $i=1,\ldots, k$ with $k\in \{1,2,3\}$. Applying the Mecke formula then gives, up to a constant,
\[t^{8-k} \int_{W^{8-k}} \PP(h_L(v_1,v_2,v_3,v_4)h_L(v_1,\ldots, v_k, w_{k+1},\ldots, w_4)=1)\dd v_1 \ldots \dd v_4 \dd w_{k+1}\ldots \dd w_4.\]
Geometrically, $h_L$ only contains information about whether a crossing exists or not, but not about which two pairs of points form that crossing (as noted before, only one of the possible combinations can provide a crossing).
Thus it is necessary to understand which pairs of points form the two crossing edges, and in particular how they relate to each other across the two $4$-tuples. This gives rise to different geometric settings, as well as different connection probabilities.
Writing $\c C\coloneq  \{(u_1,u_2,u_3,u_4)\in (\R^d)^4:  [u_1,u_2]|_L\cap [u_3,u_4]|_L\neq \emptyset\}$, we must evaluate integrals of the form
\begin{equation}\label{eq:variance integral}
I\coloneq t^{8-k} \int_W\Bigg[\int_{W^{7-k}} 
\begin{multlined}[t]
    \PP(v\con p_2, p_3\con p_4, v\con q_2, q_3\con q_4) \\ 
    \times \ind_{\c C}((v,p_2,p_3,p_4)) \ind_{\c C}((v,q_2,q_3,q_4)) \dd v_2 \cdots \dd v_4 \dd w_{k+1} \cdots \dd w_4 \Bigg]\dd v,
\end{multlined}
\end{equation}
where $ \{p_2,p_3,p_4\}=\{v_2,v_3,v_4\}$ and $\{q_2,q_3,q_4\} = \{v_2,\ldots,v_{k}, w_{k+1},\ldots, w_4\}$. 
First, we give a general overview of how these integrals can be evaluated, and give the details of the possible arrangements later on. 

Let the choice of $p_i$-s and $q_j$-s be fixed.
Denote by $m$ the number of edges in the configuration: we either have the maximal $m=4$, or two of the edges coincide and hence $m=3$. Without loss of generality, we may assume that in the case of $m=3$, the coinciding edge is the one with endpoint $v$, i.e. $p_2=q_2$.
Since the connection probabilities are translation invariant, the probability that the $m$ edges all exist depends only on the length of these edges, and not the exact positions of the points. As a consequence,
\[I= t^{8-k} \int_W \Bigg[\int_{(W-v)^{7-k}} G(\mathbf x) \ind_{\c C}((0,p_2,p_3,p_4)) \ind_{\c C}((0,q_2,q_3,q_4)) \dd v_2 \cdots \dd v_4 \dd w_{k+1} \cdots \dd w_4 \Bigg]\dd v\]
with 
\begin{equation}\label{eq:edgelengths}
\mathbf x = 
    \begin{cases}
    \big( \|p_2\|,\ \|p_3-p_4\|,\ \|q_3-q_4\| \big), & \text{ if } m=3,\\
    \big( \|p_2\|,\ \|p_3-p_4\|,\ \|q_2\|,\ \|q_3-q_4\| \big), & \text{ if } m=4,
    \end{cases}
\end{equation}
and $G(\mathbf x)$ indicating the connection probability. Note that while it does not appear in the notation, $G$ generally depends on $t$, and dependence on vertex- and edge markings is also assumed.

We transform this integral similarly to the expectation, though we now disregard dependence on $v$, as we only aim to obtain an order of magnitude.
More precisely, define the measure 
\[\mu_{\gamma}(A) = \int_{(\gamma B^d)^{7-k}} \ind_{\c C}((0,p_2,p_3,p_4)) \ind_{\c C}((0,q_2,q_3,q_4)) \ind_A(\mathbf x )\dd v_2 \cdots \dd v_4 \dd w_{k+1} \cdots \dd w_4\]
on $(\R^m, \c B(\R^m))$ for arbitrary $\gamma>0$.
Since $W$ has non-empty interior, it is contained between two balls, which can be used as approximating upper- and lower bounds; see \eqref{eq:ball sandwich} and the corresponding reasoning.
In particular, there exist constants $c_1,c_2,\gamma_1,\gamma_2>0$ such that
\[c_1 t^{8-k} \int_{[0,\gamma_1]^m} G(\mathbf s) \mu_{\gamma_1}(\dd \mathbf s )  \leq I \leq c_2t^{8-k} \int_{[0,\gamma_2]^m} G(\mathbf s) \mu_{\gamma_2}(\dd \mathbf s ).\]
We will show that $\mu_{\gamma}$ is continuous with respect to the Lebesgue measure. Given upper- and lower bounds for its density
\[c_1' f^-(\mathbf s) \leq \frac{\dd \mu_{\gamma_1}}{\dd \lambda_m}(\mathbf s)\  \text{ and }\ \frac{\dd \mu_{\gamma_2}}{\dd \lambda_m}(\mathbf s)\leq c_2' f^+(\mathbf s),\]
we obtain the bounds
\[\EE X_t(L)+c_1''\sum t^{8-k} \int_{[0,\gamma_1]^m} 
\begin{multlined}[t]
G(\mathbf s)  f^-(\mathbf s) \dd s
\\ \leq \VV X_t(L) \leq \EE X_t(L)+  c_2''\sum t^{8-k} \int_{[0,\gamma_2]^m} G(\mathbf s) f^+(\mathbf s) \dd s
\end{multlined}\]
for the variance, where the sum runs over all possible geometric configurations, and $c_1'',c_2''>0$ are constants (depending on $W$).
These can be considered matching bounds in the same sense as in the case of the expectation: if the tails of the connection probabilities are not too heavy, changing the domain of integration (corresponding to the maximal edge length allowed) by a constant will change the integral by at most a multiplicative constant factor, thus giving the order of each term. 

Lastly, we give a complete description of the possible configurations. 
\begin{figure}[ht]
    \centering
    \begin{subfigure}{0.3\textwidth} 
        \centering
        \hfill
        \begin{tikzpicture}[scale=0.5]
            \coordinate (v) at (0,0);
            \draw [fill=black] (v) circle (2pt);
            \node at (v) [below left = 2pt]{$0$};
            
            \coordinate (v2) at (3,2);
            \draw [fill=black] (v2) circle (2pt);
            \node at (v2) [right = 2pt] {$v_2$};
            
            \coordinate (v3) at (1,-1);
            \draw [fill=black] (v3) circle (2pt);
            \node at (v3) [below = 2pt] {$v_3$};
            
            \coordinate (v4) at (2.5,3);
            \draw [fill=black] (v4) circle (2pt);
            \node at (v4) [above right] {$v_4$};
            
            \coordinate (w2) at (-0.5,4.5);
            \draw [fill=black] (w2) circle (2pt);
            \node at (w2) [above right] {$w_2$};
            
            \coordinate (w3) at (0.5,2);
            \draw [fill=black] (w3) circle (2pt);
            \node at (w3) [below right] {$w_3$};
            
            \coordinate (w4) at (-2,4);
            \draw [fill=black] (w4) circle (2pt);
            \node at (w4) [above left] {$w_4$};
            
            \draw [line width = 1pt] (v)--(v2);
            \draw [line width = 1pt, dashed] (v3)--(v4);
            \draw [line width = 1pt] (v)--(w2);
            \draw [line width = 1pt, dashed] (w3)--(w4);
        \end{tikzpicture}

        \caption{Configuration 1}
    \end{subfigure}
    \hfill
    \begin{subfigure}{0.3\textwidth} 
        \centering
        \begin{tikzpicture}[scale=0.5]
            \coordinate (v) at (0,0);
            \draw [fill=black] (v) circle (2pt);
            \node at (v) [below left = 2pt]{$0$};
            
            \coordinate (v2) at (4,3);
            \draw [fill=black] (v2) circle (2pt);
            \node at (v2) [right = 2pt] {$v_2$};
            
            \coordinate (v3) at (3,-1);
            \draw [fill=black] (v3) circle (2pt);
            \node at (v3) [below = 2pt] {$v_3$};
            
            \coordinate (v4) at (1.5,3);
            \draw [fill=black] (v4) circle (2pt);
            \node at (v4) [above left] {$v_4$};
            
            \coordinate (w3) at (4.5,2);
            \draw [fill=black] (w3) circle (2pt);
            \node at (w3) [below right] {$w_3$};
            
            \coordinate (w4) at (2,4);
            \draw [fill=black] (w4) circle (2pt);
            \node at (w4) [above left] {$w_4$};
            
            \draw [line width = 1pt] (v)--(v2);
            \draw [line width = 1pt, dashed] (v3)--(v4);
            \draw [line width = 1pt, dashed] (w3)--(w4);
        \end{tikzpicture}

        \caption{Configuration 2.1}
    \end{subfigure}
    \hfill
    \begin{subfigure}{0.3\textwidth} 
        \centering
        \begin{tikzpicture}[scale=0.5]
            \coordinate (v) at (0,0);
            \draw [fill=black] (v) circle (2pt);
            \node at (v) [below left = 2pt]{$0$};
            
            \coordinate (v2) at (4,2);
            \draw [fill=black] (v2) circle (2pt);
            \node at (v2) [right = 2pt] {$v_2$};
            
            \coordinate (v3) at (3,-1);
            \draw [fill=black] (v3) circle (2pt);
            \node at (v3) [below = 2pt] {$v_3$};
            
            \coordinate (v4) at (1,1.5);
            \draw [fill=black] (v4) circle (2pt);
            \node at (v4) [above right] {$v_4$};
            
            \coordinate (w3) at (1.5,6);
            \draw [fill=black] (w3) circle (2pt);
            \node at (w3) [below right] {$w_3$};
            
            \coordinate (w4) at (-1,6);
            \draw [fill=black] (w4) circle (2pt);
            \node at (w4) [above left] {$w_4$};
            
            \draw [line width = 1pt] (v)--(v2);
            \draw [line width = 1pt] (v)--(w3);
            \draw [line width = 1pt, dashed] (v3)--(v4);
            \draw [line width = 1pt, dashed] (v2)--(w4);
        \end{tikzpicture}

        \caption{Configuration 2.2}
    \end{subfigure}
    \hfill

    \vspace{5pt}

    \hfill

    \begin{subfigure}{0.3\textwidth} 
        \centering
        \begin{tikzpicture}[scale=0.5]
            \coordinate (v) at (0,0);
            \draw [fill=black] (v) circle (2pt);
            \node at (v) [below left = 2pt]{$0$};
            
            \coordinate (v2) at (3,1);
            \draw [fill=black] (v2) circle (2pt);
            \node at (v2) [right = 2pt] {$v_2$};
            
            \coordinate (v3) at (4,-1);
            \draw [fill=black] (v3) circle (2pt);
            \node at (v3) [right=1pt] {$v_3$};
            
            \coordinate (v4) at (1,-2);
            \draw [fill=black] (v4) circle (2pt);
            \node at (v4) [below] {$v_4$};
            
            \coordinate (w3) at (-1,4);
            \draw [fill=black] (w3) circle (2pt);
            \node at (w3) [above=1pt] {$w_3$};
            
            \coordinate (w4) at (-2,3);
            \draw [fill=black] (w4) circle (2pt);
            \node at (w4) [above left] {$w_4$};
            
            \draw [line width = 1pt] (v)--(v3);
            \draw [line width = 1pt] (v)--(w3);
            \draw [line width = 1pt, dashed] (v2)--(v4);
            \draw [line width = 1pt, dashed] (v2)--(w4);
        \end{tikzpicture}

        \caption{Configuration 2.3}
    \end{subfigure}
    \hfill
    \begin{subfigure}{0.3\textwidth} 
        \centering
        \begin{tikzpicture}[scale=0.5]
            \coordinate (v) at (0,0);
            \draw [fill=black] (v) circle (2pt);
            \node at (v) [below left = 2pt]{$0$};
            
            \coordinate (v2) at (4,4.5);
            \draw [fill=black] (v2) circle (2pt);
            \node at (v2) [right = 2pt] {$v_2$};
            
            \coordinate (v3) at (4,0.5);
            \draw [fill=black] (v3) circle (2pt);
            \node at (v3) [below right=1pt] {$v_3$};
            
            \coordinate (v4) at (-1,2.5);
            \draw [fill=black] (v4) circle (2pt);
            \node at (v4) [above left] {$v_4$};
            
            \coordinate (w4) at (2,3.5);
            \draw [fill=black] (w4) circle (2pt);
            \node at (w4) [above left] {$w_4$};
            
            \draw [line width = 1pt] (v)--(v2);
            \draw [line width = 1pt, dashed] (v3)--(v4);
            \draw [line width = 1pt, dashed] (v3)--(w4);
        \end{tikzpicture}

        \caption{Configuration 3.1}
    \end{subfigure}
    \hfill
    \begin{subfigure}{0.3\textwidth} 
        \centering
        \begin{tikzpicture}[scale=0.5]
            \coordinate (v) at (0,0);
            \draw [fill=black] (v) circle (2pt);
            \node at (v) [below left = 2pt]{$0$};
            
            \coordinate (v2) at (-2,4);
            \draw [fill=black] (v2) circle (2pt);
            \node at (v2) [above left = 1pt] {$v_2$};
            
            \coordinate (v3) at (4,1);
            \draw [fill=black] (v3) circle (2pt);
            \node at (v3) [below right] {$v_3$};
            
            \coordinate (v4) at (-2,3);
            \draw [fill=black] (v4) circle (2pt);
            \node at (v4) [below left] {$v_4$};
            
            \coordinate (w4) at (3,-1);
            \draw [fill=black] (w4) circle (2pt);
            \node at (w4) [below right] {$w_4$};
            
            \draw [line width = 1pt] (v)--(v2);
            \draw [line width = 1pt] (v)--(v3);
            \draw [line width = 1pt, dashed] (v3)--(v4);
            \draw [line width = 1pt, dashed] (v2)--(w4);
        \end{tikzpicture}
        \caption{Configuration 3.2}
    \end{subfigure}
    \hfill

    \caption{Geometric configurations of two crossings.}
    \label{fig:variance}
\end{figure}
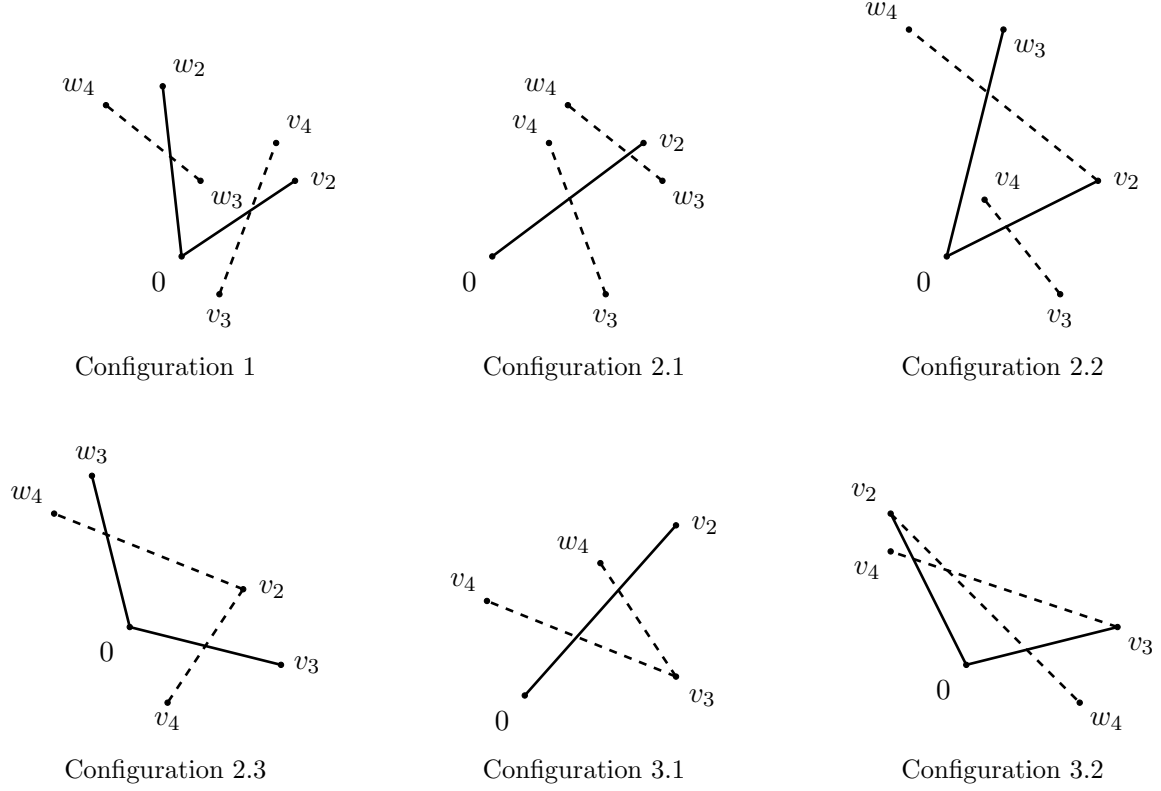
All cases are shown, in the order discussed below, in \Cref{fig:variance}.
Let $\mathbf p = (p_1,\ldots, p_m)$, $\mathbf q = (q_1,\ldots, q_m)$ and $\mathbf x = (x_1,\ldots, x_m)$.
If $k=1$, each selection for $\mathbf p$ and $\mathbf q$ is geometrically equivalent. 
If $k=2$, there are three distinct cases: the coinciding pair of points $(0,v_2)$ is either involved in the crossing in both $4$-tuples, only one, or neither. 
If $k=3$, the coinciding triple $(0,v_2,v_3)$ contains a relevant edge  in both $4$-tuples: they either coincide, or they do not.
The table \eqref{tab:variance} gives a summary of the cases; the first column contains the labels used as a reference. 

\begin{equation}\label{tab:variance}
    \begin{tabular}{c||c|cc|c|c}
        & $k$ & $\mathbf p$ & $\mathbf q$ & $m$ & $f^+(\mathbf s)$\\
        \hline
        \hline
        \phantom{1.}1$^*$ & $1$ & $(v_2,v_3,v_4)$ & $(w_2,w_3,w_4)$ & $4$ & $s_1^d s_2^d s_3^d s_4^d$\\
        \hline
        2.1$^*$ & $2$ & $(v_2,v_3,v_4)$ & $(v_2,w_3,w_4)$ & $3$ & $s_1^{d+1} s_2^d s_3^d$\\
        2.2\phantom{$^*$} & $2$ & $(v_2,v_3,v_4)$ & $(w_3,v_2,w_4)$ & $4$ & $s_1^d s_2^d s_3^{d-1} s_4^{d-1}$\\
        2.3\phantom{$^*$} & $2$ & $(v_3,v_2,v_4)$ & $(w_3,v_2,w_4)$ & $4$ & $s_1^{d-1}s_2^{d-1} s_3^{d-1} s_4^{d-1} \min\{s_1s_2, s_3s_4\}$ \\
        \hline
        3.1\phantom{$^*$} & $3$ & $(v_2,v_3,v_4)$ & $(v_2,v_3,w_4)$ & $3$ & $s_1^ds_2^{d-1} s_3^{d-1} \min\{s_2,s_3\}$ \\
        3.2\phantom{$^*$} & $3$ & $(v_2,v_3,v_4)$ & $(v_3,v_2,w_4)$ & $4$ & $s_1^{d-1} s_2^{d-1} s_3^{d-1} s_4^{d-1}$\\
    \end{tabular}
\end{equation}

The last column of the table contains easily accessible upper bounds for the density. For the configurations with an $^*$ notation, the lower bound $f^-(\mathbf s)$ is the same (that is, this is the exact behaviour of the density), but for the other cases, this is generally not true. This is primarily because the density is highly sensitive to the relative ordering of the values $s_1,\ldots,s_m$.
The computations use similar ideas to the proof of \Cref{lemma:expectation}, with more geometric considerations being necessary. 
We give a sketch below, including the approach that can be used to derive the exact order.


With the notation above, and for every configuration, we consider 
\[M \coloneq \int_{(B^d)^{7-k}} \ind_{\c C}(0,p_2,\begin{multlined}[t] p_3,p_4) \ind_{\c C}(0,q_2,q_3,q_4) \\ \ind(\mathbf x\in [s_1,s_1+h_1]\times \cdots \times [s_m,s_m+h_m])\dd v_2 \cdots \dd v_4 \dd w_{k+1} \cdots \dd w_4
\end{multlined}
\]
for $s_i\leq 1$ and $h_i$ small; we shall indicate the specific configuration by a superscript. 
Considering this expression as $h_1,\ldots, h_m\to 0$ gives the density of $\mu_1$.
Note that considering $\mu_\gamma$ generally, i.e. letting the points range over balls of radius $\gamma$ will only change the expression by a constant factor, and so it is enough for us to consider the case of the unit ball. 

The general recipe for understanding this integral is similar to that of the expectation, in particular the proof of \Cref{lemma:expectation}. 
We give a general outline, then consider the unique geometry of each configuration. We focus on deriving the upper bounds stated in \eqref{tab:variance}, and only give short remarks on how to obtain the correct order.

Step one: in every edge, one of the endpoints is reparametrised (by translation) in relation to the other endpoint. The range of the reparametrised point then becomes, due to the indicator, the annulus $R(s_i,s_i')\coloneq s_i'B^d \setminus s_i B^d$, with $s_i'\coloneq s_i+h$. 
Step two: since the remaining integrand (the indicator of the crossings) only depends on the projections to $L$, we may integrate with respect to the orthogonal component of each point. This gives the $(d-2)$-dimensional volume of points that get projected onto a certain point of $L\cap B^d \eqcolon B^2$.
For the points in an annulus, the volume of points $v_i\in R(s_i,s_i')$ projected onto some $u_i\in s_i' B^2$ with $\|u_i\|\eqcolon x$ is
\[P_i(x) \coloneq \l[((s_i')^2-x^2)^{\frac{d-2}{2}}-(s_i^2-x^2)^{\frac{d-2}{2}}\ind (x\leq s_i)\r] \cdot \kappa_{d-2}.\]
For the points running over $B^d$ (i.e. ones which were not reparametrised), we upper bound this volume by some universal constant. 
Observing that the crossing condition $[0,y_1]\cap(y_2+[0,y_3])\neq \emptyset$ is equivalent to the containment condition $y_2\in [0,y_1]+[0,-y_3]$, this yields
\[M \lesssim \int_{s_1'B^2} \cdots \int_{s_m'B^2} \int_{(B^2)^{7-k-m}} 
\ind(p_3^L\in 
\begin{multlined}[t]
    [0,p_2^L]+[0,-p_4^L])
    \ind(q_3^L\in [0,q_2^L]+[0,-q_4^L]) \\
    P_1(x_1) \cdot \ldots \cdot P_m(x_m)
    \dd u_2 \cdots \dd u_4 \dd z_{k+1} \cdots \dd z_4,
\end{multlined}
\]
where $u_i\coloneq v_i|_L$, $z_i\coloneq w_i|_L$, and $p_i^L,q_j^L$ stand for the projected version of their counterparts $p_i, q_j$, and each $x_i$ corresponds to an edge length (see \eqref{eq:edgelengths}).
The geometric constraints given by the crossing conditions need to be handled individually. 

For each configuration, we aim to integrate out everything except the edge lengths $x_i$. 
The easiest type of integral to arise is the one similar to the expectation:
\begin{equation}\label{eq:projection integral}
\int_0^{s_i}P_i(x_i)x_i^n \dd x_i\approx h_i s_i^{d+n-2}
\end{equation}
with any $n\geq 0$ as $h_i\to 0$, with only absolute constants being implied by the $\approx$ notation.

We now turn to handling the configurations one by one; we refer the reader back to \eqref{tab:variance} for the details of each setting. 

\paragraph{Configuration 1.} 
The two crossings are essentially independent, and we get an upper bound that closely resembles $J^+(0)^2$ in the proof of \Cref{lemma:expectation}.
In particular, the two geometric constraints apply to two different points, namely $p_3^L=u_3$ and $q_3^L=z_3$. Integrating out these two points, the area of the corresponding parallelograms appears:
\begin{align*}
    M^{(1)} 
    & \lesssim \int_{s_1'B^2} \int_{s_2'B^2} \int_{s_3'B^2} \int_{s_4'B^2} 
    \begin{multlined}[t]
    P_1(x_1)P_2(x_2) P_3(x_3) P_4(x_4) \\
    \lambda_2([0,u_2]+[0,-u_4])\lambda_2([0,z_2]+[0,-z_4]) \dd z_4 \dd z_2 \dd u_4 \dd u_2 
    \end{multlined}\\
    & \lesssim \int_{s_1'B^2} \int_{s_2'B^2} \int_{s_3'B^2} \int_{s_4'B^2} P_1(x_1) P_2(x_2)  P_3(x_3) P_4(x_4) x_1 x_2 x_3 x_4 \dd z_4 \dd z_2 \dd u_4 \dd u_2
\end{align*}
with $(x_1,x_2,x_3,x_4) = (\|u_2\|,\|u_4\|,\|z_2\|,\|z_4\|)$, where the $x_1x_2x_3 x_4$ term in the second line comes from the linearity of the area.
Applying a polar transform, each exponent increases by one, yielding integrals of the form \eqref{eq:projection integral} with exponent $2$ for each $x_i$. This immediately gives $s_1^ds_2^d \cdot s_3^d s_4^d$, and matching lower bound can be derived the same way.

\paragraph{Configuration 2.1.} 
As in Configuration 1, the constraints apply to $p_3^L=u_3$ and $q_3^L=z_3$, and we obtain
\begin{align*}
M^{(2.1)} 
& \lesssim \int_{s_1'B^2} \int_{s_2'B^2} \int_{s_3'B^2} 
P_1(x_1)
\begin{multlined}[t]
P_2(x_2) P_3(x_3) \\
\lambda_2([0,u_2]+[0,-u_4])\lambda_2([0,u_2]+[0,-z_4]) \dd z_4 \dd u_4 \dd u_2 
\end{multlined}\\
& \lesssim \int_{s_1'B^2} \int_{s_2'B^2} \int_{s_3'B^2} P_1(x_1) P_2(x_2)  P_3(x_3)  x_1^2 x_2 x_3 \dd z_4 \dd u_4 \dd u_2
\end{align*}
with $(x_1,x_2,x_3)=(\|u_2\|,\|u_4\|, \|z_4\|)$.
In contrast to Configuration 1, here the second power of $x_1$ appears, due to $u_2$ appearing in both parallelograms. Consequently, after polar transform, the exponents are $3,2,2$, respectively, giving the limiting $s_1^{d+1} s_2^d s_3^d$. A matching lower bound can be obtained as well.

\paragraph{Configuration 2.3.} Both geometric constraints apply to $u_2 = p_3^L = q_3^L$, hence the area of the intersection of parallelograms appear:
\[M^{(2.3)} \lesssim  \int_{s_1'B^2}\cdots \int_{s_4'B^2} 
    \lambda_2 (([0,u_3]+[0,-u_4])
    \begin{multlined}[t]
    \cap ([0,z_3]+[0,-z_4])) \\
    P_1(x_1) P_2(x_2) P_3(x_3) P_4(x_4)\dd z_4 \dd z_3 \dd u_4 \dd u_3
    \end{multlined}\]
where now  $(x_1,x_2,x_3,x_4) = (\|u_3\|, \|u_4\|, \|z_3\|, \|z_4\|)$.
As an easy upper bound, we can upper bound the area of the intersection by the minimum of the areas, yielding $\min\{x_1x_2,x_3x_4\}$ by linearity, which yields the bound stated in the table. In general, this is not the correct order. This is because after applying a polar transform, the area of the intersection heavily depends on the angles, not only the edge lengths $x_i$. In particular, the proportion of angles that meaningfully contribute to the integral might be small, and thus influence the order. We can show, with the help of tools from integral geometry, that the correct order is actually
\[M^{(2.3)}\approx h_1h_2h_3h_4s_1^{d-1} s_2^{d-1} s_3^{d-1} s_4^{d-1} \min_{i\neq j} \{s_is_j\}.\]

\paragraph{Configuration 3.1.} Similarly to Configuration~2.3, both geometric constraints apply to $u_3 = p_3^L = q_3^L$, hence the area of the intersection of parallelograms appear:
\[M^{(3.1)} \lesssim  \int_{s_1'B^2}\cdots \int_{s_4'B^2} 
    \lambda_2([0,u_2]+[0,-u_4])
    \begin{multlined}[t]
    \cap ([0,u_2]+[0,-z_4])) \\
    P_1(x_1) P_2(x_2) P_3(x_3) \dd z_4 \dd u_4 \dd u_2
    \end{multlined}\]
with $(x_1,x_2,x_3)=(\|u_2\|,\|u_4\|,\|z_4\|)$.
Upper bounding by the minimum of the two areas $\min(x_1x_2,x_1x_3)=x_1\min(x_2,x_3)$ gives the order in the table. 
For a lower bound, the same problem as for Configuration 2.3 arises: for many angles, the intersection can have small area compared to the edge lengths. 
The correct order is instead
\[M^{(3.1)} \approx h_1h_2h_3 s_1^d s_2^{d-1} s_3^{d-1} \min \{s_1,s_2,s_3\}.\]

\paragraph{Configuration 2.2.} 
Here, we can a priori only integrate out one point, and the other crossing condition cannot be easily resolved:
\[M^{(2.2)} \lesssim  \int_{s_1'B^2} \int_{s_3'B^2} 
\begin{multlined}[t]
\bigg[ \int_{s_2'B^2} P_2(x_2) \lambda_2([0,u_2]+[0,-u_4]) \dd u_4 \bigg] \\
\bigg[ \int_{s_4'B^2} P_4(x_4)\ind ([0,z_3]\cap(u_2+[0,z_4])\neq \emptyset) \dd z_4 \bigg]
P_1(x_1) P_3(x_3) \dd z_3 \dd u_2
\end{multlined}
\]
with $(x_1,x_2,x_3,x_4) = (\|u_2\|, \|u_4\|, \|z_3\|, \|z_4\|)$.
For an upper bound, one can simply drop the second crossing condition, and have the $x_1x_2$ term coming from the area factor coming from the first crossing condition.
To obtain the correct order, it is necessary to take into account the second condition: if the distance of $u_2$ and the line segment $[0,z_3]$ is larger than $s_4'$, then the integral is zero. This consideration again yields an upper bound, which is not so easily evaluated, but we believe should give the correct order. 

\paragraph{Configuration 3.2.} 
Lastly, the conditions are interlinked in a way that no point can be easily omitted:
\[M^{(3.2)} \lesssim  \int_{s_1'B^2} \int_{s_3'B^2} 
\begin{multlined}[t]
\bigg[ \int_{s_2'B^2} 
P_2(x_2)\ind ([0,u_2]\cap(u_3+[0,u_4]\neq \emptyset) \dd u_4 \bigg] \\
\bigg[ \int_{s_4'B^2} P_4(x_4)\ind ([0,u_3]\cap(u_2+[0,z_4])\neq \emptyset) \dd z_4 \bigg]
P_1(x_1) P_3(x_3) \dd u_3 \dd u_2
\end{multlined}
\]
with $(x_1,x_2,x_3,x_4) = (\|u_2\|, \|u_4\|, \|u_3\|, \|z_4\|)$
An easy upper bound is given by dropping both crossing conditions; a more involved upper bound by replacing it by the indicator that $u_3$ and $[0,u_2]$, as well as $u_2$ and $[0,u_3]$ are sufficiently close.

\section{Discussion}\label{sec:discussion}

First, we see that widely applicable expectation (\Cref{thm:crossing expectation}) and variance (\Cref{sec:crossings variance}) results for the number of crossings $X_t(L)$ of projected random geometric graphs can be derived. The variance setup (e.g. \Cref{fig:variance}) in particular shows how dissimilar the different models can be.
However, \Cref{thm:RR crossings} also illustrates well that the first two moments cannot reliably be used as indicators for the distributional behaviour of $X_t(L)$. In such models, significantly different tools are necessary to understand $X_t(L)$, which are seemingly more difficult to handle universally.

Our other main question, planarity, exhibits vastly different behaviours as well.
Contrasting \Cref{thm:RGGplanarity,thm:SRGGplanarity,thm:RRplanarity}, one observes substantially distinct characteristics: from the domination of small non-planar subgraphs and thus $K_5$, to the domination of large non-planar subgraphs (with vertex count that grows with $t$), to that of $K_{3,3}$. 

Lastly, with regards to the Pareto soft random geometric graph, observe the following. 
Recall from \Cref{thm:SRGGcliques} that the expected number of complete graphs $K_n$ is
$\Theta_t(t^n r_t^{d(n-1)})$ whenever $\alpha > \frac 2n d$, which is the same order as the hard random geometric graph HRGG (i.e. when exactly edges with length at most $r_t$ are included in the graph).
On the other hand, if $\alpha < \frac 2n d$, the expectation is of the order $t^n (r_t^{\alpha})^{\binom n2}$, which is the same order as for the Poissonised Erd\H{o}s--R\'enyi random graph with connection probability $\Theta_t(r_t^\alpha)$ (that is, each pair of $\Po(t)$ many vertices is connected by an edge independently with probability $r_t^\alpha$). 
For certain ranges of $\alpha$, the same phenomenon (namely, that the behaviour of the Pareto SRGG follows that of the Erd\H{o}s--R\'enyi random graph) occurs for other questions as well: for the variance of the complete graph, path counts, and even the planarity results. 
This can partially be attributed to the following.
The connection probability $g(x) = \min\{1, r_t^\alpha x^{-\alpha}\}$ is lower bounded by $g_-(x) \coloneq  \ind (x\leq r_t) + c_0r_t^\alpha \ind(x>r_t)$ for appropriately chosen $c_0>0$ for all points $y,z\in W$ with $x\coloneq \|y-z\|$. 
One can thus construct a (Poissonised) Erd\H{o}s--R\'enyi graph  (ER) on $\eta_t$ with connection probability $c_0r_t^\alpha$ such that the edge sets satisfy $E_{\text{SRGG}}\supset E_{\text{HRGG}}\cup E_{\text{ER}}$, which at least explains one direction of the phenomenon described above.
Unfortunately, a similar upper bound cannot be easily obtained, and if the two graphs (HRGG and ER) behave similarly (i.e. neither dominates the other in some respect), the correct SRGG behaviour is likely not given by either of them separately (as in the $\alpha = \frac 2n d$ case with the additional logarithmic factor). 
However, we believe this observation sheds light on the behaviour of heavy-tailed SRGGs, and could be used to fill some gaps in the literature: any existing results that consider polynomial tails only do so for $\alpha>d$, which is essentially the extremal connection function that still allows for the use of geometric localisation tools.

\paragraph{Funding information}
This research was funded by the Deutsche Forschungsgemeinschaft (DFG, German Research Foundation) – Project-ID 531542011 (H.~Döring) and 531562368 (K.~Nagy). 

\bibliography{references}




\end{document}